\documentclass[11pt]{article}
\usepackage[usenames,dvipsnames,svgnames,table]{xcolor} 

\usepackage{enumitem} 
\usepackage{rotfloat} 
\usepackage{graphicx}
\usepackage{epsfig}
\usepackage{amssymb, amsmath}
\usepackage{verbatim}
\usepackage{natbib}
\usepackage{authblk}
\usepackage{kotex}
\usepackage{multirow}
\usepackage[colorlinks]{hyperref} 
\usepackage[colorinlistoftodos, textsize=scriptsize]{todonotes} 
\usepackage{subfigure} 

\newtheorem{theorem}{Theorem}[section]
\newtheorem{lemma}[theorem]{Lemma}

\newtheorem{corollary}[theorem]{Corollary}

\newenvironment{proof}[1][Proof]{\begin{trivlist}
		\item[\hskip \labelsep {\bfseries #1}]}{\end{trivlist}}

\newenvironment{remark}[1][Remark]{\begin{trivlist}
		\item[\hskip \labelsep {\bfseries #1}]}{\end{trivlist}}

\newcommand{\bea}{\begin{eqnarray*}}
	\newcommand{\eea}{\end{eqnarray*}}
\newcommand{\bean}{\begin{eqnarray}}
\newcommand{\eean}{\end{eqnarray}}

\newcommand{\V}{{\rm Var}}

\newcommand{\sg}{\Sigma}
\newcommand{\what}{\widehat}

\newcommand{\Eaa}{\bbE_{\theta_0,\eta_0}} 

\newcommand{\lra}{\longrightarrow}

\newcommand{\calF}{\mathcal{F}}
\newcommand{\calG}{\mathcal{G}}
\newcommand{\calH}{\mathcal{H}}
\newcommand{\calL}{\mathcal{L}}
\newcommand{\calM}{\mathcal{M}}
\newcommand{\calP}{\mathcal{P}}

\newcommand{\calS}{\mathcal{S}}

\newcommand{\bbP}{\mathbb{P}} 
\newcommand{\bbR}{\mathbb{R}}
\newcommand{\bbE}{\mathbb{E}}

\newcommand{\bbG}{\mathbb{G}}

\newcommand*\widebar[1]{%
	\hbox{%
		\vbox{%
			\hrule height 0.5pt 
			\kern0.5ex
			\hbox{%
				\kern-0.1em
				\ensuremath{#1}%
				\kern-0.1em
			}%
		}%
	}%
} 

\begin{document}

	\title{Bayesian High-dimensional Semi-parametric Inference beyond sub-Gaussian Errors}
	\author[1]{Kyoungjae Lee}
	\author[2]{Minwoo Chae}
	\author[3]{Lizhen Lin}
	\affil[1]{Department of Statistics, Inha University}
	\affil[2]{Department of Industrial and Management Engineering, Pohang University of Science and Technology}
	\affil[3]{Department of Applied and Computational Mathematics and Statistics, The University of Notre Dame}

\maketitle

\begin{abstract}
We consider a sparse linear regression model with unknown symmetric error under the high-dimensional setting.
The true error distribution is assumed to belong to the locally $\beta$-H\"{o}lder class with an exponentially decreasing tail, which does not need to be sub-Gaussian. 
We obtain posterior convergence rates of the regression coefficient and the error density, which are nearly optimal and adaptive to the unknown sparsity level. 
Furthermore, we derive the semi-parametric Bernstein-von Mises (BvM) theorem to characterize asymptotic shape of the marginal posterior for regression coefficients.
Under the sub-Gaussianity assumption on the true score function, strong model selection consistency for regression coefficients are also obtained, which eventually asserts the frequentist's validity of credible sets.
\end{abstract}


\section{Introduction}
We consider the linear regression model
\bean\label{model}
Y  &=& X \theta + \epsilon , 
\eean
where $Y =(Y_1,\ldots,Y_n)^T \in \bbR^n$ is a vector of response variables, $X= (x_{ij})\in \bbR^{n\times p}$ is the $n\times p$ matrix of covariates whose $i$-th row is $x_i^T = (x_{i1},\ldots, x_{ip})$, $\theta \in \bbR^p$ is the $p$-dimensional regression coefficient and  $\epsilon =(\epsilon_1,\ldots, \epsilon_n) \in \bbR^n$ is the vector of random errors with $\epsilon_i \overset{i.i.d.}{\sim} \eta$ for $i=1,\ldots,n$.
Statistical inference with the model \eqref{model} in high-dimensional settings has received increasing attention in recent years.
For the estimability of  $\theta$ under large $p$, certain sparsity condition is often imposed which assumes most components of $\theta$ are nearly zero.
Under the sparsity assumption, regularization methods have been at the center of statistical research due to their computational tractability, ease of interpretation, elegant theory and good performance in practice.
Some pioneering references include \cite{tibshirani1996regression}, \cite{fan2001variable}, \cite{tibshirani2005sparsity}, \cite{zou2005regularization}, \cite{zou2006adaptive}, \cite{candes2007dantzig} and \cite{zhang2014confidence}.
We also refer to  the monograph \cite{buhlmann2011statistics} for reviews with abundant examples.

In a Bayesian framework, the sparsity can be expressed through a prior on $\theta$ for which
there are two well-known classes: spike-and-slab and continuous shrinkage priors.
The former has been considered as the gold standard for sparse priors supported by rich theory, see \cite{castillo2015bayesian}, \cite{rovckova2016spike}, \cite{martin2017empirical} and \cite{chae2019bayesian}.
Continuous shrinkage priors have been developed as computationally efficient alternatives of spike-and-slab prior, see \cite{polson2010shrink}, \cite{carvalho2010horseshoe}, \cite{armagan2013generalized}, \cite{armagan2013posterior} and \cite{bhattacharya2015dirichlet}.

With regard to the high-dimensional regression model \eqref{model}, there are three fundamental problems attracting statistical interest: i) recovery of $\theta$; ii) selection of nonzero coefficients; and iii) quantifying the uncertainty of inference.
Note that even for Bayesian methods, it is  common  to analyse the performance of those methods from a frequentist's perspective by assuming a true data-generating distribution.
Under the assumption that errors are i.i.d.\ from the standard Gaussian, \cite{castillo2015bayesian} investigated the posterior convergence rate, strong model selection consistency and Bernstein-von Mises (BvM) theorem.
Slightly different sets of conditions and priors also lead to similar results, see \cite{shin2015scalable}, \cite{song2017nearly}, \cite{yang2016computational}, \cite{martin2017empirical}, \cite{yang2017posterior}.
Although some of their results, {\it e.g.} the recovery of $\theta$, tend to be robust to the misspecification of error distribution, Gaussian models have certain limitations; for example, they are vulnerable to outliers.
Some theoretical justification for this can be found in \cite{castillo2015bayesian} and \cite{buhlmann2011statistics}.

Another problem of a misspecified Gaussian model arises in model selection.
It should be noted that the sub-Gaussianity of the score function is a very important condition for consistent model selection, see \cite{kim2016consistent} and \cite{chae2019bayesian}.
Although it is not clear whether this is a necessary condition, empirical results given in \cite{rossell2017tractable} show that a Gaussian model might lead to inconsistency in model selection when true error distributions are heavy-tailed.
There are a few works concerning Bayesian variable selection beyond the Gaussian assumption, which however  often suffered from lack of theory in high-dimensional setting.
See \cite{rossell2017tractable} and references therein for recent advances on Bayesian variable selection without Gaussianity.

Uncertainty quantification, in particular its theoretical justification, is perhaps the most difficult task.
In Bayesian methods, the uncertainty of parameters based on posteriors is typically expressed through a credible set, which has frequentist's validity in a smooth parametric model by the BvM theorem, see \emph{e.g.} \cite{van1998asymptotic}.
Although the BvM theorem cannot be fully extended to high- or infinite-dimensional models, in some models with carefully chosen priors, credible sets can provide valid confidence satisfying certain frequentist's criteria of optimality, often called as non- or semi-parametric BvM theorem, see \cite{castillo2013nonparametric}, \cite{castillo2014bernstein}, \cite{castillo2015bernstein}, \cite{panov2015finite} and \cite{chae2019semiparametric}.
If the model is misspecified, however, the credible set loses the frequentist's validity even in a very simple parametric model \citep{kleijn2012bernstein}.
Some adjusting techniques are known \citep{yang2016posterior}, but they are not  applicable more generally.

In this paper, we study frequentist's property of Bayesian methods for model \eqref{model} by investigating large sample behavior of the posterior distributions.
We assume a symmetric error density $\eta$ rather than assuming a Gaussian error density.
The symmetric assumption might be slightly restrictive in practice, but a good compromise for the theoretical analysis.
In fact, a zero mean or median condition might be more realistic, but without symmetric assumption, uncertainty quantification is challenging in a semi-parametric Bayesian framework. 
%

Asymptotic properties of the posterior distribution in a high-dimensional semi-parametric regression model has been extensively studied in \cite{chae2019bayesian} under a rather strong assumption on $\eta$.
In particular, they assumed that $\eta$ is a mixture of Gaussians with a compactly supported mixing distribution, still falling into a sub-Gaussian framework.
In this paper, we use the result of \cite{shen2013adaptive} to eliminate this strong assumption.
Specifically, the true error density will be assumed to be in a locally $\beta$-H\"{o}lder class with an exponentially decreasing tail.
This is a much weaker assumption than that given in \cite{chae2019bayesian}. 
In particular, the true error density need to be neither a mixture of Gaussians nor sub-Gaussian. 
For the prior, a spike-and-slab and a symmetrized Dirichlet process (DP) mixture priors are imposed on $\theta$ and $\eta$, respectively.
Asymptotic results given in this paper provide reasonable sufficient conditions for the frequentist's validity on i) recovery of $\theta$, ii) variable selection, and iii) uncertainty quantification.

It would be worthwhile to mention some technical contributions of this paper.
First of all, our results allow error densities whose tails are thicker than sub-Gaussian for which well-known concentration bounds such as the Hoeffding's inequality make the proof simpler.
Although the results are limited to exponentially decaying tails, it is highly expected that recent advances on heavy tail distributions \citep{canale2017posterior} are also applicable.
Secondly, we provide simpler proof for posterior convergence rates compared to  that of \cite{chae2019bayesian}.
To derive the posterior convergence rates, they used the misspecified LAN (local asymptotic normality) and some bounded conditions for empirical process, which  turn out to be not necessary using our techniques.

The rest of the paper is organized as follows.
In section \ref{sec:prelim}, we define the model and prior with some preliminary materials.
In section \ref{sec:main}, main results on posterior convergence rates, asymptotic shape and selection  property are presented.
Concluding remarks follow in section \ref{sec:disc}, and technical proofs are given in the Supplementary Material.

\section{Preliminaries}\label{sec:prelim}
\subsection{Notations}\label{subsec:notation}
For any positive sequences $a_n$ and $b_n$, $a_n = o(b_n)$ implies that $a_n/b_n \lra 0$ as $n\to\infty$. 
We denote $a_n \lesssim b_n$, or equivalently $a_n = O(b_n)$, if $a_n \le C b_n$ for all sufficiently large $n$ and some constant $C>0$, which is an absolute constant or at least does not depend on $n$ and $p$.
For any $x \in \bbR$,  $\lfloor x \rfloor$ is the largest integer which is smaller than or equal to $x$.
For any constants $a$ and $b$, we denote $a \vee b$ as the maximum of $a$ and $b$. 
We denote the indicator function for some set $A$ as $I_A(\cdot)$ and $I(\cdot\in A)$.
For any $\theta\in \bbR^p$, the support of $\theta$ is denoted by $S_\theta$, which is the nonzero index of $\theta$, i.e. $S_\theta = \{1\le i \le p: \theta_i \neq 0 \}$. We denote the cardinality of $S_\theta$ as $s_\theta = |S_\theta|$. For any index set $S \subseteq \{1,\ldots,p \}$ and $n\times p$ matrix $X$, let $\theta_S = (\theta_i)_{i\in S} \in \bbR^{|S|}$, $\widetilde{\theta}_S = (\theta_i I(i\in S) )_{1\le i \le p} \in \bbR^p$ and $X_S = (X_j)_{j\in S} \in \bbR^{n\times |S|}$, where $X_j$ is the $j$-th column of $X$.
For any $y\in \bbR$ and density $\eta$, we denote $\ell_\eta(y) = \log \eta(y)$, $\dot{\ell}_\eta(y) = \partial \ell_\eta(y)/\partial y$, $\ddot{\ell}_\eta(y) = \partial^2 \ell_\eta(y)/(\partial y)^2$ and $\dddot{\ell}_\eta(y) = \partial^3 \ell_\eta(y)/(\partial y)^3$, whenever they exist.
Similarly, for any $x \in \bbR^p$ and $\theta \in \bbR^p$, let $\ell_{\theta,\eta}(x,y) = \ell_\eta(y- x^T\theta)$, $\dot{\ell}_{\theta,\eta}(x,y) = \dot{\ell}_\eta(y-x^T\theta)x$ and $\ddot{\ell}_{\theta,\eta}(x,y) = \ddot{\ell}_\eta(y-x^T\theta)x x^T$.
Let $\bbE_{\theta_0,\eta_0}$ be the expectation under $\bbP_{\theta_0, \eta_0}$ and $\bbP_{\theta,\eta}$ be the probability measure corresponding to the model \eqref{model}.
We denote $\bbE_{\eta_0} = \bbE_{0, \eta_0}$   for simplicity of exposition.
For given a sequence of random variables $Y_n$,  $Y_n = o_{P_0}(1)$ means that $Y_n$ converges to zero in $\bbP_{\theta_0,\eta_0}$-probability as $n\to\infty$.
For given a real function $f:\bbR^p \times \bbR \mapsto \bbR$ and the data $D_n = ((Y_i,x_i))_{i=1}^n$ from the model \eqref{model},
we define $L_n(\theta,\eta) = \sum_{i=1}^n \ell_{\theta,\eta}(x_i,Y_i)$, $R_n(\theta,\eta) = \prod_{i=1}^n \eta(Y_i-x_i^T\theta)/\eta_0(Y_i-x_i^T\theta_0)$,
\bea
\bbP_n f &=& \frac{1}{n} \sum_{i=1}^n f(x_i, Y_i),   \\
\bbG_n f &=& \frac{1}{\sqrt{n}} \sum_{i=1}^n \left\{ f(x_i, Y_i) - \bbE_{\theta_0,\eta_0} \left[f(x_i,Y_i)\right]  \right\} \,\text{ and} \\
V_{n,\eta} &=& \frac{1}{n} \sum_{i=1}^n\bbE_{\theta_0,\eta_0}\left[ \dot{\ell}_{\theta_0,\eta} \dot{\ell}^T_{\theta_0,\eta_0}(x_i,Y_i) \right] .
\eea
Note that $V_{n,\eta} = \nu_\eta \sg$, where $\nu_\eta =  \bbE_{\eta_0}(\dot{\ell}_\eta \dot{\ell}_{\eta_0} )$, $\sg = n^{-1} X^T X$.
Let $N_{n,\eta,S}$ be the $|S|$-dimensional normal distribution with mean $V_{n,\eta,S}^{-1} G_{n,\eta,S}$ and variance $V_{n,\eta,S}^{-1}$, where $G_{n,\eta,S}$ is the $|S|$-dimensional projection of $\bbG_n \dot{\ell}_{\theta_0,\eta}$, $V_{n,\eta,S}= \nu_\eta \sg_S$ and $\sg_S = n^{-1}X_S^T X_S$.
For simplicity, we denote $G_{n,\eta_0,S}, V_{n,\eta_0,S}$ and $N_{n,\eta_0,S}$ as $G_{n,S}, V_{n,S}$ and $N_{n,S}$, respectively.
For given positive real numbers $a$ and $b$, we denote $IG(a,b)$ as an inverse gamma distribution whose shape and scale parameters are $a$ and $b$, respectively. 
For given positive integer $p$, $\mu_0 \in \bbR^p$ and $p\times p$ positive definite matrix $\sg_0$, we denote $N_p(\mu_0,\sg_0)$ as a $p$-dimensional normal distribution with mean $\mu_0$ and covariance matrix $\sg_0$.

For any $\theta\in \bbR^p$, denote the vector $\ell_q$-norm as $\|\theta\|_q := \left(\sum_{j=1}^p |\theta_i|^q  \right)^{1/q}.$  
For any pair of densities $\eta_1$ and $\eta_2$ with respect to a probability measure $\mu$, define the total variation and Hellinger distance as $d_V(\eta_1,\eta_2) := \int|\eta_1-\eta_2| d\mu$ and $d_H^2(\eta_1, \eta_2) := \int (\sqrt{\eta_1} - \sqrt{\eta_2})^2 d\mu$, respectively. 
For any pairs of vectors $\theta^1,\theta^2 \in \bbR^p$  and densities $\eta_1,\eta_2$, we define the mean Hellinger distance as 
\bea
d_n^2\left( (\theta^1,\eta_1) , (\theta^2,\eta_2) \right) &=& \frac{1}{n} \sum_{i=1}^n d_H^2(p_{\theta^1,\eta_1,i}, p_{\theta^2,\eta_2,i} ),
\eea
where $p_{\theta,\eta,i}(y) = \eta(y- x_i^T \theta)$.

\subsection{Prior}

As mentioned earlier, we consider the following model
\bea
Y_i &=& x_i^T \theta + \epsilon_i , \\
\epsilon_i &\overset{i.i.d.}{\sim}& \eta , \quad i=1,\ldots,n,
\eea
where $\theta\in\bbR^p$ and $\eta$ is a symmetric density.
We impose prior distributions on $\theta$ and $\eta$ to conduct  Bayesian inference. 
Let $\Theta = \bbR^p$ and $\calH$ be the class of symmetric and continuously differentiable densities equipped with the Hellinger metric. We use a product prior $\Pi = \Pi_\Theta \times \Pi_\calH$ for $(\theta,\eta)$, where $\Pi_\Theta$ and $\Pi_\calH$ are Borel probability measures on $\Theta$ and $\calH$, respectively.

For the prior $\Pi_\Theta$ on the coefficient vector $\theta$, we select (i) the number of nonzero components $s$ from a prior $\pi_p(s)$ on $\{0,\ldots,p\}$, (ii) a random set $S \subseteq \{1,\ldots,p \}$ whose cardinality is $s=|S|$ from the uniform prior, and (iii) the nonzero values $\theta_S$ from a prior $g_S$ on $\bbR^{|S|}$ in turn.
Specifically, we consider the following prior distribution on $(S,\theta)$:
\bea
(S,\theta) &\mapsto& \pi_p(|S|) \, \frac{1}{\binom{p}{|S|}} \, g_S(\theta_S) \, \delta_0 (\theta_{S^c}),
\eea
where $\delta_0$ is the Dirac measure at 0.
This type of prior has been studied by \cite{george2000calibration}, \cite{scott2010bayes}, \cite{castillo2012needles} and \cite{castillo2015bayesian}.
For the prior $\pi_p$ and $g_S$, we assume that
\bean
A_1 p^{-A_3} \pi_p(s-1)  &\le& \pi_p (s) \,\,\le\,\, A_2 p^{-A_4} \pi_p(s-1), \quad s=1, \ldots , p \label{prior_p} \\
g_S(\theta_S) &=& \left(\frac{\lambda}{2}\right)^{|S|} \exp ( - \lambda \|\theta_S\|_1) , \quad \frac{\sqrt{n}}{p} \le \lambda \le \sqrt{n \log p}, \label{prior_gS}
\eean
for some positive constants $A_1,A_2,A_3$ and $A_4$. 
Note that the prior $g_S$ is the product of the Laplace distribution $g(\theta) = \lambda \exp(-\lambda|\theta|)/2$, i.e., $g_S(\theta_S)= \prod_{i\in S}g(\theta_i)$.

For the prior $\Pi_\calH$ on the error density $\eta$, we consider the location mixture of a symmetrized DP, 
\bea
\eta(x) &=& \int \phi_\sigma (x-z) d \widebar{F}(z), \\
F &\sim& DP(\alpha) , \\
\sigma^2 &\sim& G 
\eea
where $\phi_\sigma (x)  := (\sqrt{2\pi} \sigma)^{-1} \exp \{-x^2/(2\sigma^2) \}, \widebar{F} := (F+ F^-)/2, dF^-(z) := dF(-z)$ and $DP(\alpha)$ is the Dirichlet process with a finite positive measure $\alpha$. 
For the base measure $\alpha$ and the prior on $\sigma^2$, we further assume that
\bean
\widebar{\alpha} &\in& {\cal{M}}[ -C'n , C'n  ] , \label{al_n_n}\\
\widebar{\alpha}( [-x, x]^c ) &\le& \exp(-C'' x^{a_1}) \text{ for all sufficiently large } x>0, \label{al_tail}\\
G(\sigma^2 \le x) &\le& \exp(- C'' x^{-a_2} ) \text{ for all sufficiently small } x>0 , \label{G1}\\
G(\sigma^2 \ge x) &\le& x^{-a_3} \text{ for all sufficiently large } x>0 ,\label{G2}\\
G(s < \sigma^{-2} < s(1+t)) &\ge& a_6 s^{a_4} t^{a_5} \exp(-C'' s^{\kappa/2}) \text{ for any $s>0$ and $t\in(0,1)$,}\quad\, \label{G3}
\eean
for some positive constants $a_1,\ldots, a_6, C', C''$ and $\kappa$, where $\widebar{\alpha}=  \alpha /\alpha([-C'n,C'n])$ and $\calM[a,b]$ is the set of  probability measures on $(a,b)$.
We assume that $\widebar{\alpha}$ has a positive density function on $(-C' n, C'n)$.

We need additional assumptions to achieve a distributional approximation and model selection consistency.
Specifically, we assume that 
\bean
\widebar{\alpha} &\in& \calM[-C'(\log n)^{\frac{2}{\tau}}, C'(\log n)^{\frac{2}{\tau}} ] , \label{al1} \\
G &\in& \calM[0, C' \log n], \label{G4}
\eean
and  $\widebar{\alpha}=  \alpha /\alpha([-C'(\log n)^{\frac{2}{\tau}}, C'(\log n)^{\frac{2}{\tau}} ])$ has a positive density function on $(-C'(\log n)^{\frac{2}{\tau}}, C'(\log n)^{\frac{2}{\tau}} )$, 
where $\tau>0$ will be used to define true parameter class (condition \hyperref[D2]{(D2)}) in section \ref{subsec:true den}.

\begin{remark}
	The above prior conditions are mild which include popular prior choices.
	If we choose $\widebar{\alpha}$ as a truncated normal distribution on interval $[-n,n]$, conditions \eqref{al_n_n} and \eqref{al_tail} are satisfied with $a_1=2$.
	If we consider $\sigma^{m_0}\sim IG(a_0,b_0)$ for some positive constants $a_0$, $b_0$ and $m_0$,  conditions \eqref{G1}-\eqref{G3} are satisfied with $a_2= m_0/2$ and $\kappa=m_0$.
	For conditions \eqref{al1} and \eqref{G4}, it suffices to consider the truncated normal and inverse-gamma distribution on $(-C'(\log n)^{\frac{2}{\tau}}, C'(\log n)^{\frac{2}{\tau}} )$ and $(0, C'\log n)$, respectively, for some large constant $C'>0$.
	As $n$ grows to infinity, the above supports in \eqref{al1} and \eqref{G4} are getting close to the whole supports, $\bbR$ and $\bbR^+ = (0, \infty)$.
\end{remark}

\subsection{True Parameter Class}\label{subsec:true den}

We focus on the ``large $p$ and small $n$'' setting, i.e. $p\ge n$, throughout the paper.
We assume $\theta_0 \in \bbR^p$ to be a $s_0$-sparse vector, which means the number of nonzero elements of $\theta_0$ is equal to $s_0$. The support of $\theta_0$ is denoted by $S_0 = S_{\theta_0}$.  
Further conditions on the sparsity $s_0$ and the magnitude of $\theta_0$ will be introduce in the main theorems in section \ref{sec:main}.
We introduce here conditions \hyperref[D1]{(D1)}-\hyperref[D5]{(D5)} for the true error density $\eta_0$:  
\begin{itemize}
	\item[(D1)]\label{D1} (Locally $\beta$-H\"{o}lder class) for given positive constants $\beta$, $\tau_0$ and a real-valued function $L$, $\eta_0 \in {\cal{C}}^{\beta,L,\tau_0}(\bbR)$ where ${\cal{C}}^{\beta,L,\tau_0}(\bbR)$ is the class of every density $\eta$ whose $k$th order derivative $\eta^{(k)}$ exists up to $k \le \lfloor \beta \rfloor$ and for $k_1 = \lfloor \beta \rfloor$, 
	\bea
	\left| \eta^{(k_1)}(x+y) - \eta^{(k_1)}(x)  \right| &\le& L(x) \exp (\tau_0 y^2) |y|^{\beta - \lfloor \beta \rfloor} ,\quad \forall x,y \in \bbR.
	\eea
	\item[(D2)]\label{D2} (Light tail) There exist positive constants $a,b$ and $\tau$ such that
	\bea
	\eta_0(x) &\le& \exp(- b |x|^\tau) , \quad |x| >a.
	\eea
	\item[(D3)]\label{D3} There exists a constant $\upsilon > 0$ such that $\bbE_{\eta_0} \left( |\eta_0^{(k)} |/\eta_0 \right)^{(2\beta+\upsilon)/k}< \infty$ and $\bbE_{\eta_0} \left( L/\eta_0 \right)^{(2\beta+\upsilon)/\beta} < \infty$ for $1\le k\le \lfloor \beta \rfloor$.
	\item[(D4)]\label{D4} (Symmetry) $\eta_0(x) = \eta_0(-x)$ and $\eta_0(x)>0$ for all $x \in \bbR$.
	\item[(D5)]\label{D5} there exist positive constants $\gamma_1,\gamma_2,\gamma_3,b', C_{\eta_0}$ and $\tau' <\tau$ such that for any $y\in\bbR$, 
	\bean
	| \dot{\ell}_{\eta_0}(y)| &\le&  C_{\eta_0}(|y|^{\gamma_1} + 1), \label{eta0_1} \\
	| \ddot{\ell}_{\eta_0}(y)| &\le&  C_{\eta_0}(|y|^{\gamma_2} + 1), \label{eta0_2} \\
	| \dddot{\ell}_{\eta_0}(y)| &\le&  C_{\eta_0}(|y|^{\gamma_3} + 1), \label{eta0_3} 
	\eean
	and, for any small $|x|$,
	\bean
	\frac{\eta_0(y+x)}{\eta_0(y)} &\le& C_{\eta_0} e^{b' |y|^{\tau'}}. \label{eta0_4}
	\eean
\end{itemize}

Now we describe the above conditions in more details.
Condition \hyperref[D1]{(D1)}, locally $\beta$-H\"{o}lder class, has been extensively studied in \cite{kruijer2010adaptive}, \cite{shen2013adaptive}, \cite{canale2017posterior} and \cite{bochkina2017adaptive}.
This class is much more general than the H\"{o}lder class because it only requires the local smoothness by adopting $L(x)$ instead of a constant $L>0$.
Furthermore,  due to condition \hyperref[D2]{(D2)}, it is essentially weaker than the condition in \cite{kruijer2010adaptive}, which assumes $\log \eta_0 \in {\cal{C}}^{\beta,L,\tau_0}(\bbR)$ \citep{shen2013adaptive}.

Condition \hyperref[D2]{(D2)} ensures that the true density has an exponentially light tail.
It is mainly required to prove the prior thickness condition for the density part and use the Hanson-Wright inequality for the strong model selection consistency.
The technical details for the former issue can be found in \cite{shen2013adaptive} (Lemma 2, Theorem 3 and Proposition 1).
Recently, \cite{bochkina2017adaptive} adopted much weaker tail condition, 
$\int_x^\infty y^2 \eta_0(y) dy  \le C(1+x)^{-\tau}$ for some constants $C>0$ and $\tau>0$, which includes some polynomially decreasing tail densities.
However, they considered only the densities on $\bbR^+$, and it is unclear whether their techniques are applicable to the densities on $\bbR$.

Condition \hyperref[D3]{(D3)} is needed for the prior thickness condition for the density \citep{shen2013adaptive} and implicitly controls the tail behavior of $\eta_0$.
It has been commonly used in literature including \cite{kruijer2010adaptive}, \cite{shen2013adaptive} and \cite{bochkina2017adaptive}.

Condition \hyperref[D4]{(D4)} is not needed for proving the optimal convergence results (Theorems \ref{thm:dimupper} and \ref{thm:convrate_mH}).
However, the symmetric assumption will play an important role in proving the BvM theorem (Theorem \ref{thm:BvM}).
Based on current techniques in this paper, this assumption is also needed to prove Corollary \ref{cor:convrate_H} and Corollary \ref{cor:convrate_coef}, although we suspect that this can be weakened.



Condition \hyperref[D5]{(D5)} is required only for the BvM theorem and selection consistency results.
This condition is closely related to the tail of $\eta_0$ and satisfied for a wide range of densities. 
For example, if $\eta_0(y) \propto \exp(-a|y|^b)$ for some constants $a,b >0$ and every large enough $|y|$, condition \hyperref[D5]{(D5)} is met.
In fact, it holds unless $\eta_0$ has an extremely thin tail.

\subsection{Design Matrix}
We consider a fixed design matrix $X\in \bbR^{n\times p}$ and assume that every element of the design matrix is bounded by $\sqrt{\log p}$ up to some constant, i.e. $\sup_{i,j}|x_{ij}| \le M \sqrt{\log p}$ for some constant $M>0$. 
The upper bound for entries of the design matrix is introduced due to technical reasons, and some recent works \citep{narisetty2019skinny, song2017nearly} on high-dimensional Bayesian inference also used similar conditions.
In this paper, we require this condition mainly to (i) derive the posterior convergence rate for $\|X(\theta-\theta_0)\|_2$ using Corollary 3.2 of \cite{chae2019bayesian} and (ii) obtain upper bounds for $\dot{\ell}_{\theta,\eta}(x,y)$ (or its derivatives) based on $\dot{\ell}_{\eta}(y)$ (or its derivatives).

In high-dimensional linear regression model \eqref{model}, certain {\it regularity conditions} have been imposed on the design matrix $X$ for the estimability of $\theta$. 
In this paper, we define the {\it uniform compatibility number} by 
\bea
\phi^2(s) &=& \inf\left\{ s_\theta \cdot\frac{ \theta^T \sg \theta}{\|\theta\|_1^2} :\,\, \theta\in\bbR^p,\,\, 0< s_\theta \le s \right\}
\eea
and the {\it restricted eigenvalue} by
\bea
\psi^2(s) &=& \inf \left\{  \frac{\theta^T \sg \theta}{\|\theta\|_2^2} :\,\, \theta\in\bbR^p,\,\, 0< s_\theta \le s \right\}
\eea
for any $1\le s \le p$, where $\sg = n^{-1}X^T X$, $s_\theta=|S_\theta|$ and $S_\theta$ is the support of $\theta$. 
These quantities have been commonly used in literature \citep{van2009conditions,bickel2009simultaneous,castillo2015bayesian}, and the bounded below conditions have been introduced for  consistent estimation.
Note that the infimum, which is used to define compatibility number (or restricted eigenvalue), is often taken over all $S\subseteq\{1,\ldots,p\}$ and $\theta \in\bbR^p$ such that $\|\theta_{S^c}\|_1  \le c\|\theta_S\|_1$ for some constant $c>0$.
However, our definitions for $\phi^2(s)$ and $\psi^2(s)$ focus on sparse vectors $\theta$ such that $0<s_\theta \le s$.
For example, \cite{castillo2015bayesian} uses similar definitions.  
Bounded below assumption on $\psi(s)$ is required for the convergence rate under $\ell_2$ norm and BvM result, while the same assumption on $\phi(s)$ is required for the convergence rate under $\ell_1$ norm.
If the restricted eigenvalue $\psi^2(s)$ is bounded away from zero, it implies that $\sg_S$ is positive definite for any $|S|= s$. 
Note that $\psi(s) \le \phi(s)$ because $\|\theta\|_1^2 \le s_\theta \|\theta\|_2^2$ by the Cauchy-Schwartz inequality.   
Thus, the restricted eigenvalue conditions is stronger than the uniform compatibility number condition. 

Because $p \ge n$, if we consider, for example,  a random design matrix $X= (x_{ij})$, where $x_{ij}$'s are random samples from the standard normal, $\sup_{i,j}|x_{ij}| \le M \sqrt{\log p}$ and the restricted eigenvalue condition are met with high probability tending to 1 as $p\to\infty$.
Furthermore, by Lemma 6.1 in \cite{narisetty2014bayesian}, these conditions are also satisfied with high probability tending to 1 if the rows of $X$ are independent isometric sub-Gaussian random vectors.

\section{Main Results}\label{sec:main}
\subsection{Posterior Convergence Rates}\label{subsec:conv}
The first theorem is about the model dimension which states that the posterior distribution puts most of its mass on moderately small dimensional models.
We denote the posterior distribution based on $D_n$ as $\Pi(\cdot \mid D_n)$.

\begin{theorem}\label{thm:dimupper}
	Assume that conditions \eqref{prior_p}-\eqref{G3} hold, $\lambda\|\theta_0\|_1=O(s_0\log p)$  and $\log p \le n^2$. Then, for any $\eta_0$ satisfying \hyperref[D1]{(D1)}-\hyperref[D4]{(D4)}, there exists a constant $K_{\rm dim}>1$ not depending on $n$ and $p$ such that
	\bea
	\Eaa \Pi \left( s_\theta > K_{\rm dim} \big\{ s_0 \vee  n^{\kappa^*/(2\beta +\kappa^*)} (\log n)^{2t -1}  \big\}   \mid D_n \right) &=& o(1)
	\eea
	where $\kappa^*:= (\kappa \vee 1)$ and $t > \{ \kappa^*(1+ \tau^{-1}+ \beta^{-1}) + 1 \}/(2+ \kappa^* \beta^{-1})$.
\end{theorem} 

Since we use a Laplace prior for nonzero coefficients, the condition $\lambda\|\theta_0\|_1 = O(s_0\log p)$ might seem to a bit restrictive.
Note that this condition can be avoided in Gaussian models by utlizing explicit form of the log-likelihood, see \cite{castillo2015bayesian}, \cite{van2016conditions} and \cite{gao2015general}.
To use the same technique in our semi-parametric model, quadratic approximation of the log-likelihood should be preceded, for which empirical process techniques can be applied.
However,  quadratic approximation is highly difficult when models have many nonzero coefficients.
Therefore, the proof of Theorem \ref{thm:dimupper} heavily relies on the prior, requiring an additional condition $\lambda\|\theta_0\|_1 = O(s_0\log p)$.
An empirical Bayes approach proposed in \cite{martin2017empirical} might be an alternative way to relax this condition.
However, the choice of the least squared estimators as the center of the prior may yield another problems when errors have heavier tails than the sub-Gaussian tail.
Since we believe the condition $\lambda\|\theta_0\|_1 = O(s_0\log p)$ is not too restrictive under the large $\lambda$ regime, we leave the problem of relaxing this condition as future work.


For a given $t > \{ \kappa^*(1+ \tau^{-1}+ \beta^{-1}) + 1 \}/(2+ \kappa^* \beta^{-1})$, let $s_n := 2 K_{\rm dim} \{ s_0 \vee n^{\kappa^*/(2\beta + \kappa^*)} (\log n)^{2t -1}  \}$.
Theorem \ref{thm:dimupper} effectively reduces the meaningful parameter space when $s_n$ is not too big and makes the theoretical development easier.
Theorem \ref{thm:convrate_mH} describes a result on posterior convergence rate under the mean Hellinger distance. 
The obtained rate has the term $s_n$ defined above, where $s_0$ and $n^{\kappa^*/(2\beta + \kappa^*)} (\log n)^{2t -1}$ come from the coefficient and density estimation, respectively.

\begin{theorem}\label{thm:convrate_mH}
	Assume that conditions \eqref{prior_p}-\eqref{G3} hold, $\lambda\|\theta_0\|_1 = O(s_0 \log p)$ and $s_n  \log p =o(n)$. Then, for any $\eta_0$ satisfying \hyperref[D1]{(D1)}-\hyperref[D4]{(D4)},
	\bea
	\bbE_{\theta_0,\eta_0} \Pi \left( d_n \big( (\theta,\eta), (\theta_0,\eta_0) \big) > K_{\rm Hel} \sqrt{\frac{s_n \log p }{n} }  \,\,\bigg|\,\, D_n\right) &=& o(1), 
	\eea
	for some constant $K_{\rm Hel}>0$  not depending on $n$ and $p$.
\end{theorem}

\begin{remark} 
The symmetric condition \hyperref[D4]{(D4)} is not directly used in the proof of Theorems \ref{thm:dimupper} and \ref{thm:convrate_mH}.
Hence, they can be easily re-stated without \hyperref[D4]{(D4)}.
We did not try to re-state them because it entails a redefinition of the prior and a lot of minor changes. We need the symmetric assumption for the BvM theorem, particularly for proving that the score function has zero expectation, that is, $\mathbb{E}_{\theta_0, \eta_0} \dot\ell_{\theta_0, \eta} = 0$ for symmetric $\eta$.
This will play an important role in the proof of the misspecified LAN, see Lemma \ref{lemma:LAN} of the Supplement.
Finally, we note that the current proof of Corollaries 1 and 2 below relies on the symmetric assumption \hyperref[D4]{(D4)}, but it might be possible to prove them without it.
\end{remark}

Based on Theorem \ref{thm:convrate_mH}, the posterior convergence rate of $\eta$ and $\theta$ can be achieved as follows.
The proof of Corollary \ref{cor:convrate_coef} is straightforward by Theorem \ref{thm:convrate_mH} and similar arguments used in the proof of Corollary 3.2 of \cite{chae2016arxiv}, so we omit the proof here. 

\begin{corollary}\label{cor:convrate_H}
	Under the conditions of Theorem \ref{thm:convrate_mH}, we have
	\bea
	\bbE_{\theta_0,\eta_0} \Pi \left( d_H( \eta, \eta_0 ) > K_{\rm eta} \sqrt{\frac{s_n \log p }{n} }  \,\,\bigg|\,\, D_n\right) &=& o(1), 
	\eea
	for some constant $K_{\rm eta}>0$  not depending on $n$ and $p$, and for any $\eta_0$ satisfying \hyperref[D1]{(D1)}-\hyperref[D4]{(D4)} with $2\beta + \upsilon \ge 2$.
\end{corollary}

\begin{corollary}\label{cor:convrate_coef}
	Under the conditions of Theorem \ref{thm:convrate_mH} and $s_n  \log p/\phi(s_n) = o(\sqrt{n})$, we have
	\bea
	\bbE_{\theta_0,\eta_0} \Pi \left( \|\theta-\theta_0\|_1 > K_{\rm theta} \frac{s_n}{\phi(s_n)} \sqrt{\frac{\log p }{n} }  \,\,\bigg|\,\, D_n\right) &=& o(1),  \\
	\bbE_{\theta_0,\eta_0} \Pi \left( \|\theta-\theta_0\|_2 > K_{\rm theta} \frac{1}{\psi(s_n)} \sqrt{\frac{s_n \log p }{n} }  \,\,\bigg|\,\, D_n\right) &=& o(1),\\
	\bbE_{\theta_0,\eta_0} \Pi \left( \|X(\theta-\theta_0)\|_2 > K_{\rm theta}  \sqrt{s_n \log p }  \,\,\bigg|\,\, D_n\right) &=& o(1),
	\eea
	for some constant $K_{\rm theta}>0$  not depending on $n$ and $p$, and for any $\eta_0$ satisfying \hyperref[D1]{(D1)}-\hyperref[D4]{(D4)}.  
\end{corollary}
If we assume that $\phi(s_n)$ is bounded away from zero, the condition $s_n\log p/\phi(s_n)$ $=o(\sqrt{n})$ in Corollary \ref{cor:convrate_coef} becomes $s_n\log p = o(\sqrt{n})$.
Similar condition was made by \cite{chae2019bayesian} to convert the convergence rate of the mean Hellinger distance $d_n((\theta,\eta),(\theta_0,\eta_0) )$ to that of $\|\theta-\theta_0\|_1$.
Note that \cite{chae2019bayesian} assumed $s_n\sqrt{\log p}/\phi(s_n)=o(\sqrt{n})$ under the bounded design matrix assumption, $\sup_{i,j}|x_{ij}| \le M$.
If we assume $\sup_{i,j}|x_{ij}| \le M$, the condition in Corollary \ref{cor:convrate_coef} is also relaxed to $s_n\sqrt{\log p}/\phi(s_n)=o(\sqrt{n})$.
It is a quite natural condition to obtain a meaningful convergence rate  tending to zero under the $\ell_1$-norm.

The posterior convergence rate in Theorem \ref{thm:convrate_mH} is {\it nearly} optimal if the hyperparameter $\kappa$ is set equal to $1$. Note that $\kappa^* = 1$ in this case.
For example, if $s_0 \ge n^{1/(2\beta + 1)} (\log n)^{2t -1}$, the posterior convergence rate with respect to the mean Hellinger distance is $\sqrt{s_0\log p/n}$, leading to the same marginal convergence rate for $\theta$ in $\ell_2$-norm.
Note that the minimax rate is $\sqrt{s_0 \log (p/s_0)/n }$ \citep{ye2010rate}. 
If $s_0 < n^{1/(2\beta + 1)} (\log n)^{2t -1}$, the marginal convergence for $\eta$ rate with respect to the Hellinger distance is $n^{-\beta/(2\beta+1)}\times$ $\sqrt{ (\log n)^{2t-1}\log p }$ which is the minimax rate up to a logarithmic factor and the same as that of \cite{shen2013adaptive}.
In conclusion, the global rate for the whole parameter $(\theta, \eta)$ is determined by the slower one among the two rates for $\theta$ and $\eta$, where both of them are close to the minimax rate provided that $\log p$ is negligible relative to $n$.

\subsection{Bernstein von-Mises Theorem}

In this subsection, we study  the distributional limit of the marginal posterior distribution for $\theta$.
Assume for a moment that $p$ is moderately slowly increasing and the model is not sparse.
Since we are working with a smooth semi-parametric model, it is highly expected that asymptotic shape of the map $\theta \mapsto L_n(\theta, \eta) - L_n(\theta_0, \eta)$ is quadratic around $\theta_0$ for every $\eta$.
Since the posterior mass is concentrated around $(\theta_0, \eta_0)$, we only need to consider $\eta$'s that are sufficiently close to $\eta_0$.
The assertion leads to the semi-parametric BvM theorem which guarantees the asymptotic efficiency of Bayes estimator.

With a sparse model considered in this paper, the marginal posterior distribution cannot converge to a single normal distribution unless posterior puts most of its mass on a single model.
To be more specific, note that the marginal posterior distribution of $\theta$ is given as
\bea
d \Pi(\theta \mid D_n) &=&  \sum_{S \subseteq \{1,\ldots,p\}} w_S \, dQ_S(\theta_S) \, d\delta_0(\theta_{S^c}),
\eea
where 
\bea
w_S &\propto& \frac{\pi_p(|S|)}{\binom{p}{|S|}} \int \int \exp \left( L_n(\widetilde{\theta}_S, \eta) - L_n(\theta_0,\eta_0) \right)  d\Pi_{\calH}(\eta) g_S(\theta_S) d\theta_S 
\eea
and
\bea
Q_S(\theta_S \in B) &=& \frac{\int_B \int \exp \left( L_n(\widetilde{\theta}_S, \eta) - L_n(\theta_0,\eta_0) \right)  d\Pi_{\calH}(\eta) g_S(\theta_S) d\theta_S }{\int \int \exp \left( L_n(\widetilde{\theta}_S, \eta) - L_n(\theta_0,\eta_0) \right)  d\Pi_{\calH}(\eta) g_S(\theta_S) d\theta_S}
\eea
for every measurable set $B\subseteq \bbR^{|S|}$.
Here, $Q_S$ can be understood as the conditional posterior distribution of $\theta_S$ given $S_\theta = S$.
If $|S|$ is not too large, $Q_S$ is expected to be asymptotically normal as in the semi-parametric BvM theorem described in the previous paragraph.
As a consequence, if there is a limit distribution of the marginal posterior for $\theta$, it should be a mixture of the form
\bea
d\Pi^\infty(\theta \mid D_n) &=& \sum_{S \subseteq\{1,\ldots,p\}} w_S\, n^{-\frac{|S|}{2}} \, dN_{n,S}( h_S ) \, d \delta_0(\theta_{S^c}),
\eea 
where $h_S =\sqrt{n}(\theta_S-\theta_{0,S})$, $n^{-\frac{|S|}{2}}$ is the determinant of the Jacobian matrix, and $N_{n,S}$ is defined in section \ref{subsec:notation}.
Theorem \ref{thm:BvM} says that the semi-parametric BvM theorem holds under slightly stronger condition than those needed for the posterior convergence rate results.

\begin{theorem}[Bernstein von-Mises]\label{thm:BvM}
	Assume that the prior conditions \eqref{prior_p}, \eqref{prior_gS}, \eqref{al_tail}-\eqref{G4} hold with $a_2=3$, $\lambda\|\theta_0\|_1 = O(s_0 \log p)$ and $\lambda s_n {\log p} = o(\sqrt{n})$.
	Further assume that 
	$s_n^6 \{ (\log p)^{11}\vee s_n^{\frac{5}{12}} \left( \log p \right)^{8 + \frac{11}{12}}   \}  =o(n^{1-\zeta}) $ 
	holds for some constant $\zeta>0$ and $\psi(s_n)$ is bounded away from zero. Then, we have
	\bea
	\Eaa \Big[d_V\left( \Pi(\cdot |D_n),  \Pi^\infty(\cdot| D_n)  \right)\Big] &=& o(1)
	\eea
	for any $\eta_0$ satisfying \hyperref[D1]{(D1)}-\hyperref[D5]{(D5)}. 
\end{theorem}

The bounded condition on $\psi(s_n)$ ensures that the quadratic term of log-likelihood ratio does not vanish.
The condition $\lambda s_n {\log p}=o(\sqrt{n})$ is required to wash out the prior effect and is a quite mild condition if we consider the small $\lambda$ regime such as $\lambda=\sqrt{n}/p$. 
To prove the BvM theorem, \cite{castillo2015bayesian} also used similar condition, $\lambda s_n \sqrt{\log p} = o(\|X\|)$, where $\|X\|$ is the maximum $\ell_2$-norm of the columns of matrix $X$.
Note that if $\log p =o(n)$ and  $x_{ij}$'s are random samples from $N(0,1)$, it coincide with $\lambda s_n {\log p}=o(\sqrt{n})$  with high probability  tending to $1$, as $p\to\infty$.

The condition $s_n^6 \{ (\log p)^{11}\vee s_n^{\frac{5}{12}} \left( \log p \right)^{8 + \frac{11}{12}}   \}  =o(n^{1-\zeta}) $ for some constant $\zeta>0$, are sufficient conditions for 
\bea
\int \sup_{\eta\in \calH_n^*} \left(\dot{\ell}_\eta(y) - \dot{\ell}_{\eta_0}(y) \right)^2 dP_{\eta_0}(y)
\eea
converging to zero at a certain rate, where $\calH_n^*$ is a neighborhood of $\eta_0$ to which the posterior distribution contracts. 
See Lemmas 6 and 12 in Supplementary Material for details.
To satisfy this condition, a certain level of smoothness of $\eta_0$ is essential.
For example, suppose we consider the prior $\sigma^6 \sim IG(a_0, b_0)$ for some positive constants $a_0$ and $b_0$, i.e., $a_2=3$ and $\kappa=6$. 
Then, the condition $\left( s_n\log p \right)^{6 + \frac{5}{12}}(\log p)^{\frac{5}{2}} = o(n^{1-\zeta})$ is satisfied when $(s_0 \log p)^{6 + \frac{5}{12}}(\log p)^{\frac{5}{2}} = o(n^{1-\zeta})$ and $n^{\frac{77}{4\beta+12}} (\log p)^{\frac{107}{12} } = o(n^{1-\zeta})$, which hold for $\beta > 16.25$ provided that $\log p$ is negligible relative to $n$.
\cite{chae2019bayesian} assumed $(s_0 \log p)^6 = o(n^{1-\zeta})$ for some constant $\zeta>0$ to establish the semi-parametric BvM theorem.
Our condition is slightly stronger due to relaxation on the tail condition of $\eta_0$.

\subsection{Strong Model Selection Consistency}
Theorem \ref{thm:selection} states that the posterior probability of $S_\theta$ for the strict supersets of the true model $S_0$ tends to zero. 
It implies that the posterior probability is asymptotically concentrated on the union of some strict subset of $S_0$ and possibly other coordinates of $S_0^c$.

\begin{theorem}[No superset]\label{thm:selection}
	Under the conditions of Theorem \ref{thm:BvM} and $\tau \ge 2\gamma_1$, we have
	\bea
	\Eaa \Pi ( S_\theta \supsetneq S_0 \mid D_n) &=& o(1)
	\eea
	for any $\eta_0$ satisfying  \hyperref[D1]{(D1)}-\hyperref[D5]{(D5)}, provided that $A_4 > K_{\rm sel}$ for some constant $K_{\rm sel}$ depending only on $\eta_0$.
\end{theorem}
Since we assume that $\eta_0(y) \lesssim \exp(-b|y|^{\tau})$ and $|\dot{\ell}_{\eta_0}(y)| \lesssim |y|^{\gamma_1}+C$, the condition $\tau \ge 2\gamma_1$ implies that $\dot{\ell}_{\eta_0}(y_i - x_i^T \theta_0)$ is a sub-Gaussian random variable. The sub-Gaussian assumption enables us to use the Hanson-Wright inequality \citep{hanson1971bound,wright1973bound}, which is one of  the key properties for proving Theorem \ref{thm:selection}.
Note that a normal distribution and a location-scale mixture of normal with compact mixing distribution satisfy the above condition.
One important consequence of Theorem \ref{thm:selection} is that if we assume that 
\bean\label{betamin}
\min \Big\{ |\theta_{0,j}|: \theta_{0,j}\neq 0,\,\, 1\le j \le p  \Big\} &\ge& \frac{K_{\rm theta}}{\psi(s_n)} \sqrt{ \frac{s_n\log p}{n}},
\eean
Corollary \ref{cor:convrate_coef} and Theorem \ref{thm:selection} guarantee the strong model selection consistency, which means $\Eaa\Pi(S_\theta = S_0\mid D_n) \lra 1$ as $n\to\infty$.
The above condition \eqref{betamin} is called the {\it beta-min} condition commonly assumed to obtain the model selection consistency \citep{castillo2015bayesian,song2017nearly}.
The following corollary asserts that one can achieve the selection consistency and efficiently capture the uncertainty of the nonzero coordinates under the beta-min condition.

\begin{corollary}[Selection]\label{cor:selection}
	Let $\widehat{\theta}_{S_0} = n^{-1/2} V_{n, S_0}^{-1} G_{n,S_0}+ \theta_{0,S_0}$, $\widehat{\sg}_{S_0} = n^{-1} V_{n, S_0}^{-1}$, and $\delta_{S_0^c}$ be the Dirac measure at $0 \in \bbR^{|S^c|}$. 
	Denote $\theta \sim N_{|S_0|}( \widehat{\theta}_{S_0} ,\, \widehat{\sg}_{S_0} ) \otimes \delta_{S_0^c}$ if $\theta_{S_0}\sim N_{|S_0|}( \widehat{\theta}_{S_0} ,\, \widehat{\sg}_{S_0} )$ and $\theta_{S_0^c} = 0$, independently.
	Under the conditions of Theorem \ref{thm:selection} and \eqref{betamin}, we have
	\bea
	\Eaa \left[ d_V\left( \Pi( \cdot |D_n), N_{|S_0|}( \widehat{\theta}_{S_0} ,\, \widehat{\sg}_{S_0} ) \otimes \delta_{S_0^c}  \right)  \right]  &=& o(1)
	\eea
	for any $\eta_0$ satisfying  \hyperref[D1]{(D1)}-\hyperref[D5]{(D5)}, provided that $A_4 > K_{\rm sel}$ for some constant $K_{\rm sel}$ depending only on $\eta_0$.
\end{corollary}

\begin{remark}
	\cite{yang2017posterior} proved the asymptotic normality for an individual coordinate $\theta_i$ without the beta-min condition.
	However, her results focus on the posterior distribution of an individual coordinate under the normal error distribution and cannot be extended to the posterior distribution of the whole $\theta$.
\end{remark}

\section{Discussion}\label{sec:disc}

In this paper, we study asymptotic properties of posterior distributions for high-dimensional linear regression models under unknown symmetric error.
We extend the previous works on Bayesian asymptotic theory to deal with much more general error densities beyond the sub-Gaussian class. To the best of our knowledge, this is the first work that has proved posterior convergence rates and BvM theorem for high-dimensional linear regression model without the sub-Gaussian assumption.
For the BvM theorem and selection consistency, the conditions, $s_n^6(\log p)^{11}=o(n^{1-\zeta})$ and 
$\left( s_n\log p \right)^{6 + \frac{5}{12}}(\log p)^{\frac{5}{2}} = o(n^{1-\zeta})$, are needed, which requires that the true error distribution is smooth enough.

Note that  algorithms for sampling a DP mixture  and a spike-and-slab prior can be suitably combined to generate MCMC samples from the posterior distribution in our semiparametric model, see Section 4 of \cite{chae2019bayesian}.
Although our theoretical analysis is limited to error densities with exponentially decaying tails, results of numerical experiments in \cite{chae2019bayesian} demonstrate that a semi-parametric estimator performs much better in prediction, model selection and uncertainty quantification than a parametric counterpart when the tail of error density is polynomially decaying.
In particular, with a location-scale mixture of Gaussians with a conjugate DP prior, the selection consistency and BvM phenomena seem to hold while a location mixture only does not provide  satisfactory results.

Future work will focus on the theoretical development of a location-scale mixtures with heavy-tailed components such as the Student's $t$ distributions.
This will likely entail  new techniques for Bayesian asymptotic, see \cite{chae2017novel} for example.

\appendix

\section{Notation for Proofs}

For given real-valued functions $l$ and $u$, we define the bracket $[l,u]$ as the set of all functions $f$ such that $l\le f \le u$. 
We call a bracket $[l,u]$ an $\epsilon$-bracket if $d(l,u)<\epsilon$ for a given constant $\epsilon>0$ and a (semi-)metric $d$.
For a given class of real-valued functions $\calF$, the bracketing number $N_{[\, ]}(\epsilon,\calF,d)$ is the minimal number of $\epsilon$-brackets which is needed to cover $\calF$.
The covering number $N(\epsilon, \calF, d)$ is the minimal number of $\epsilon$-balls, $\{ g: d(f,g) < \epsilon \}$, which is needed to cover $\calF$.


For given constant $\epsilon>0$, the class of real-valued functions $\calF$ on $\bbR^p \times \bbR$ and the data $D_n =\{(Y_1,x_1),\ldots, (Y_n,x_n) \}$, we denote $N_{[\,]}^n(\epsilon,\calF)$ as the minimal number of partition $\{\calF_1,\ldots,\calF_N \}$ of $\calF$ such that
\bea
\sup_{1\le j \le N} \frac{1}{n} \sum_{i=1}^n \bbE_{\theta_0,\eta_0} \left[\sup_{f,g \in \calF_j} |f(x_i,Y_i) - g(x_i,Y_i)|^2  \right]  &\le& \epsilon^2.
\eea

We define the set of density functions
\bean\label{Hmix_def}
\calH_{\rm mix} &:=& \left\{ \eta(\cdot) = \int \phi_\sigma (\cdot-z) d \widebar{F}(z) : \sigma>0,\,\, F \in \calM[-C'n, C'n] \right\} 
\eean
for the constant $C'>0$ used in \eqref{al_n_n}.
Recall that $\widebar{F} = (F+F^{-})/2$, $dF^{-}(z) = dF(-z)$ and $\phi_\sigma(z) = (\sqrt{2\pi}\sigma)^{-1}\exp \{ -z^2/(2\sigma^2)\}$, for any $z\in \bbR$.

\section{Proofs for Posterior Convergence Rates}

\begin{lemma}\label{lemma:dem_lb}
	Assume that the prior conditions (2)-(8) hold and $\eta_0 $ satisfies \hyperref[D1]{(D1)}-\hyperref[D4]{(D4)}.
	If $\log p \le n^2$, then there exists a constant $C_{\rm lower} >0$ not depending on $(n,p)$ such that the $\bbP_{\theta_0,\eta_0}$-probability of the event
	\bean\label{En_event}
	&& \int_{\Theta \times \calH_{\rm mix}} R_n(\theta,\eta) d\Pi(\theta,\eta) \nonumber\\
	&& \quad \ge \,\, \exp \big[ C_{\rm lower} \{ \log \pi_p(s_0) - s_0 \log p - \lambda \|\theta_0\|_1 - n \tilde{\epsilon}_n^2  \}  \big]   \quad\quad
	\eean
	converges to 1 as $n\to \infty$, where $\tilde{\epsilon}_n = n^{-\beta/(2\beta +\kappa^*)} (\log n)^{t_0}$ and $t_0 = \{\kappa^*(1+ \tau^{-1}+\beta^{-1}) + 1 \}/(2+ \kappa^*\beta^{-1})$.
\end{lemma}

\begin{proof}
	Let $\tilde{\sigma}_{0n}^\beta = \tilde{\epsilon}_n (\log (1/\tilde{\epsilon}_n))^{-1}$, and define 
	\bean\label{tilde_calH_def}
	\widetilde{\calH}_n := \left\{ \eta \in \calH_{\rm mix} : \bbE_{\eta_0}\left( \log \eta_0/\eta  \right) \le A\tilde{\epsilon}_n^2, \,\, \bbE_{\eta_0}\left( \log \eta_0/\eta  \right)^2 \le A \tilde{\epsilon}_n^2 ,  \right.  \\
	\left. \,\, \sigma^{-2} \le \tilde{\sigma}_{0n}^{-2}(1+ \tilde{\sigma}_{0n}^{2\beta})  \right\}, \quad\quad  \nonumber
	\eean
	for some constant $A>0$, and
	\bea
	\widetilde{\Theta}_n &:=& \left\{ \theta \in \Theta : \|\theta-\theta_0\|_1 \le n^{-5}, \,\, S_\theta=S_0  \right\} .
	\eea
	
	Note that
	\bea
	\int_{\Theta \times \calH_{\rm mix}} R_n(\theta,\eta) d\Pi(\theta,\eta) 
	&\ge& \int_{\widetilde{\Theta}_n \times \widetilde{\calH}_n} R_n(\theta,\eta) d\Pi(\theta,\eta) \\
	&=& \int_{\widetilde{\Theta}_n \times \widetilde{\calH}_n} R_n(\theta,\eta) d \widetilde{\Pi}(\theta,\eta) \cdot \Pi(\widetilde{\Theta}_n \times \widetilde{\calH}_n) ,
	\eea
	where $\widetilde{\Pi} = \Pi \,|_{\widetilde{\Theta}_n \times \widetilde{\calH}_n}$ is the restricted and renormalized prior on $\widetilde{\Theta}_n \times \widetilde{\calH}_n$, that is, $\widetilde{\Pi}(\cdot ) = \Pi(\cdot \cap \widetilde{\Theta}_n \times \widetilde{\calH}_n )/\Pi(\widetilde{\Theta}_n \times \widetilde{\calH}_n)$. We will show that
	\bean\label{prior_lb}
	\Pi (\widetilde{\Theta}_n \times \widetilde{\calH}_n) &\ge& \exp \left[ \tilde{C}_1 \left(\log \pi_p(s_0) - s_0 \log p - \lambda \|\theta_0\|_1 - n \tilde{\epsilon}_n^2  \right)  \right]
	\eean
	for some constant $\tilde{C}_1>0$ and all sufficiently large $n$, and 
	\bean\label{Rn_lb}
	&& \bbP_{\theta_0,\eta_0} \left( \int_{\widetilde{\Theta}_n \times \widetilde{\calH}_n} R_n(\theta,\eta) d \widetilde{\Pi}(\theta,\eta) \le \exp (-\tilde{C}_2 n \tilde{\epsilon}_n^2)   \right)   \nonumber \\
	&&\quad \le \,\, \frac{2(A+M^2)}{(\tilde{C}_2 - A- 2M)^2 n \tilde{\epsilon}_n^2} \quad\quad
	\eean
	for some constant $\tilde{C}_2>A+2M$. 
	Then, \eqref{prior_lb} and \eqref{Rn_lb} complete the proof by taking $C_{\rm lower} = (\tilde{C}_1 \vee \tilde{C}_2)$.

	To obtain inequality \eqref{prior_lb}, because $\Pi (\widetilde{\Theta}_n \times \widetilde{\calH}_n) = \Pi_{\Theta} (\widetilde{\Theta}_n) \, \Pi_{\calH}(\widetilde{\calH}_n)$, we derive lower bounds for $\Pi_{\Theta} (\widetilde{\Theta}_n)$ and $\Pi_{\calH}(\widetilde{\calH}_n)$ separately.
	By Lemma \ref{lemma:modif_of_shen2013}, we have
	\bean\label{tilHn_lb}
	\Pi_\calH(\widetilde{\calH}_n) &\ge& \exp(-C_{\calH}  n \tilde{\epsilon}_n^2)
	\eean
	for all sufficiently large $n$ and some constant $C_{\calH} >0$ not depending on $(n,p)$.  
	By the definition of $\Pi_{\Theta}$, we have
	\bea
	\Pi_{\Theta} (\widetilde{\Theta}_n) &=& \int_{\widetilde{\Theta}_n} d\Pi_{\Theta}(\theta) 
	= \pi_p(s_0) \binom{p}{s_0}^{-1} \int_{\widetilde{\Theta}_n} g_{S_0}(\theta_{S_0}) d\theta_{S_0}
	\eea
	and
	\bea
	&& \int_{\widetilde{\Theta}_n} g_{S_0}(\theta_{S_0}) d\theta_{S_0}  \\
	&\ge& e^{-\lambda \|\theta_0\|_1} \int_{\widetilde{\Theta}_n} g_{S_0}(\theta_{S_0} - \theta_{0,S_0}) d\theta_{S_0} \\
	&=& e^{-\lambda \|\theta_0\|_1}  \int_{\widetilde{\Theta}_n}  \left(\frac{\lambda}{2} \right)^{s_0} e^{-\lambda\|\theta_{S_0} - \theta_{0,S_0}\|_1}   d\theta_{S_0}  \\
	&\ge& e^{-\lambda \|\theta_0\|_1} \left(\frac{\lambda}{2} \right)^{s_0} e^{-\lambda n^{-5}} \int_{ \{\theta_{S_0}\in \bbR^{s_0}: \|\theta_{S_0}-\theta_{0,S_0}\|_2 \le (s_0 n^{10})^{-1/2}  \}  }  d\theta_{S_0} \\
	&\ge& e^{-\lambda \|\theta_0\|_1} \left(\frac{\lambda}{2} \right)^{s_0} e^{-\lambda n^{-1/2}} \frac{\pi^{s_0/2}}{\Gamma(s_0/2+1)} (s_0 n^{10})^{-s_0/2} .
	\eea
	Thus, the lower bound for $\Pi_\Theta(\widetilde{\Theta}_n)$ is given by
	\bea
	&&  \Pi_{\Theta} (\widetilde{\Theta}_n) \\
	&\ge&
	\pi_p(s_0) \binom{p}{s_0}^{-1}  e^{-\lambda\|\theta_0\|_1 - \lambda n^{-1/2} } \left( \frac{\lambda \sqrt{\pi}}{2 \sqrt{s_0}n^5} \right)^{s_0} \frac{1}{\Gamma(s_0/2 +1)} \\
	&\ge& \pi_p(s_0)  p^{-s_0} \Gamma(s_0+1) e^{-\lambda\|\theta_0\|_1 - \sqrt{\log p}} \left( \frac{\sqrt{\pi} \sqrt{n} /p}{ 2 \sqrt{s_0} n^5}  \right)^{s_0} \frac{1}{\Gamma(s_0/2 +1)}  \\
	&\ge& \exp\left\{ \log \pi_p(s_0) - s_0\log p - \lambda\|\theta_0\|_1 - \sqrt{\log p}  \right\}  \left( \frac{1}{\sqrt{s_0}n^5 p} \right)^{s_0} \\
	&\ge&  \exp \left\{ \log \pi_p(s_0) - s_0\log p - \lambda\|\theta_0\|_1 - \frac{1}{2}s_0\log p - s_0 \log (\sqrt{s_0} n^5 p)  \right\}   \\
	&\ge& \exp \Big[  8 \big\{\log \pi_p(s_0) - s_0\log p - \lambda\|\theta_0\|_1 \big\} \Big]
	\eea
	for all sufficiently large $n$ because we assume $p \ge n$.
	Thus, 
	\bea
	&& \Pi (\widetilde{\Theta}_n \times \widetilde{\calH}_n)  \\
	&=& \Pi_{\Theta} (\widetilde{\Theta}_n) \Pi_\calH(\widetilde{\calH}_n) \\
	&\ge& \exp \Big[  8 \big\{\log \pi_p(s_0) - s_0\log p - \lambda\|\theta_0\|_1 \big\} \Big]  \exp(-C_{\calH}  n \tilde{\epsilon}_n^2) \\
	&\ge& \exp \Big[  (8 \vee C_{\calH} ) \big\{\log \pi_p(s_0) - s_0\log p - \lambda\|\theta_0\|_1 - n \tilde{\epsilon}_n^2\big\} \Big] ,
	\eea
	which implies \eqref{prior_lb} by taking $\tilde{C}_1 = (8 \vee C_{\calH} ) $.

	By the Jensen's inequality, 
	\bean
	&& \bbP_{\theta_0, \eta_0} \left( \int_{\widetilde{\Theta}_n \times \widetilde{\calH}_n} R_n(\theta,\eta) d \widetilde{\Pi}(\theta,\eta) \le \exp (-\tilde{C}_2 n \tilde{\epsilon}_n^2)   \right) \nonumber \\
	&\le& \bbP_{\theta_0, \eta_0} \left( \int_{\widetilde{\Theta}_n \times \widetilde{\calH}_n} \sum_{i=1}^n \Big\{ \log  \frac{\eta(Y_i- x_i^T \theta)}{\eta_0(Y_i- x_i^T \theta_0)} \Big\} d \widetilde{\Pi}(\theta,\eta) \le -\tilde{C}_2 n \tilde{\epsilon}_n^2    \right) \nonumber\\
	&=& \bbP_{\theta_0, \eta_0} \left(  \sqrt{n}( \widetilde{\bbP}_n - P_0 ) \le - \tilde{C}_2 \sqrt{n}\tilde{\epsilon}_n^2 - \sqrt{n} P_0   \right), \label{PnminusP0}
	\eean
	where $\widetilde{\bbP}_n := n^{-1} \sum_{i=1}^n \int_{\widetilde{\Theta}_n \times \widetilde{\calH}_n} \log [\eta(Y_i- x_i^T \theta)/\eta_0(Y_i- x_i^T \theta_0)] d \widetilde{\Pi}(\theta,\eta)$ and $P_0 := \bbE_{\theta_0,\eta_0}[\widetilde{\bbP}_n]$.
	Note that
	\bea
	-P_0 &\le&  \max_{i} \,\,  \bbE_{\theta_0,\eta_0} \left[ \int_{\widetilde{\Theta}_n \times \widetilde{\calH}_n} \log \frac{\eta_0(Y_i- x_i^T \theta_0)}{\eta(Y_i- x_i^T \theta)} d \widetilde{\Pi}(\theta,\eta) \right] \\
	&=& \max_i \int_{\widetilde{\Theta}_n \times \widetilde{\calH}_n} \Eaa \left( \log \frac{\eta_0(Y_i- x_i^T \theta_0)}{\eta(Y_i- x_i^T \theta)} \right)  d \widetilde{\Pi}(\theta,\eta)  \\
	&=& \max_i \int_{\widetilde{\Theta}_n \times \widetilde{\calH}_n} \Eaa \left( \log \frac{\eta_0(Y_i- x_i^T \theta_0)}{\eta(Y_i- x_i^T \theta_0)} + \log \frac{\eta(Y_i- x_i^T \theta_0)}{\eta(Y_i- x_i^T \theta)} \right) d \widetilde{\Pi}(\theta,\eta) \\
	&\le& A \tilde{\epsilon}_n^2 + \max_i \int_{\widetilde{\Theta}_n \times \widetilde{\calH}_n} \int   \log \frac{\eta(y_i- x_i^T \theta_0)}{\eta(y_i- x_i^T \theta)}  \eta_0(y_i-x_i^T\theta_0) \, dy_i \, \, d \widetilde{\Pi}(\theta,\eta)	
	\eea
	and
	\bean\label{elleta_eta0}
	&&\int   \log \frac{\eta(y- x^T \theta_0)}{\eta(y- x^T \theta)}  \eta_0(y-x^T\theta_0) dy \nonumber \\
	&\le& | x^T(\theta-\theta_0)| \int | \dot{\ell}_\eta(y - x^T \theta_0 + t x^T(\theta_0-\theta) ) | \eta_0(y-x^T\theta_0) dy
	\eean
	for some $t\in [0,1]$ by the mean value theorem. 
	Note that for any $y\in \bbR$,
	\bea
	\sup_{\eta\in \widetilde{\calH}_n}|\dot{\ell}_\eta(y)| &\le& \sup_{\eta\in \widetilde{\calH}_n} \frac{\frac{1}{\sigma^2}\int |y-z| \phi_\sigma(y-z) d\widebar{F}(z) }{\int \phi_\sigma(y-z) d\widebar{F}(z)} \\
	&\le& \sup_{\eta\in \widetilde{\calH}_n} \frac{1}{\sigma^2} (|y| + C'n)   \\
	&\le& \tilde{\sigma}_{0n}^{-2} ( 1+ \tilde{\sigma}_{0n}^{2\beta}) ( |y| + C' n ) \\
	&\le& n^2 ( |y| + n)
	\eea
	for all sufficiently large $n$.
	The above supremum is essentially taken over $(F, \sigma)$ satisfying \eqref{tilde_calH_def} because of definitions of \eqref{Hmix_def} and \eqref{tilde_calH_def}.
	Thus, the right hand side of \eqref{elleta_eta0} is bounded above by
	\bea
	&& M\sqrt{\log p}\, \|\theta - \theta_0\|_1 \int n^2 \left( |y-x^T\theta_0| + |x^T(\theta-\theta_0)| + n  \right) \eta_0(y-x^T\theta_0) dy \\
	&\le& M\sqrt{\log p}\, \|\theta - \theta_0\|_1  n^2 \\
	&& \times \left\{ \int |y-x^T\theta_0| \eta_0(y- x^T\theta_0) dy + M \sqrt{\log p}\|\theta-\theta_0\|_1 + n  \right\} \\
	&\le& 2M \sqrt{\log p} \,\, n^{-2} \,\, \le \,\, 2M n^{-1}
	\eea  
	for all sufficiently large $n$ on $\widetilde{\Theta}_n \times \widetilde{\calH}_n$, because we assume condition \hyperref[D2]{(D2)} and $\log p \le n^2$. 
	Therefore, \eqref{PnminusP0} is bounded above by
	\bea
	&& \bbP_{\theta_0,\eta_0} \left(  \sqrt{n}( \widetilde{\bbP}_n - P_0 ) \le - \tilde{C}_2 \sqrt{n}\tilde{\epsilon}_n^2  + \sqrt{n}(A \tilde{\epsilon}_n^2 + 2M n^{-1} ) \right)  \\
	&\le& \bbP_{\theta_0,\eta_0} \left(  \sqrt{n}( \widetilde{\bbP}_n - P_0 ) \le - (\tilde{C}_2 - A- 2M) \sqrt{n}\tilde{\epsilon}_n^2   \right)  \\
	&\le& \frac{1}{(\tilde{C}_2 - A- 2M)^2 n \tilde{\epsilon}_n^4} \\
	&&  \times \max_i \V_{\theta_0,\eta_0}\left[ \int_{\widetilde{\Theta}_n \times \widetilde{\calH}_n}  \log \eta(Y_i- x_i^T \theta) - \log\eta_0(Y_i- x_i^T \theta_0) d \widetilde{\Pi}(\theta,\eta)  \right]  \\
	&\le& \frac{1}{(\tilde{C}_2 - A- 2M)^2 n \tilde{\epsilon}_n^4}  \\
	&&  \times  \max_i  \Eaa  \left[ \int_{\widetilde{\Theta}_n \times \widetilde{\calH}_n}  \left(\log \eta(Y_i- x_i^T \theta) - \log \eta_0(Y_i- x_i^T \theta_0) \right) d \widetilde{\Pi}(\theta,\eta)  \right]^2 \\
	&\le& \frac{1}{(\tilde{C}_2 - A- 2M)^2 n \tilde{\epsilon}_n^4} \\
	&&  \times \max_i  \Eaa  \left[ \int_{\widetilde{\Theta}_n \times \widetilde{\calH}_n}  \left(\log \eta(Y_i- x_i^T \theta) - \log \eta_0(Y_i- x_i^T \theta_0) \right)^2 d \widetilde{\Pi}(\theta,\eta)  \right]\\
	&=& \frac{1}{(\tilde{C}_2 - A- 2M)^2 n \tilde{\epsilon}_n^4} \max_i   \int_{\widetilde{\Theta}_n \times \widetilde{\calH}_n}  \Eaa\left(\log \frac{ \eta_0(Y_i- x_i^T \theta_0)}{\eta(Y_i- x_i^T \theta)} \right)^2 d \widetilde{\Pi}(\theta,\eta)  
	\eea
	for all sufficiently large $n$ and any constant $\tilde{C}_2>A+2M$. 
	The second and fourth inequalities follow from the Chebyshev's inequality and Jensen's inequality, respectively. 
	Note that 
	\bea
	&& \Eaa\left(\log \frac{ \eta_0(Y_i- x_i^T \theta_0)}{\eta(Y_i- x_i^T \theta)} \right)^2 \\
	&\le& 2 \Eaa\left(\log \frac{ \eta_0(Y_i- x_i^T \theta_0)}{\eta(Y_i- x_i^T \theta_0)} \right)^2 
	+ 2 \Eaa\left(\log \frac{ \eta(Y_i- x_i^T \theta_0)}{\eta(Y_i- x_i^T \theta)} \right)^2 \\
	&\le& 2 A \tilde{\epsilon}_n^2 + 2 \Eaa\left(\log \frac{ \eta(Y_i- x_i^T \theta_0)}{\eta(Y_i- x_i^T \theta)} \right)^2
	\eea
	and 
	\bea
	&& \int\left(\log \frac{ \eta(y- x^T \theta_0)}{\eta(y- x^T \theta)} \right)^2 \eta_0(y-x^T\theta_0) dy  \\
	&\le& \{x^T (\theta-\theta_0) \}^2 \int \big|  \dot{\ell}_\eta (y- x^T\theta_0 + t x^T(\theta_0-\theta))  \big|^2 \eta_0(y- x^T\theta_0) dy \\
	&\le& M^2 \log p \|\theta-\theta_0 \|_1^2 n^4 \Big\{ \int 2y^2 \eta_0(y)dy + 4 M^2\log p  \|\theta-\theta_0\|_1^2 + 4n^2 \Big\} \\
	&\le&   M^2  n^{-1}
	\eea
	for all sufficiently large $n$ on $\widetilde{\Theta}_n \times \widetilde{\calH}_n$.
	Thus, we have
	\bea
	\bbP_{\theta_0, \eta_0} \left( \int_{\widetilde{\Theta}_n \times \widetilde{\calH}_n} R_n(\theta,\eta) d \widetilde{\Pi}(\theta,\eta) \le \exp (-\tilde{C}_2 n \tilde{\epsilon}_n^2)   \right) 
	&\le& \frac{2(A+M^2)}{(\tilde{C}_2 - A- 2M)^2 n \tilde{\epsilon}_n^2}
	\eea
	for all sufficiently large $n$, which completes the proof. \hfill $\blacksquare$
\end{proof}

\begin{lemma}\label{lemma:modif_of_shen2013}
	Under the conditions in Lemma \ref{lemma:dem_lb},
	\bea
	\Pi_\calH(\widetilde{\calH}_n) &\ge& \exp(- C_{\calH} n \tilde{\epsilon}_n^2),
	\eea
	for some constant $C_{\calH} >0$ not depending on $(n,p)$, where $\widetilde{\calH}_n$ and $\tilde{\epsilon}_n$ are defined at \eqref{tilde_calH_def} and Lemma \ref{lemma:dem_lb}, respectively.
\end{lemma}
\begin{proof}
	We closely follow the steps in the proof of Theorem 4 in \cite{shen2013adaptive}.
	We consider the univariate density case while the original proof in \cite{shen2013adaptive} considers $d$-dimensional case.
	
	By Proposition 1 in \cite{shen2013adaptive}, there exist constants $\delta, s_0, a_0, B_0$ and $K_0$ not depending on $(n,p)$ such that 
	\bean\label{Lem2_1}
	d_H (\eta_0, K_\sigma \tilde{h}_{\sigma}  )&\le& K_0 \sigma^{\beta} 
	\eean
	and 
	\bea
	\bbP_{\theta_0, \eta_0} (E_\sigma^c) &\le& B_0 \sigma^{4\beta+2\nu+8}
	\eea
	for any $\sigma \in (0, s_0)$, where $K_\sigma \tilde{h}_{\sigma} = \int \phi_\sigma(x-z) \tilde{h}_\sigma(z) dz$, $\tilde{h}_\sigma$ is a probability density function with support inside $(- a_\sigma, a_\sigma)$, $a_\sigma = a_0 \{\log (1/\sigma) \}^\tau$ and $E_\sigma := \{x \in \bbR : \eta_0(x) \ge \sigma^{(4\beta +2\nu + 8)/\delta}  \} \subset \{ x\in\bbR: |x| \le a_\sigma \}$.
	Fix $b_1 > \{1 \vee 1/(2\beta) \}$ such that $\tilde{\epsilon}_n^{b_1} \{\log (1/\tilde{\epsilon}_n) \}^{5/4} \le \tilde{\epsilon}_n$.
	Let $S_{\tilde{\sigma}_{0n}} = \{\sigma>0 : \sigma^{-2} \in  [\tilde{\sigma}_{0n}^{-2}, \tilde{\sigma}_{0n}^{-2}(1+ \tilde{\sigma}_{0n}^{2\beta}) ]  \}$, where $\tilde{\sigma}_{0n} = \tilde{\epsilon}_n^{1/\beta}\{ \log(1/\tilde{\epsilon}_n)\}^{-1/\beta}$.
	Suppose that $\sigma \in S_{\tilde{\sigma}_{0n}}$.

	By Corollary B1 in \cite{shen2013adaptive}, there exists a probability measure $F_\sigma = \sum_{j=1}^N p_j \delta_{z_j}$ satisfying
	\bean\label{Lem2_2}
	d_H (K_\sigma \tilde{h}_{\sigma}  , \eta_{{F}_\sigma,\sigma} ) &\le& \tilde{A}_1 \tilde{\epsilon}_n^{b_1} \{\log (1/\tilde{\epsilon}_n) \}^{1/4} ,
	\eean
	where $N \le D_0 \sigma^{-1} \{\log (1/\sigma) \}^{1/\tau} \log (1/\tilde{\epsilon}_n)  $, $z_i \in [- a_\sigma,a_\sigma]$ $(i=1,\ldots, n)$  and $\min_{i\neq j} |z_i - z_j| \ge \sigma \tilde{\epsilon}_n^{2b_1}$, 
	for some universal constants $\tilde{A}_1$ and $D_0>0$.
	Note that $N \le D_0 \sigma^{-1} \{\log (1/\sigma) \}^{1/\tau} \log (1/\tilde{\epsilon}_n)  \le D_1 \sigma^{-1} \{\log (1/\tilde{\epsilon}_n)\}^{1+1/\tau}$ for some universal constant $D_1>0$.

	Let $U_j = \{x\in\bbR: |x-z_j| \le \sigma \tilde{\epsilon}_n^{2b_1}/4 \}$ for all $j=1,\ldots, N$.
	Then, one can choose $U_{N+1},\ldots, U_K$ such that 
	(i) $\{U_1,\ldots, U_K\}$ is a partition of $[-a_\sigma, a_\sigma]$, 
	(ii) each $U_j \, (j=N+1,\ldots, K)$ has a diameter at most $\sigma$ and 
	(iii) $K\le D_2 \sigma^{-1} \{\log (1/\tilde{\epsilon}_n)\}^{1+1/\tau} $
	for some universal constant $D_2>0$.
	Furthermore, one can extend this to a partition $\{U_1,\ldots, U_M\}$ of $[-C'n, C'n]$ such that $M \le  D_2' \sigma^{-1} \{\log (1/\tilde{\epsilon}_n)\}^{1+1/\tau} \le D_2' \tilde{\epsilon}_n^{-1/\beta} \{\log (1/\tilde{\epsilon}_n)\}^{1+1/\tau+ 1/\beta} $ and  $D_3 \sigma \tilde{\epsilon}_n^{2b_1} \le \alpha(U_j) \le 1$ for all $j=1,\ldots, M$ and for some universal constants $D_2'$ and $D_3>0$ because of the continuity and positivity of $\alpha$.

	Let $p_j=0$ for all $j=N+1,\ldots, M$.
	Define $\calP_{\tilde{\sigma}_{0n}}$ as the set of probability measures $F$ on $[-C'n, C'n]$ such that 
	\bea
	\sum_{j=1}^M | F(U_j) - p_j |  \le 2 \tilde{\epsilon}_n^{2b_1} \,\,\text{ and }\,\, 	\min_{1\le j \le M} F(U_j) \ge \frac{1}{2} \tilde{\epsilon}_n^{4b_1}  .
	\eea
	Then, we have 
	$\tilde{\epsilon}_n^{2b_1} M \le D_2' \tilde{\epsilon}_n^{2b_1 - 1/\beta} \{ \log(1/\tilde{\epsilon}_n ) \}^{1+ 1/\tau + 1/\beta}  \le 1$ and $\min_{1\le j \le M}$ $\alpha(U_j) \ge D_3 \sigma \tilde{\epsilon}_n^{2b_1} \ge D_3 \tilde{\epsilon}_n^{4b_1}$ for all large $n$.
	By Lemma 10 in \cite{ghosal2007posterior}, 
	\bea
	\pi (\calP_{\tilde{\sigma}_{0n}} )  &\ge& C_1 \exp \big\{ - c_1 M \log (1/ \tilde{\epsilon}_n)  \big\} \\
	&\ge& C_1 \exp \big[  - c_1 D_2' \tilde{\epsilon}_n^{-1/\beta} \{\log (1/\tilde{\epsilon}_n)\}^{2+1/\tau +1/\beta}     \big]
	\eea
	for some universal constants $C_1$ and $c_1>0$.
	In fact, $C_1=\Gamma (\alpha([-C'n, C'n ]))$, but it can be replaced with a universal constant not depending on $n$ by considering $\Gamma (\alpha([-C'n, C'n ])) \ge \Gamma (\alpha([-C', C' ])) = : C_1$.
	Also note that, by \eqref{G3}, 
	\bea
	\pi (S_{\tilde{\sigma}_{0n}}) 
	&\ge& a_6 \tilde{\sigma}_{0n}^{-2a_4} \tilde{\sigma}_{0n}^{2\beta a_5} \exp (-C''  \tilde{\sigma}_{0n}^{2\beta} )    \\
	&\ge& D_4 \exp \big[  - D_5 \tilde{\epsilon}_n^{-\kappa/ \beta}  \{\log(1/\tilde{\epsilon}_n) \}^{\kappa/\beta}   \big]
	\eea
	for some universal constant $D_4>0$ and some constant $D_5>0$ depending only on $C''>0$ in \eqref{G3}.
	Therefore, by Lemma B1 in \cite{shen2013adaptive} with $V_j = U_j$ for $j=1,\ldots, N$ and $V_0 = \cup_{j=N+1}^M U_j$, we have
	\bean\label{Lem2_3}
	d_H(  \eta_{{F}_\sigma,\sigma} , \eta_{{F},\sigma}  )   &\le& \tilde{A}_2 \tilde{\epsilon}_n^{b_1}
	\eean
	for any $F\in \calP_{\tilde{\sigma}_{0n}}$, $\sigma \in S_{\tilde{\sigma}_{0n}}$ and some constant $\tilde{A}_2>0$ not depending on $(n,p)$.	
	Thus, by \eqref{Lem2_1}--\eqref{Lem2_3}, 
	\bea
	d_H( \eta_0,  \eta_{{F},\sigma})  &\le& \tilde{A}_3 \tilde{\sigma}_{0n}^\beta
	\eea
	for any $F\in \calP_{\tilde{\sigma}_{0n}}$, $\sigma \in S_{\tilde{\sigma}_{0n}}$ and some constant $\tilde{A}_3>0$ not depending on $(n,p)$.
	Note that $d_H^2(\eta_0, \eta_{{F},\sigma}) = d_H^2(\eta_0, \eta_{{F^-{}},\sigma})$ due to condition \hyperref[D4]{(D4)} and
	\bea
	d_H^2(\eta_0, \eta_{\bar{F},\sigma}) &=& \int \big( \sqrt{\eta_0} - \sqrt{\eta_{\bar{F},\sigma}}  \big)^2  d\mu  \\
	&=& \int \big( \sqrt{\eta_0} - \sqrt{(\eta_{F,\sigma}+ \eta_{F^{-},\sigma})/2}  \big)^2  d\mu  \\
	&=& \int_{|\sqrt{\eta_0} - \sqrt{\eta_{F,\sigma}} | > |\sqrt{\eta_0} - \sqrt{\eta_{F^{-},\sigma}} |} \big( \sqrt{\eta_0} - \sqrt{(\eta_{F,\sigma}+ \eta_{F^{-},\sigma})/2}  \big)^2  d\mu  \\
	&+& \int_{|\sqrt{\eta_0} - \sqrt{\eta_{F,\sigma}} | \le |\sqrt{\eta_0} - \sqrt{\eta_{F^{-},\sigma}} |} \big( \sqrt{\eta_0} - \sqrt{(\eta_{F,\sigma}+ \eta_{F^{-},\sigma})/2}  \big)^2  d\mu  \\
	&\le& \int_{|\sqrt{\eta_0} - \sqrt{\eta_{F,\sigma}} | > |\sqrt{\eta_0} - \sqrt{\eta_{F^{-},\sigma}} |} \big( \sqrt{\eta_0} - \sqrt{\eta_{F,\sigma}}  \big)^2  d\mu  \\
	&+& \int_{|\sqrt{\eta_0} - \sqrt{\eta_{F,\sigma}} | \le |\sqrt{\eta_0} - \sqrt{\eta_{F^{-},\sigma}} |} \big( \sqrt{\eta_0} - \sqrt{\eta_{F^{-},\sigma}}  \big)^2  d\mu \\
	&\le& d_H^2(\eta_0, \eta_{{F},\sigma}) + d_H^2(\eta_0, \eta_{{F^-{}},\sigma}) = 2d_H^2(\eta_0, \eta_{{F},\sigma}).
	\eea
	Therefore, we have 
	\bea
	d_H( \eta_0,  \eta_{\bar{F},\sigma})  &\le& \sqrt{2} \tilde{A}_3 \tilde{\sigma}_{0n}^\beta
	\eea
	for any $F\in \calP_{\tilde{\sigma}_{0n}}$ and $\sigma \in S_{\tilde{\sigma}_{0n}}$.

	Note that for any $F\in \calP_{\tilde{\sigma}_{0n}}$, $\sigma \in S_{\tilde{\sigma}_{0n}}$ and $x\in [-a_\sigma,a_\sigma]$,
	\bea
	\frac{\eta_{\bar{F},\sigma}(x)}{\eta_0(x)} 
	&\ge& \{ \sup_{t\in \bbR} \eta_0(t) \}^{-1}  (2\pi \tilde{\sigma}_{0n}^2)^{-1/2} \int  \exp \Big\{ - \frac{(x-z)^2}{2 \tilde{\sigma}_{0n}^2}  \Big\}  d\bar{F}(z)  \\
	&\ge& K_1 \tilde{\sigma}_{0n}^{-1} \{F( U_{J(x)} )  \wedge F( U_{J(-x)} )\}  \,\, \ge \,\, \frac{K_1}{2} \tilde{\sigma}_{0n}^{-1} \tilde{\epsilon}^{4b_1}
	\eea
	for some universal constant $K_1>0$, where $J(x)$ is the index $j\in\{1,\ldots, M\}$ for which $x \in U_j$.
	On the other hand, for any $F\in \calP_{\tilde{\sigma}_{0n}}$, $\sigma \in S_{\tilde{\sigma}_{0n}}$ and $x\notin [-a_\sigma,a_\sigma]$, 
	\bea
	\frac{\eta_{\bar{F},\sigma}(x)}{\eta_0(x)} 
	&\ge&  K_1 \tilde{\sigma}_{0n}^{-1}  \int_{|z| \le a_\sigma} \exp \Big\{ - \frac{(x-z)^2}{2 \tilde{\sigma}_{0n}^2}  \Big\}  d\bar{F}(z)  \\
	&\ge& K_1 \tilde{\sigma}_{0n}^{-1}  \exp \Big( - \frac{2x^2}{\tilde{\sigma}_{0n}^2}  \Big) F \big( Z:  |Z| \le a_\sigma \big)  \\
	&\ge& K_1 \tilde{\sigma}_{0n}^{-1}   \exp \Big( - \frac{2x^2}{\tilde{\sigma}_{0n}^2}  \Big) (1 -2 \tilde{\epsilon}_n^{2b_1} ) \\
	&\ge& \frac{K_1}{2} \tilde{\sigma}_{0n}^{-1}   \exp \Big( - \frac{2x^2}{\tilde{\sigma}_{0n}^2}  \Big) 
	\eea
	for all large $n$.
	The third inequality holds because $F\in \calP_{\tilde{\sigma}_{0n}}$.
	Define $\vartheta = \tilde{\sigma}_{0n}^{-1}  \tilde{\epsilon}_n^{4b_1} K_1 /2$, then $\log (1/\vartheta)  \le K_2 \log (1/\tilde{\epsilon}_n)$ for some constant $K_2>0$ depending only on $b_1$.
	Then, for any $F\in \calP_{\tilde{\sigma}_{0n}}$ and $\sigma \in S_{\tilde{\sigma}_{0n}}$, 
	\bea
	&&\bbE_{\eta_0} \Big[ \Big\{ \log \big(\frac{\eta_0}{\eta_{\bar{F},\sigma} } \big)  \Big\}^2 I \Big( \frac{\eta_{\bar{F},\sigma} }{\eta_0} \le \vartheta \Big)  \Big] \\
	&\le&  \int_{|x|> a_{\tilde{\sigma}_{0n}} } \Big\{ \log \big(\frac{\eta_0(x)}{\eta_{\bar{F},\sigma}(x) } \big)  \Big\}^2  \eta_0(x)  dx \\
	&\le&  \int_{|x|> a_{\tilde{\sigma}_{0n}} } \Big[  \log \Big\{ \frac{2\tilde{\sigma}_{0n} }{K_1}  \exp \Big(  \frac{2x^2}{\tilde{\sigma}_{0n}^2}  \Big) \Big\}  \Big]^2 \eta_0(x) dx \\
	&\le& \frac{K_3}{\tilde{\sigma}_{0n}^4}  \int_{|x|> a_{\tilde{\sigma}_{0n}} } x^4 \eta_0(x)  dx  \\
	&\le&  \frac{K_3}{\tilde{\sigma}_{0n}^4}  \Big( \bbE_{\eta_0} X^8 \Big)^{1/2} \bbP_{\eta_0} ( E_{\tilde{\sigma}_{0n}}^c )  \\
	&\le& K_4 \tilde{\sigma}_{0n}^{2\beta+ \nu}
	\eea
	for some constants $K_3$ and $K_4>0$ not depending on $(n,p)$ by construction of $E_{\tilde{\sigma}_{0n}}$.
	Since $\vartheta < e^{-1}$, it implies that 
	\bea
	\bbE_{\eta_0} \Big\{ \log \big(\frac{\eta_0}{\eta_{\bar{F},\sigma} } \big)   I \Big( \frac{\eta_{\bar{F},\sigma} }{\eta_0} \le \vartheta \Big)  \Big\}  
	&\le& K_4 \tilde{\sigma}_{0n}^{2\beta+ \nu}  .
	\eea

	Therefore, by Lemma B2 in \cite{shen2013adaptive}, for any $F\in \calP_{\tilde{\sigma}_{0n}}$ and $\sigma \in S_{\tilde{\sigma}_{0n}}$, 
	\bea
	&& \bbE_{\eta_0}\Big\{  \log \Big(\frac{\eta_0}{\eta_{\bar{F},\sigma}}  \Big)   \Big\} \\
	&\le& d_H^2( \eta_0,  \eta_{\bar{F},\sigma}) \big\{ 1+ 2 \log (1/\vartheta)  \big\}  + 2 \bbE_{\eta_0} \Big\{ \log \big(\frac{\eta_0}{\eta_{\bar{F},\sigma} } \big)   I \Big( \frac{\eta_{\bar{F},\sigma} }{\eta_0} \le \vartheta \Big)  \Big\}    \\
	&\le& 2\tilde{A}_3^2 \tilde{\sigma}_{0n}^{2\beta} \{ 1+ 2K_2 \log (1/\tilde{\epsilon}_n) \}  + 2 K_4 \tilde{\sigma}_{0n}^{2\beta+\nu} \\
	&\le& 2\tilde{A}_3^2 (12 + 2K_2^2) \tilde{\epsilon}_n^2 
	\eea
	and 
	\bea
	&&\bbE_{\eta_0}\Big\{  \log \Big(\frac{\eta_0}{\eta_{\bar{F},\sigma}}  \Big)   \Big\}^2  \\
	&\le&  d_H^2( \eta_0,  \eta_{\bar{F},\sigma}) \big[ 12 + 2 \big\{\log (1/\vartheta)  \big\}^2 \big] + 8  \bbE_{\eta_0} \Big[ \Big\{ \log \big(\frac{\eta_0}{\eta_{\bar{F},\sigma} } \big)  \Big\}^2 I \Big( \frac{\eta_{\bar{F},\sigma} }{\eta_0} \le \vartheta \Big)  \Big] \\
	&\le& 2\tilde{A}_3^2 \tilde{\sigma}_{0n}^{2\beta}  \big[ 12 + 2 K_2^2 \{ \log(1/\tilde{\epsilon}_n) \}  \big]  + 8 K_4\tilde{\sigma}_{0n}^{2\beta+\nu} \\ 
	&\le& 2\tilde{A}_3^2 (12 + K_2^2 )  \tilde{\epsilon}_n^2 .
	\eea
	Thus, by taking $A = 2 \tilde{A}_3^2(12 + 2K_2^2)$ in $\widetilde{\calH}_n$ defined at \eqref{tilde_calH_def}, we have
	\bea
	&&\Pi_{\calH} (\widetilde{\calH}_n)  \\
	&\ge& \Pi_{\calH} \big( (F,\sigma):  F\in \calP_{\tilde{\sigma}_{0n}},  \sigma \in S_{\tilde{\sigma}_{0n}} \big)  \\
	&\ge&  C_1 D_4 \exp \Big[  - c_1 D_2' \tilde{\epsilon}_n^{-1/\beta} \{\log(1/\tilde{\epsilon}_n) \}^{2+1/\tau+1/\beta}  - D_5\tilde{\epsilon}_n^{-\kappa/\beta}  \{\log(1/\tilde{\epsilon}_n) \}^{\kappa/\beta}   \Big] \\
	&\ge& C_1 D_4 \exp \Big[  -  ( c_1 D_2' \vee D_5)  \tilde{\epsilon}_n^{\kappa^*/\beta} \{\log(1/\tilde{\epsilon}_n) \}^{2+1/\tau+\kappa^*/\beta}    \Big]  \\
	&\ge& \exp \big\{ - ( c_1 D_2' \vee D_5)  n \tilde{\epsilon}_n^2   \big\}
	\eea
	for all large $n$ and some constants $c_1, D_2'$ and $D_5>0$ not depending on $(n,p)$.
	By taking $C_{\calH} =( c_1 D_2' \vee D_5)$, it completes the proof. \hfill $\blacksquare$
\end{proof}

\bigskip

\begin{proof}[Proof of Theorem \ref{thm:dimupper}]
	Suppose $\lambda \|\theta_0\|_1 \le C_\lambda s_0 \log p$ for some constant $C_\lambda >0$.
	Let $B:= \{(\theta,\eta): s_\theta \ge R \}$ for some $R > s_0$ and $E_n$ be the event \eqref{En_event}, then we have
	\bea
	&& \Eaa \Pi(B \mid D_n) \\
	&\le& \Eaa\left[\Pi(B\mid D_n) I_{E_n} \right] + \bbP_{\theta_0, \eta_0}(E_n^c) \\
	&\le& \Eaa\left[ \frac{\int_B R_n(\theta,\eta) d\Pi(\theta,\eta) }{\int R_n(\theta,\eta) d\Pi(\theta,\eta) }  I_{E_n} \right] + o(1) \\
	&\le& \exp \left[C_{\rm lower} \big\{  -\log \pi_p(s_0) + s_0 \log p + \lambda\|\theta_0\|_1 +  n \tilde{\epsilon}_n^2 \big\} \right] \cdot \Pi(B) + o(1) \\
	&\le& \exp \left[C_{\rm lower} \big\{   (A_3+1) s_0 \log p + s_0 \log p + C_\lambda s_0 \log p +  n^{\frac{\kappa^*}{2\beta+\kappa^*}}  (\log n)^{2t_0} \big\} \right] \\
	&& \times \Pi(B) + o(1) \\
	&\le& \exp \left[C_{\rm lower}(A_3+2 +C_\lambda) \{s_0 \vee  n^{\frac{\kappa^*}{2\beta+\kappa^*}}  (\log n)^{2t_0-1}\}\log p  \right] \cdot \Pi(B) + o(1) 
	\eea
	by Lemma \ref{lemma:dem_lb} and condition \eqref{prior_p}.
	Note that
	\bea
	\Pi(B) &\le& \sum_{s=R}^p \pi_p(s_0) \left(\frac{A_2}{p^{A_4}} \right)^{s-s_0} \\
	&\le&  2\pi_p(s_0) \left(\frac{A_2}{p^{A_4}} \right)^{R-s_0}  \\
	&\le& \exp \Big\{  - (R - s_0) \frac{A_4}{2} \log p   \Big\}
	\eea
	by condition \eqref{prior_p}. Thus, we have
	\bea
	&& \Eaa \Pi(B \mid D_n) \\
	&\le& \exp \Big[ - \Big\{  (K_{\rm dim} -1) \frac{A_4}{2} - C_{\rm lower} (A_3+2+C_\lambda )  \Big\}  \\
	&& \times \{s_0 \vee  n^{\frac{\kappa^*}{2\beta+\kappa^*}}  (\log n)^{2t_0-1}\} \log p \Big]  + o(1)  \\
	&=& o(1)
	\eea
	by taking $R = K_{\rm dim} \{ s_0 \vee  n^{\frac{\kappa^*}{2\beta+\kappa^*}}  (\log n)^{2t_0-1} \}$ for some large constant $K_{\rm dim}> 1 + 2 A_4^{-1} C_{\rm lower} (A_3+2+C_\lambda ) $,  which completes the proof. \hfill $\blacksquare$
\end{proof}

\begin{proof}[Proof of Theorem \ref{thm:convrate_mH}]
	Define 
	\bea
	\Theta_n &:=& \left\{ \theta\in \Theta : \|\theta-\theta_0\|_1 \le p^2 (p+\sqrt{n}) + \|\theta_0\|_1, \,\, s_\theta \le s_n/2  \right\}
	\eea
	and for positive constants $C_1$ and $C_2$, which will be described below, define
	\bean\label{Hn}
	\begin{split}
		\calH_n &:= \Bigg\{ \eta(\cdot) = \int \phi_\sigma(\cdot -z) d\widebar{F}(z) \text{ with } F = \sum_{h=1}^\infty \pi_h \delta_{z_h} : \\ 
		& \quad\quad   z_h \in [-a_n,a_n], h\le H_n ;  \sum_{h> H_n}\pi_h < \epsilon_n ; \sigma^2 \in [\sigma_{0n}^2, \sigma_{0n}^2(1+\epsilon_n^2 )^{M_n})  \Bigg\},
	\end{split}
	\eean
	where $a_n^{a_1}= \sigma_{0n}^{-2a_2}= M_n =n, \epsilon_n^2 = C_1 s_n\log p/n$ and $H_n= \lfloor C_2 s_n \log p/\log n \rfloor$. We first prove that
	\bean
	\Eaa \Pi(\theta\in \Theta_n^c \mid D_n) &=& o(1) \quad\text{ and} \label{thetaconc} \\
	\Eaa \Pi(\eta \in \calH_n^c \mid D_n) &=& o(1). \label{etaconc}
	\eean

	Suppose $\lambda \|\theta_0\|_1 \le C_\lambda s_0 \log p$ for some constant $C_\lambda >0$.
	By Lemma \ref{lemma:dem_lb} and Theorem \ref{thm:dimupper},
	\bea
	&& \Eaa \Pi(\theta\in \Theta_n^c \mid D_n) \\
	&\le& \Eaa \Pi \big( \|\theta-\theta_0\|_1 > p^2(p+\sqrt{n}) + \|\theta_0\|_1 \mid D_n \big) + \Eaa\Pi(s_\theta > s_n/2 \mid D_n) \\
	&\le& \Eaa \left[ \Pi\left( \|\theta-\theta_0\|_1 > p^2(p+\sqrt{n}) + \|\theta_0\|_1 \mid D_n \right) I_{E_n} \right] + o(1) \\
	&\le& \Pi_{\Theta}\left( \|\theta-\theta_0\|_1> p^2(p+\sqrt{n})+\|\theta_0\|_1 \right)  \cdot \exp\Big\{  \frac{C_{\rm lower}(A_3+2+C_\lambda) }{2 K_{\rm dim}} s_n \log p \Big\} \\
	&& + o(1),
	\eea
	where $E_n$ is the event \eqref{En_event}. Note that
	\bea
	&&\Pi_{\Theta}\left( \|\theta-\theta_0\|_1> p^2(p+\sqrt{n})+\|\theta_0\|_1 \right) \\
	&\le& \Pi_\Theta \left(\|\theta\|_1 > p^2(p+\sqrt{n})  \right) \\
	&=& \sum_{s=1}^p \Pi_\Theta \left( \|\theta\|_1 > p^2(p+\sqrt{n}) \mid s_\theta=s  \right) \pi_p(s) \\
	&\le& \sum_{s=1}^p s \cdot \max_{1\le h\le s} \Pi_\Theta \left( |{\theta}_h| >p(p+\sqrt{n}) \right) \cdot p^{-A_4 s} A_2^s  \\
	&\le& p \cdot \exp \left( -\lambda p (p+\sqrt{n}) \right) \\
	&\le& \exp\left\{- \frac{1}{2} (n+p) \right\}  
	\eea
	because $\lambda p \ge \sqrt{n}$. 
	Thus, \eqref{thetaconc} holds due to condition $s_n\log p = o(n)$.	
	On the other hand, by Proposition 2 of \cite{shen2013adaptive}, 
	\bea
	&&\Pi_\calH (\calH_n^c) \\
	&\lesssim&  H_n \exp( - C'' a_n^{a_1} ) +  \left\{ \frac{e \alpha(\bbR) }{H_n} \log \frac{1}{\epsilon_n} \right\}^{H_n} + \exp(- C'' \sigma_{0n}^{-2a_2}) \\
	&& + \sigma_{0n}^{-2a_3}\left(1+ \epsilon_n^2 \right)^{-2M_na_3}  \\
	&\le&  C_2 \frac{ s_n \log p}{\log n} \exp (-C''n) + \exp (-C_2 s_n \log p) + \exp(-C'' n) \\
	&& + \exp(-C_1 a_3 s_n \log p)  \\
	&\le& \exp \Big\{  - \frac{(C_1 a_3 \wedge C_2 )}{2}s_n \log p \Big\}  .
	\eea
	Then, we have
	\bea
	&&\Eaa \Pi(\eta \in \calH_n^c \mid D_n) \\
	&\le& \Eaa \left[ \Pi\left(\eta \in \calH_n^c \mid D_n \right) I_{E_n} \right] + o(1) \\
	&\lesssim& \Pi_{\calH}(\calH_n^c) \cdot \exp\left\{ \frac{C_{\rm lower}(A_3+2+C_\lambda)}{2K_{\rm dim} } {s}_n \log p \right\} + o(1) \\
	&\le& \exp \left[ - \frac{1}{2}\Big\{ (C_1 a_3 \wedge C_2) - \frac{C_{\rm lower}(A_3+2+C_\lambda)}{2K_{\rm dim}}   \Big\}s_n \log p \right]  \,\, = \,\, o(1)
	\eea
	for some large constants $C_1$ and $C_2 > 0$.
	Thus, we have proved \eqref{thetaconc} and \eqref{etaconc}.

	By Lemma 2 and Lemma 9 of \cite{ghosal2007convergence}, if for some nonincreasing function $\epsilon \mapsto N(\epsilon)$ and some $\epsilon_n' \ge 0$,
	\bea
	N\left( \frac{\epsilon}{36}, \Theta_n\times \calH_n , d_n \right) &\le& N(\epsilon),
	\eea
	for all $\epsilon>\epsilon_n'$, then there exists test functions $\phi_n$ such that
	\bean
	\begin{split}\label{Lem2and9_GS}
		\bbP_{\theta_0,\eta_0}\phi_n &\lesssim \exp \left( -\frac{n}{2}\epsilon_n^2 + \log N(\epsilon_n) \right) \quad \text{ and} \\
		\sup_{\substack{(\theta,\eta)\in \Theta_n\times \calH_n \\ d_n((\theta,\eta),(\theta_0,\eta_0 )) > \epsilon_n } } \bbP_{\theta,\eta} (1-\phi_n) &\lesssim \exp \left( - \frac{n}{2} \epsilon_n^2 \right)
	\end{split}
	\eean
	for all $\epsilon_n > \epsilon_n'$. For any $(\theta^i, \eta_i)\in \Theta_n\times \calH_n$, $i=1,2$,
	\bea
	&&d_H^2 ( \eta_1(\cdot - x^T\theta^1) , \eta_2(\cdot-x^T\theta^2) )  \\
	&=& \int \left( \sqrt{\eta_1(y - x^T\theta^1)} - \sqrt{\eta_2(y - x^T\theta^2)}  \right)^2 dy \\
	&\le&  2\int \left\{ \sqrt{\eta_1(y - x^T\theta^1)} - \sqrt{\eta_1(y - x^T\theta^2)}  \right\}^2 dy \\
	&+& 2\int \left\{ \sqrt{\eta_1(y - x^T\theta^2)} - \sqrt{\eta_2(y - x^T\theta^2)}  \right\}^2 dy \\
	&\le& 2 \left\{ | x^T(\theta^1- \theta^2)|^2 \int \left( \int_0^1 \frac{\dot{\eta_1}(y+t d_{12})}{\sqrt{\eta_1(y+ td_{12})}} dt  \right)^2 dy +  d_H^2(\eta_1,\eta_2)  \right\} \\
	&\le& 2 \left\{   M^2 \log p \, \|\theta^1-\theta^2\|_1^2 \int_0^1\int  \left( \frac{\dot{\eta_1}(y+t d_{12})}{\eta_1(y+ td_{12})} \right)^2 \eta_1(y+t d_{12}) dy dt   \right. \\
	&& \quad\quad+ d_H^2(\eta_1,\eta_2) \Bigg\} \\ 
	&\le& 2 \Big\{  M^2 \log p \, \|\theta^1-\theta^2\|_1^2 \cdot n^{1/a_2} + d_H^2(\eta_1,\eta_2) \Big\} ,
	\eea 
	where $d_{12} := x^T(\theta^1-\theta^2)$. The last inequality holds because 
	\bea
	\left( \frac{\dot{\eta}(y)}{\eta(y)} \right)^2 \eta(y)  &=&    \frac{ \{ \dot{\eta}(y) \}^2 }{\eta(y)}  \\
	&\le& \frac{ \Big\{  \int \frac{|y-z|}{\sigma^2} \phi_\sigma(y-z) d\widebar{F}(z)  \Big\}^2  }{\eta(y)}  \\
	&\le& \int \left( \frac{y-z}{\sigma^2} \right)^2 \phi_\sigma(y-z) d \widebar{F}(z) 
	\eea
	by H\"{o}lder's inequality and  
	\bea
	\int \left( \frac{\dot{\eta}(y)}{\eta(y)} \right)^2 \eta(y) dy
	&\le& \int \int \left( \frac{y-z}{\sigma^2} \right)^2 \phi_\sigma (y -z) d \widebar{F}(z) dy \\
	&=& \frac{1}{\sigma^2}  \int \int \left( \frac{y-z}{\sigma} \right)^2 \phi_\sigma (y -z) dy  d \widebar{F}(z)  \\
	&=& \frac{1}{\sigma^2} \,\, \le \,\, \sigma_{0n}^{-2}=\,\, n^{1/a_2}.
	\eea
	
	Thus, we have
	\bea
	&& \log N\left( \frac{\epsilon}{36}, \Theta_n\times \calH_n , d_n \right) \\
	&\lesssim& 
	\log N \left( \frac{\epsilon}{72M  n^{1/(2a_2)}\sqrt{\log p} }, \Theta_n , \|\cdot\|_1 \right) + \log N \left( \frac{\epsilon}{72}, \calH_n, d_H \right) \\
	&\le& \log  \left(  \sum_{j=0}^{s_n/2} \binom{p}{j}  \Big[  \frac{p^2(p+\sqrt{n}) + \|\theta_0\|_1 }{\epsilon} 72 M \sqrt{n^{1/a_2} \log p}  \Big]^j  \right)   \\
	&& \,\,+ K \Big\{ H_n \log \Big(\frac{a_n}{\sigma_{0n} \epsilon } \Big) + H_n \log \Big( \frac{1}{\epsilon}\Big) + \log M_n    \Big\}  \\
	&\le& \log \left( \sum_{j=0}^{s_n/2} \Big[  \frac{p \{p^2(p+\sqrt{n}) + \|\theta_0\|_1\} }{\epsilon} 72 M \sqrt{n^{1/a_2} \log p}    \Big]^j \right) \\
	&& \,\,+ K \Big\{ H_n \log \Big(\frac{a_n}{\sigma_{0n} \epsilon } \Big) + H_n \log \Big( \frac{1}{\epsilon}\Big) + \log M_n    \Big\}  \\  
	&\le& s_n \log \Big( \frac{p^4}{\epsilon}\Big)  + K \Big\{ \frac{C_2 s_n \log p}{\log n} \log \Big(\frac{n^{1/a_1 + 1/(2a_2)}}{\epsilon} \Big) + \frac{C_2 s_n \log p}{ \log n} \log \Big(\frac{1}{\epsilon} \Big)   \\
	&& \quad\quad+\log n\Big\} \\
	&=:& \log N(\epsilon) 
	\eea
	for some universal constant $K>0$ by Proposition 2 of \cite{shen2013adaptive}. 
	Note that in the last term, we do not have the term $M_n \epsilon_n^2$ while Proposition 2 in \cite{shen2013adaptive} includes this term, because they considered $d$-dimensional densities. 
	It is easy to see that from their proof, the term $M_n \epsilon_n^2$ can be omitted if we focus on univariate $(d=1)$ densities. 
	Note that 
	\bea
	\log N(\epsilon_n) 
	&\le& 5 s_n \log p + K C_2 \big\{ 2+ a_1^{-1} + (2a_2)^{-1} \big\} s_n \log p \\
	&=& \big[ 5 + K C_2 \big\{ 2 + a_1^{-1} + (2a_2)^{-1}  \big\}  \big] s_n \log p
	\eea
	Thus, by \eqref{Lem2and9_GS}, there exist test functions $\phi_n$ such that
	\bea
	\bbP_{\theta_0,\eta_0}\phi_n  &\lesssim& \exp \left( -\frac{C_1}{2}s_n \log p +  \big[5 + K C_2 \big\{ 2 + a_1^{-1} + (2a_2)^{-1}  \big\} \big] s_n \log p \right)
	\eea
	and 
	\bea
	\sup_{\substack{(\theta,\eta)\in \Theta_n\times \calH_n \\ d_n((\theta,\eta),(\theta_0,\eta_0 )) > \epsilon_n } } \bbP_{\theta,\eta} (1-\phi_n) &\lesssim& \exp \left( - \frac{C_1}{2} s_n \log p  \right).
	\eea
	Therefore, by Lemma \ref{lemma:dem_lb}, for a large constant $C_1 >0$ such that $C_1 > 10 + 2 KC_2 \big\{ 2 + a_1^{-1} + (2a_2)^{-1}  \big\}$ and $C_1 > C_{\rm lower}(A_3+2+C_\lambda)/K_{\rm dim}$,
	\bea
	&&  \bbE_{\theta_0,\eta_0} \Pi \left( d_n( (\theta,\eta), (\theta_0,\eta_0) ) > \epsilon_n  \,\,\bigg|\,\, D_n\right) \\
	&\le&  \Eaa \Pi \left((\theta,\eta)\in \Theta_n\times \calH_n: d_n( (\theta,\eta), (\theta_0,\eta_0) ) > \epsilon_n  \,\,\bigg|\,\, D_n \right) + o(1) \\
	&\le&  \Eaa \left[  \Pi \left((\theta,\eta)\in \Theta_n\times \calH_n: d_n( (\theta,\eta), (\theta_0,\eta_0) ) > \epsilon_n  \,\,\bigg|\,\, D_n \right) (1-\phi_n) \right]\\
	&& + o(1) \\
	&\lesssim& \sup_{\substack{(\theta,\eta)\in \Theta_n\times \calH_n \\ d_n((\theta,\eta),(\theta_0,\eta_0 )) > \epsilon_n } } \bbP_{\theta,\eta} (1-\phi_n) \cdot \exp\Big\{ \frac{C_{\rm lower}(A_3+2+C_\lambda)}{2K_{\rm dim}} s_n \log p \Big\} + o(1) \\
	&=& o(1) . 
	\eea
	It completes the proof by taking $K_{\rm Hel} = \sqrt{C_1} > \sqrt{C_{\rm lower}(A_3+2+C_\lambda)/K_{\rm dim}} \vee \sqrt{10 + 2 KC_2 \big\{ 2 + a_1^{-1} + (2a_2)^{-1}  \big\}}$.  \hfill $\blacksquare$
\end{proof}

\begin{proof}[Proof of Corollary \ref{cor:convrate_H}]
	Let $(T_z(\eta))(x) = \eta(x+ z)$.
	Note that for any $\eta_0$ satisfying  \hyperref[D1]{(D1)}-\hyperref[D4]{(D4)} and $\eta\in \calH_{\rm mix}$,
	\bea
	\inf_{z\in \bbR} d_H ( \eta , T_z(\eta_0)) 
	&\le& d_H(\eta, T_{x^T(\theta-\theta_0)}(\eta_0) )\\
	&=& \left[ \int \big(\sqrt{\eta(y)} - \sqrt{\eta_0(y+x^T(\theta-\theta_0))} \,\, \big)^2 dy \right]^{1/2} \\
	&=& \left[ \int \big(\sqrt{\eta(y - x^T\theta)} - \sqrt{\eta_0(y- x^T\theta_0)} \,\, \big)^2 dy \right]^{1/2} \\
	&=& d_H (\eta(\cdot-x^T\theta), \eta_0(\cdot-x^T\theta_0) ),
	\eea
	thus
	\bea
	\inf_{z \in \bbR} d_H(\eta, T_z(\eta_0)) 
	&\le& \left[\frac{1}{n}\sum_{i=1}^n d_H^2( \eta(\cdot-x_i^T\theta), \eta_0(\cdot-x_i^T\theta_0) )  \right]^{1/2} \\
	&=& d_n ((\theta,\eta), (\theta_0,\eta_0)  ).
	\eea
	For any $z\in \bbR$,
	\bea
	d_H^2(\eta_0, T_z(\eta_0)) &=&
	\int (\sqrt{\eta_0(y+z)} - \sqrt{\eta_0(y)} )^2 dy \\
	&\le& z^2 \int \left( \int_0^1 \frac{\dot{\eta_0}(y+tz)}{\sqrt{\eta_0(y+tz)}} dt  \right)^2 dy \\
	&\le& z^2 \int \int_0^1 \left( \frac{\dot{\eta_0}(y+tz)}{\eta_0(y+tz)} \right)^2 \eta_0(y+tz) dt dy \\
	&=& z^2 \int_0^1 \int  \left( \frac{\dot{\eta_0}(y+tz)}{\eta_0(y+tz)} \right)^2 \eta_0(y+tz) dy dt \\
	&=& z^2\bbE_{\eta_0}\Big( \frac{\dot{\eta}_0 }{\eta_0} \Big)^2  \\
	&\le& z^2\bbE_{\eta_0}\Big( \frac{|\dot{\eta}_0| }{\eta_0} \Big)^{2\beta+\nu}  \,\, \le\,\, z^2 C_{2\beta+\nu}  
	\eea
	for some constant $C_{2\beta+\nu}  >0$ depending only on $(\beta,\nu)$ because of  condition \hyperref[D3]{(D3)} on $\eta_0$ and $2\beta + \upsilon \ge 2$. 
	
	If $|z| \le d_H(\eta,\eta_0)/(2\sqrt{C_{2\beta+\nu}})$, then
	\bea
	d_H(\eta, T_z(\eta_0)) 
	&\ge& d_H(\eta,\eta_0) - d_H(\eta_0, T_z(\eta_0)) \\
	&\ge& d_H(\eta,\eta_0) - \sqrt{C_{2\beta+\nu}}|z| \\
	&\ge& \frac{1}{2}d_H(\eta,\eta_0),
	\eea
	and otherwise, if $|z| > d_H(\eta,\eta_0)/(2\sqrt{C_{2\beta+\nu}})$
	\bean
	d_H(\eta, T_z(\eta_0)) &\ge& \frac{1}{2} d_V(\eta, T_z(\eta_0)) \nonumber\\
	&=&  \sup_B |\eta(B) - T_z(\eta_0)(B)| \nonumber\\
	&\ge&  \bigg| \int_0^\infty \eta(y)dy - \int_0^\infty \eta_0(y+z) dy  \bigg|  \nonumber\\
	&=& \bigg| \int_0^\infty \eta(y)dy - \int_{-z}^\infty \eta_0(y+z) dy  - \int_0^{-z} \eta_0(y+z) dy \bigg|  \nonumber \\ 
	&=&\int_0^{|z|} \eta_0(y) dy \label{symm_ineq}  \\
	&\ge& \big\{ \int_0^1 \eta_0(y) dy  \big\}  \wedge  \big\{  (2 \sqrt{C_{2\beta+\nu}} )^{-1} d_H(\eta,\eta_0) \inf_{0\le y \le 1} \eta_0(y) \big\} , \nonumber
	\eean
	where \eqref{symm_ineq} holds due to the symmetric assumption \hyperref[D4]{(D4)} and $\eta\in\calH_{\rm mix}$.

	Thus, we have
	\bea
	K_{\rm Hel} \sqrt{\frac{s_n \log p}{n}} &\ge& d_n\big( (\theta,\eta) , (\theta_0,\eta_0)  \big)  \\
	&\ge& \inf_{z \in \bbR} d_H(\eta, T_z(\eta_0))   \\
	&\ge& \Big[ \frac{1}{2} \wedge \Big\{  \frac{1}{2\sqrt{C_{2\beta+\nu}}} \inf_{0\le y \le 1}\eta_0(y) \Big\}  \Big]  d_H(\eta,\eta_0)
	\eea
	because $s_n\log p =o(n)$, which completes the proof by taking $K_{\rm eta} = K_{\rm Hel} \Big[ \frac{1}{2} \wedge \Big\{  \frac{1}{2\sqrt{C_{2\beta+\nu}}} \inf_{0\le y \le 1}\eta_0(y) \Big\}  \Big]^{-1}$.	 \hfill $\blacksquare$
\end{proof}

\section{Proofs for Bernstein von-Mises Theorem}

We first present three lemmas (Lemma \ref{lemma:post_conc_Hn}, Lemma \ref{lemma:normal} and Lemma \ref{thm:misLAN}), which directly appear in the proof of Theorem \ref{thm:BvM}.
Other auxiliary results used to prove these lemmas will be provided in Section \ref{sec:auxiliary}.


\begin{lemma}\label{lemma:post_conc_Hn}
	Assume that the prior conditions \eqref{prior_p}, \eqref{prior_gS} and \eqref{al_tail}-\eqref{G4} hold.
	Let
	\bea
	\begin{split}
		\calH_n' &= \Bigg\{ \eta(\cdot) = \int \phi_\sigma(\cdot -z) d\widebar{F}(z) \text{ with } F = \sum_{h=1}^\infty \pi_h \delta_{z_h} : \\ 
		& \hspace{-.3cm}  z_h \in [-a_n,a_n], h\le H_n ;  \sum_{h> H_n}\pi_h < \epsilon_n ; \sigma^2 \in [\sigma_{0n}^2,\, \log n \wedge \{\sigma_{0n}^2(1+\epsilon_n^2 )^{M_n}\}  )  \Bigg\},
	\end{split}
	\eea
	where $a_n = (\log n)^{\frac{2}{\tau}}$, $\epsilon_n^2 = C_1 s_n\log p/n$, $H_n =\lfloor C_2 s_n\log p/\log n \rfloor$, $\sigma_{0n}^{-2a_2} = s_n \log p$, $M_n=n$ for some positive constants $C_1$ and $C_2$,
	and define
	\bean\label{H_n_star}
	{\calH}_n^* &:=& \left\{ \eta \in \calH_n' : d_H(\eta,\eta_0) \le K_{\rm eta} \sqrt{s_n \log p/n } \, \, \right\}.
	\eean	
	Then,
	\bea
	\Eaa \Pi \left( \eta \in \left({\calH^*_n}\right)^c \mid D_n \right) &=& o(1)
	\eea
	for any $\eta_0$ satisfying \hyperref[D1]{(D1)}-\hyperref[D5]{(D5)}.
\end{lemma}

\begin{proof} 
	We have
	\bean
	&&\Eaa \Pi \left( \eta \in \left({\calH^*_n}\right)^c \mid D_n \right) \nonumber \\
	&\le& \Eaa \Pi \left( \eta \in \left({\calH'_n}\right)^c \mid D_n \right) \nonumber \\
	&+& \Eaa \Pi \left( d_H(\eta,\eta_0) > K_{\rm eta} \sqrt{\frac{s_n\log p}{n} }   \,\Big|\, D_n \right)  . \quad \quad \,\,\label{eta_H_restricted}
	\eean
	Note that Lemma \ref{lemma:dem_lb} still holds for the prior $\Pi_\calH$ with the support conditions \eqref{al1} and \eqref{G4} because the proof of Theorem 4 of \cite{shen2013adaptive} can be easily modified for the priors with the restricted support with \eqref{al1} and \eqref{G4}.
	Thus,
	\bea
	&& \Eaa \Pi \left( \eta \in \left({\calH'_n}\right)^c \mid Dn \right) \nonumber \\
	&\le& \Eaa \left[ \Pi \left( \eta \in \left({\calH'_n}\right)^c \mid D_n \right) I_{E_n}\right] + o(1) \nonumber \\
	&\le& \Pi_{\calH} \big(\left({\calH'_n}\right)^c \big) \exp \Big\{ \frac{C_{\rm lower}(A_3+2+C_\lambda)}{2K_{\rm dim} } \tilde{s}_n \log p \Big\} + o(1) ,
	\eea
	where $E_n$ is the event \eqref{En_event}, $\tilde{s}_n = 2 K_{\rm dim} \{s_0 \vee  n^{\frac{\kappa^*}{2\beta+\kappa^*}}  (\log n)^{2t_0-1}  \}$ and $t_0 = \{\kappa^*(1+ \tau^{-1}+\beta^{-1}) + 1 \}/(2+ \kappa^*\beta^{-1})$.
	With a slight modification of the proof of Proposition 2 in \cite{shen2013adaptive},   
	\bea
	\Pi_{\calH}\big(\left({\calH'_n}\right)^c \big) &\lesssim& H_n \exp \Big\{ - C'' a_n^{a_1} \Big\}  + \Big\{ \frac{e \alpha(\bbR)}{H_n} \log \Big(\frac{1}{\epsilon_n} \Big)  \Big\}^{H_n}  \\
	&& +\,\, \exp (-C'' \sigma_{0n}^{-2a_2} )  + \sigma_{0n}^{-2a_3} ( 1+ \epsilon_n^2)^{-2M_n a_3}  \\
	&\le& \exp \Big\{  -\frac{1}{2}(C_1 a_3 \wedge C_2 \wedge C'') s_n \log p  \Big\}   .
	\eea
	Thus, 
	\bea
	&& \Eaa \Pi \left( \eta \in \left({\calH'_n}\right)^c \mid Dn \right)  \\
	&\lesssim& \exp \Big\{  -\frac{1}{2}(C_1 a_3 \wedge C_2 \wedge C'') s_n \log p  +\frac{C_{\rm lower}(A_3+2+C_\lambda)}{2K_{\rm dim} } \tilde{s}_n \log p  \Big\} + o(1)  \\
	&=& o(1)
	\eea
	for some large constant $K_{\rm dim}>1$.
	Furthermore, it is easy to see that Corollary \ref{cor:convrate_H} also holds for for the prior $\Pi_\calH$ with \eqref{al1} and \eqref{G4}, which implies that \eqref{eta_H_restricted} is of order $o(1)$.	 \hfill $\blacksquare$
\end{proof}

\begin{lemma}\label{lemma:normal}
	Suppose that $(s_n \log p )^{1+ \frac{ 8}{a_2}} = o(n^{1-\zeta})$ holds for some constant $\zeta>0$. Further assume that $\psi(s_n)$ is bounded away from zero. Let $A_S := \{ h \in \bbR^{|S|} : \|h\|_1 > M_n s_n \sqrt{\log p}  \}$  for some sequence $M_n$ such that $\sqrt{\log p} = o(M_n)$. 
	Then
	\bean\label{normal_conc}
	\sup_{S \in \calS_n} \sup_{\eta \in \calH_n^*} \frac{\int_{A_S} \exp \left(h^T G_{n,\eta,S} - \frac{1}{2}h^T V_{n,\eta,S} h \right) dh }{\int_{\bbR^{|S|}} \exp \left(h^T G_{n,\eta,S} - \frac{1}{2}h^T V_{n,\eta,S} h \right) dh} &=& o_{P_0}(1),
	\eean
	where $\calH_n^*$ defined at \eqref{H_n_star} and 
	\bea
	\calS_n &:=& \left\{ S: |S| \le \frac{s_n}{2}, \, \|\theta_{0, S^c}\|_2 \le \frac{K_{\rm theta}}{\psi(s_n)}\sqrt{\frac{s_n \log p}{n}}  \right\} .
	\eea
\end{lemma}

\begin{proof}
	Note that
	\bea
	\Eaa \left( \sup_{S\in \calS_n} \sup_{\eta \in \calH_n^*} \| G_{n,\eta,S} \|_\infty  \right)
	&\lesssim& {\log p}
	\eea 
	by Lemma \ref{lemma:Gnelldot} and $| h^T G_{n,\eta, S} | \le \|h\|_1 \cdot \|G_{n,\eta,S}\|_\infty$.
	Also note that
	\bea
	h^T V_{n,\eta,S} h &=& \nu_\eta \cdot h^T \sg_S h \\
	&=& \frac{\nu_\eta}{n} \cdot\|X_S h \|_2^2 \\
	&\ge& \nu_\eta \cdot \phi^2( s_n) \|h\|_1^2 \cdot\frac{1}{s_n} \,\,\ge\,\,  \nu_\eta \cdot \psi^2( s_n) \|h\|_1^2 \cdot\frac{1}{s_n}.
	\eea
	Thus, we have
	\bea
	\sup_{S\in \calS_n} \sup_{h\in A_S} \sup_{\eta\in \calH_n^*} \frac{|h^T G_{n,\eta,S}|}{h^T V_{n,\eta,S} h} 
	&\lesssim& \sup_{S\in \calS_n} \sup_{h\in A_S} \sup_{\eta\in \calH_n^*} \frac{\|h\|_1 \cdot \|G_{n,\eta,S}\|_\infty \cdot s_n}{\nu_\eta \psi^2(s_n) \cdot \|h\|_1^2} \\
	&\le& o_{P_0}(1),
	\eea
	because $\sqrt{\log p} = o(M_n)$ and  $\nu_{\eta_0} \gtrsim 1$ holds by Lemma \ref{lemma:score} and assumptions on $\eta_0$.
	It implies that
	\bea
	&& \sup_{S\in \calS_n } \sup_{\eta\in \calH_n^*} \int_{A_S} \exp \left( h^T G_{n,\eta,S} - \frac{1}{2} h^T V_{n,\eta,S} h \right) dh \\
	&\le& \sup_{S\in \calS_n } \sup_{\eta\in \calH_n^*} \int_{A_S} \exp \left( - C h^T V_{n,\eta,S} h \right) dh \\
	&\le& \int_{A_S} \exp \left( - \tilde{C}\|h\|_2^2 \right) dh \\
	&\le& \left( \sqrt{\pi} M_n^2 s_n \log p \right)^{\frac{s_n}{2}} \exp \left( -\frac{1}{3} \tilde{C}' M_n^2 s_n \log p \right)
	\eea
	for some positive constants $C,\tilde{C}$ and $\tilde{C}'$, and all sufficiently large $n$ with $\bbP_{\theta_0,\eta_0}$-probability tending to 1. 
	It is easy to show that
	\bea 
	&&\int \exp \left( h^T G_{n,\eta,S} - \frac{1}{2}h^T V_{n,\eta,S} h \right) dh\\
	&=& (2\pi )^{\frac{|S|}{2}} | V_{n,\eta,S} |^{-\frac{1}{2}} \exp \left( \frac{1}{2\nu_\eta} \|H_S \dot{L}_{n,\eta} \|_2^2 \right),
	\eea
	where $H_S = X_S (X_S^T X_S)^{-1} X_S^T$ and $\dot{L}_{n,\eta} = \left( \dot{\ell}_\eta(y_i-x_i^T \theta_0)  \right)_{i=1}^n \in \bbR^n$. 
	Therefore, the log of the left hand side of \eqref{normal_conc} is bounded above by
	\bea
	&& \frac{s_n}{2} \log \left( \sqrt{\pi}M_n^2 s_n \log p \right) - \frac{1}{3} \tilde{C}' M_n^2 s_n \log p -\frac{|S|}{2} \log (2\pi) + \frac{1}{2} \log |V_{n,\eta,S}|\\
	&& - \frac{1}{2\nu_\eta} \| H_S \dot{L}_{n,\eta}\|_2^2  \\
	&\le& \frac{s_n}{2} \log \left( \sqrt{\pi}M_n^2 s_n \log p \right) - \frac{1}{3} \tilde{C}' M_n^2 s_n \log p + \frac{s_n}{4} \log \left(M_n^2 \nu_\eta \right)
	\eea
	with $\bbP_{\theta_0,\eta_0}$-probability tending to 1. The last term tends to $-\infty$ as $n\to\infty$, thus we get the desired result. \hfill $\blacksquare$
\end{proof}

\bigskip
Define 
\bean
\Theta_n^* &:=& \bigg\{\theta\in \Theta: S_\theta \in \calS_n , \, \|\theta-\theta_0\|_1 \le K_{\rm theta}\frac{s_n}{\phi(s_n)}\sqrt{\frac{\log p}{n}}, \label{theta_n_star} \\
&&  \|\theta-\theta_0\|_2 \le K_{\rm theta}\frac{1}{\psi(s_n)}\sqrt{\frac{s_n\log p}{n}},  \, \|X(\theta-\theta_0)\|_2 \le K_{\rm theta} \sqrt{s_n \log p}\,   \bigg\} \nonumber,
\eean
and let $M_n\Theta_n^*$ be the variant of $\Theta_n^*$ with $M_n K_{\rm theta}$ instead of $K_{\rm theta}$.

\begin{lemma}[Misspecified LAN: version 2]\label{thm:misLAN}
	Suppose that $s_n^6(\log p)^{11}=o(n^{1-\zeta})$, $(s_n\log p )^{1+ \frac{15}{a_2}} = o(n^{1-\zeta})$ and 
	$\left( s_n\log p \right)^{6 + \frac{5}{4a_2}}(\log p)^{\frac{5}{2}} = o(n^{1-\zeta})$ hold for some constant $\zeta>0$. 
	Further assume that $\psi(s_n)$ is bounded away from zero. Define $\Theta_n^*$ and $\calH_n^*$ as \eqref{theta_n_star} and \eqref{H_n_star}, respectively, and let
	\bea
	&& r_n(\theta,\eta) \\
	&:=& L_n(\theta,\eta) - L_n(\theta_0,\eta) - \sqrt{n}(\theta-\theta_0)^T \bbG_n \dot{\ell}_{\theta_0,\eta_0} + \frac{n}{2} (\theta-\theta_0)^T V_{n,\eta_0} (\theta-\theta_0).
	\eea
	Then, we have 
	\bea
	\Eaa \left( \sup_{\theta\in M_n\Theta_n^*} \sup_{\eta\in\calH_n^*} |r_n(\theta,\eta)|  \right) &=& o(1)
	\eea
	for any $\eta_0$ satisfying  \hyperref[D1]{(D1)}-\hyperref[D5]{(D5)} and some sequence $M_n$ such that $\sqrt{\log p}=o(M_n)$.
\end{lemma}

\begin{proof}
	Define $\tilde{r}_n(\theta,\eta)$ as in Lemma \ref{lemma:LAN}.
	Note that
	\bea
	&& \Eaa \left( \sup_{\theta\in M_n\Theta_n^*} \sup_{\eta\in\calH_n^*} |r_n(\theta,\eta)|  \right)\\
	&\le& \Eaa \left( \sup_{\theta\in M_n\Theta_n^*} \sup_{\eta\in\calH_n^*} |r_n(\theta,\eta) - \tilde{r}_n(\theta,\eta)|  \right) \\
	&&+ \Eaa \left( \sup_{\theta\in M_n\Theta_n^*} \sup_{\eta\in\calH_n^*} |\tilde{r}_n(\theta,\eta)|  \right) ,
	\eea
	and, by Lemma \ref{lemma:LAN},
	\bea
	&&\Eaa \left( \sup_{\theta\in M_n\Theta_n^*} \sup_{\eta\in\calH_n^*} |\tilde{r}_n(\theta,\eta)|  \right) \\
	&\lesssim& \frac{M_n^2 s_n^2}{\phi^2(s_n)} \log p  \cdot \sqrt{\frac{ s_n(\log p)^3 +  \left(s_n \log p \right)^{ \frac{3}{a_2} }(\log p)^4 }{n}}  \left(s_n \log p \right)^{\zeta'}   \\
	&+& \frac{M_n^3 s_n}{\phi(s_n)} \sqrt{\frac{\log p}{n}} \cdot s_n (\log p)^{\frac{3}{2}} \\
	&=&  o(1)
	\eea 
	for some small constant $\zeta'>0$ and some sequence $M_n$ when $(s_n\log p )^{1+ \frac{15}{a_2}} = o(n^{1-\zeta})$ and 
	$\left( s_n\log p \right)^{6 + \frac{5}{4a_2}}(\log p)^{\frac{5}{2}} = o(n^{1-\zeta})$. 
	Thus, it suffices to show that
	\bea
	\Eaa \left( \sup_{\theta\in M_n\Theta_n^*} \sup_{\eta\in\calH_n^*} |r_n(\theta,\eta) - \tilde{r}_n(\theta,\eta)|  \right) 
	&=& o(1).
	\eea
	
	By the definition of $r_n(\theta,\eta)$ and $\tilde{r}_n(\theta,\eta)$,
	\bean
	|r_n(\theta,\eta) - \tilde{r}_n(\theta,\eta)|
	&\le& \sqrt{n}\left| (\theta-\theta_0)^T \bbG_n \left( \dot{\ell}_{\theta_0,\eta} - \dot{\ell}_{\theta_0,\eta_0} \right)  \right|  \label{Gn_diff} \\
	&+& \frac{n}{2} \left| (\theta-\theta_0)^T (V_{n,\eta}  - V_{n,\eta_0}) (\theta-\theta_0)  \right|. \label{Vn_diff}
	\eean
	The supremum of \eqref{Vn_diff} is easily bounded above by
	\bea
	&&\sup_{\theta\in M_n\Theta_n^*} \sup_{\eta\in\calH_n^*}n \left| (\theta-\theta_0)^T (V_{n,\eta}  - V_{n,\eta_0}) (\theta-\theta_0)  \right|  \\
	&=&  \sup_{\theta\in M_n\Theta_n^*} \sup_{\eta\in\calH_n^*}|\nu_\eta - \nu_{\eta_0}| \cdot \|X(\theta-\theta_0)\|_2^2 \\
	&\lesssim& \sup_{\theta\in M_n\Theta_n^*} \sup_{\eta\in\calH_n^*} \epsilon_n^{\frac{2}{5} - \zeta} M_n^2 s_n \log p
	\eea
	by Lemma \ref{lemma:score2}, where $\epsilon_n = K_{\rm eta}\sqrt{s_n\log p/n}$, which is of order $o(1)$ under the assumption  $s_n^6 (\log p )^{11} = o(n^{1-\zeta})$. 
	Note that
	\bea
	\sqrt{n}\left| (\theta-\theta_0)^T \bbG_n \left( \dot{\ell}_{\theta_0,\eta} - \dot{\ell}_{\theta_0,\eta_0} \right)  \right| 
	&\le& \sqrt{n} \|\theta-\theta_0\|_1 \cdot \| \bbG_n ( \dot{\ell}_{\theta_0,\eta} - \dot{\ell}_{\theta_0,\eta_0} ) \|_\infty \\
	&\lesssim& \frac{M_n s_n}{\phi(s_n)} \sqrt{\log p} \cdot \sup_{\eta\in \calH_n^*} \| \bbG_n ( \dot{\ell}_{\theta_0,\eta} - \dot{\ell}_{\theta_0,\eta_0} ) \|_\infty.
	\eea
	Define 
	\bea
	\calL_{n,j} &:=& \left\{ M_n s_n \sqrt{\log p} \cdot e_j^T \left(\dot{\ell}_{\theta_0,\eta} - \dot{\ell}_{\theta_0,\eta_0} \right) : \eta \in \calH_n^*  \right\}
	\eea
	and $\calL_n := \cup_{j=1}^p \calL_{n,j}$, where $e_j$ is the $j$th unit vector in $\bbR^p$.  
	Then 
	$L_n(x,y) := M \sqrt{\log p}\cdot M_n s_n \sqrt{\log p} \cdot \sup_{\eta \in\calH_n^*} |\dot{\ell}_{\eta}(y) - \dot{\ell}_{\eta_0}(y)|$ is an envelop function of $\calL_n$, and
	\bea
	\|L_n \|_n
	&\lesssim& M_n s_n {\log p} \cdot \bigg\{ \Eaa \Big[ \sup_{\eta\in\calH_n^*} \big( \dot{\ell}_{\eta}(Y) - \dot{\ell}_{\eta_0}(Y) \big)^2\Big] \bigg\}^{\frac{1}{2}} \\
	&\lesssim& M_n s_n {\log p} \cdot \left( \frac{s_n \log p}{n} \right)^{\frac{1}{5}- \zeta} 
	\eea
	by Lemma \ref{lemma:score2}.
	We will use Corollary A.1 in \cite{chae2019bayesian}, which implies 
	\bea
	&&M_n s_n \sqrt{\log p}  \cdot \Eaa \left( \sup_{\eta\in \calH_n^*} \| \bbG_n ( \dot{\ell}_{\theta_0,\eta} - \dot{\ell}_{\theta_0,\eta_0} ) \|_\infty \right)\\
	&\lesssim& \int_0^{\|L_n\|_n} \sqrt{\log N_{[\,]}^n (\epsilon, \calL_{n} )} d\epsilon .
	\eea
	
	Note that
	\bea
	N_{[\,]}^n(\epsilon, \calL_{n,j}) 
	&\le& N_{[\, ]} \left( \frac{\epsilon}{M M_n s_n {\log p}}, \calG_n, L_2(P_{\eta_0})  \right),
	\eea  
	where $\calG_n := \{\dot{\ell}_\eta : \eta\in \calH_n^*  \}$, and 
	\bean\label{logN_L2todH}
	\begin{split}
		\log N_{[\,]} \left(\epsilon, \calG_n, L_2(P_{\eta_0})  \right)
		&\le \log N_{[\, ]} \left( \epsilon^\gamma, \calH_n^*, d_H \right) \\
		&\le \log N_{[\, ]} \left( \epsilon^\gamma, \calH_n, d_H \right).
	\end{split}
	\eean
	Let $a_n = (\log n)^{\frac{2}{\tau}}, b_{1n} = \left(s_n \log p \right)^{-\frac{1}{2a_2}}$ and $b_{2n} = \sqrt{\log n}$.
	By Lemma 3 of \cite{ghosal2007posterior}, 
	\bea
	\log N\left(\epsilon , \calH_n , \|\cdot\|_\infty \right)
	&\lesssim& \frac{a_n}{b_{1n}} \cdot \log \frac{1}{\epsilon}\cdot  \left( \log \frac{1}{\epsilon} + \log \frac{a_n}{b_{1n}}  \right).
	\eea
	Now we use the similar argument to the proof of Theorem 6 of \cite{ghosal2007posterior}. Define
	\bea
	H(x) &=& b_{1n}^{-1} \phi \left(\frac{x}{2b_{2n}} \right)I(|x| > 2a_n) + b_{1n}^{-1} \phi(0) I( |x| \le 2a_n),
	\eea
	where $\phi$ is the density function of the standard normal distribution. $H$ is an envelop function for $\calH_n$.
	For some $\varrho >0$, let $g_1,\ldots, g_T$ be a $\varrho$-net for $\|\cdot\|_\infty$, $l_i := (g_i - \varrho) \vee 0$ and $u_i := (g_i + \varrho) \wedge H$. Then, the brackets $[l_i, u_i]$ cover $\calH_n$. 
	Let $\varrho = C \epsilon^2 (a_n b_{2n})^{-1} [ \log (1/\epsilon)  ]^{-\frac{1}{2}} $ for some constant $C>0$, then for $D_n = 2a_n b_{2n} [ \log (1/\epsilon)  ]^{\frac{1}{2}} > 2a_n$,
	\bea
	\int (u_i - l_i)\, d\mu 
	&\lesssim& \|u_i - l_i \|_\infty \cdot D_n + \int_{|x|> D_n}  \frac{1}{b_{1n}}\phi \left( \frac{x}{2b_{2n}} \right)  dx \\
	&\lesssim& \varrho \cdot D_n + \frac{b_{2n}}{b_{1n}} \exp \left( - \frac{D_n^2}{8 b_{2n}^2} \right) \\
	&\lesssim& \epsilon^2 + \frac{b_{2n}}{b_{1n}}\cdot \epsilon^{ca_n^2}  \\
	&\lesssim& \epsilon^2
	\eea
	for some constant $c>0$ and any $\epsilon<1$. The second inequality follows from the Chernoff's inequality.
	Thus,
	\bea
	\log N_{[\,]} \left( \epsilon ,\calH_n , d_H \right)
	&\le& \log N_{[\,]} \left( \epsilon^2, \calH_n, \|\cdot\|_1 \right) \\
	&\le& \log N \left( C \cdot \frac{\epsilon^2}{a_n b_{2n}} \left[\log \frac{1}{\epsilon} \right]^{-\frac{1}{2}} , \calH_n, \|\cdot\|_\infty \right) \\
	&\lesssim&  \frac{a_n}{b_{1n}}\cdot \left[ \left(\log \frac{1}{\epsilon} \right)^2 + \left(\log n \right)^2 \right],
	\eea
	and by \eqref{logN_L2todH},
	\bea
	\log N_{[\,]}^n \left(\epsilon, \calL_n \right)
	&\le& \log p + \log N_{[\, ]} \left( \frac{\epsilon}{M M_n s_n {\log p}}, \calG_n, L_2(P_{\eta_0})  \right) \\
	&\lesssim& \log p + \left(s_n \log p \right)^{\frac{1}{2a_2}} [\log n]^{\frac{2}{\tau}} \cdot \left[ \left(\log \frac{1}{\epsilon} \right)^2 + \left(\log n \right)^2 \right].
	\eea  
	Then by Corollary A.1 in \cite{chae2019bayesian}, we have 
	\bean
	&& \Eaa \left( \sup_{\eta\in \calH_n^*} \| \bbG_n ( \dot{\ell}_{\theta_0,\eta} - \dot{\ell}_{\theta_0,\eta_0} ) \|_\infty \right) \cdot \frac{M_n s_n}{\phi(s_n)} \sqrt{\log p} \nonumber \\
	&\lesssim& \int_0^{\|L_n\|_n} \sqrt{\log N_{[\,]}^n (\epsilon, \calL_{n} )} d\epsilon  \nonumber \\
	&\lesssim& \int_0^{\|L_n\|_n} \sqrt{\log p} + \left(s_n \log p \right)^{\frac{1}{4a_2}} [\log n]^{\frac{1}{\tau}} \cdot \left( \log \frac{1}{\epsilon} + \log n  \right) d\epsilon \nonumber \\
	&\lesssim& \|L_n\|_n \sqrt{\log p} +  \left(s_n \log p \right)^{\frac{1}{4a_2}} [\log n]^{\frac{1}{\tau} +1} \cdot \int_0^{\|L_n\|_n}  \log \frac{1}{\epsilon}\,\, d \epsilon  \nonumber \\
	&\lesssim& M_n s_n {\log p} \cdot \left( \frac{s_n \log p}{n} \right)^{\frac{1}{5}- \zeta'} \left\{ \sqrt{\log p}  +  \left(s_n \log p \right)^{\frac{1}{4a_2}} [\log n]^{\frac{1}{\tau} +1}   \right\} \label{Gnelldot_diff_rate}
	\eean
	because $\int_0^u \log (1/\epsilon) d\epsilon \le \int_0^u \epsilon^{-1+\zeta''} d\epsilon \lesssim u^{1-\zeta''}$ for any small $\zeta''>0$ and $0<u<1$. \eqref{Gnelldot_diff_rate} converges to zero as $n\to\infty$ under the assumptions $\left( s_n\log p \right)^{6 + \frac{5}{4a_2}}(\log p)^{\frac{5}{2}} = o(n^{1-\zeta})$ and $s_n^6(\log p)^{11}= o(n^{1-\zeta})$ for some constant $\zeta>0$. 
	Thus, we have shown \eqref{Gn_diff}, and this completes the proof.  \hfill $\blacksquare$
\end{proof}

\bigskip

Now, we prove Theorem \ref{thm:BvM} using the above results (Lemma \ref{lemma:post_conc_Hn}, Lemma \ref{lemma:normal} and Lemma \ref{thm:misLAN}) and posterior convergence rate results (Theorem \ref{thm:dimupper}, Corollary \ref{cor:convrate_H} and Corollary \ref{cor:convrate_coef}).

\begin{proof}[Proof of Theorem \ref{thm:BvM}]
	Let $\Theta_n^*$ and $\calH_n^*$ be defined as \eqref{theta_n_star} and \eqref{H_n_star}, respectively.
	Define $\breve{\Pi}_\Theta := \Pi_\Theta \,|_{M_n\Theta_n^*}$ and $\breve{\Pi}_\calH := \Pi_\calH \,|_{\calH_n^*}$ as the restricted and renormalized priors on $M_n\Theta_n^*$ and $\calH_n^*$, respectively. Let $\breve{\Pi}(\cdot|D_n)$ be the posterior distribution corresponding to the prior $\breve{\Pi} = \breve{\Pi}_\Theta \times \breve{\Pi}_\calH$.
	We first prove that
	\bean
	d_V \left( \breve{\Pi}(\cdot| D_n) , \Pi(\cdot| D_n) \right) &=& o_{P_0}(1) \quad\text{ and} \label{restPi}\\
	d_V \left( \breve{\Pi}^\infty(\cdot| D_n) , \Pi^\infty(\cdot| D_n) \right) &=& o_{P_0}(1), \label{restPiinf}
	\eean
	where $\breve{\Pi}^\infty(\cdot| D_n) :=\Pi^\infty(\cdot| D_n) \, |_{M_n\Theta_n^*}$. 
	Note that for any measurable set $A \in \Theta\times \calH$,
	\bea
	\breve{\Pi}(A \mid D_n) &=& \frac{\Pi(A \cap [M_n\Theta_n^* \times \calH_n^*] \mid D_n )}{\Pi( M_n\Theta_n^* \times \calH_n^* \mid D_n )} \\
	&=& \frac{\Pi(A \mid D_n) - \Pi(A \cap [M_n\Theta_n^* \times \calH_n^*]^c \mid D_n ) }{\Pi(\Theta\times \calH_{\rm mix} \mid D_n ) - \Pi([M_n\Theta_n^* \times \calH_n^*]^c \mid D_n ) } \\
	&=& \Pi(A \mid D_n) + o_{P_0}(1)
	\eea
	by Corollaries \ref{cor:convrate_H}, \ref{cor:convrate_coef} and Lemma \ref{lemma:post_conc_Hn}, which implies \eqref{restPi}. 	
	Define 
	\bean\label{sn_set}
	\calS_n &:=& \left\{S: |S| \le \frac{s_n}{2}, \quad \|\theta_{0,S^c}\|_2 \le \frac{K_{\rm theta}}{\psi(s_n)} \sqrt{\frac{s_n \log p}{n}}  \right\} ,
	\eean
	$\Theta_S^* := \{ \theta_S \in \bbR^{|S|} : \widetilde{\theta}_S \in M_n \Theta_n^* \}$ and $H_S := \sqrt{n}(\Theta_{S}^* - \theta_{0,S})$ for some sequence  $M_n$ such that $\sqrt{\log p}=o(M_n)$ and  
	\bea
	\sup_{\theta\in M_n\Theta_n^*} \sup_{\eta\in\calH_n^*} |r_n(\theta,\eta)|  &=& o_{P_0}(1),
	\eea
	where $r_n(\theta,\eta)$ is defined in Lemma \ref{thm:misLAN}.
	Then, 
	\bea
	d\breve{\Pi}(\theta \mid D_n) &=& \sum_{S\in \calS_n} \widetilde{w}_S \cdot d \widetilde{Q}_{S}(\theta_S) d \delta_0(\theta_{S^c}), \\
	d\breve{\Pi}^\infty(\theta \mid D_n) &=& \sum_{S\in \calS_n} \widetilde{w}_S^\infty \cdot n^{-\frac{|S|}{2}}  d \widetilde{N}_{n,S}(h_S) d \delta_0(\theta_{S^c}),
	\eea
	where $\widetilde{Q}_S = Q_{S} \,|_{\Theta_S^*}$ and $\widetilde{N}_{n,S} := {N}_{n,S}\,|_{H_S}$ are the restricted and renormalized distributions, 
	\bea
	\widetilde{w}_S  &:=&  \frac{Q_S(\Theta_S^*)}{\sum_{S' \in \calS_n} w_{S'} Q_{S'} (\Theta_{S'}^*)  } \cdot w_S,    \\
	\widetilde{w}_S^\infty &:=&  \frac{N_{n,S}(H_S)}{\sum_{S' \in \calS_n} w_{S'} N_{n,S'} (H_{S'})  } \cdot w_S,
	\eea
	and $h_S = \sqrt{n}(\theta_S-\theta_{0,S}) \in H_S$. It is easy to show that
	\bean
	\sup_{S\in \calS_n} \left| 1- \frac{w_S}{\widetilde{w}_S^\infty } \right| &=& o_{P_0}(1) \quad \text{ and} \label{W1} \\
	\sup_{S\in \calS_n} d_V \left(N_{n,S}, \widetilde{N}_{n,S} \right) &=& o_{P_0}(1) \label{W2}
	\eean
	hold by Theorem \ref{thm:dimupper} and Lemma \ref{lemma:normal}. 
	Then, by Lemma 4.5 in \cite{chae2019bayesian}, 
	\bea
	&&d_V \left( \breve{\Pi}^\infty(\cdot | D_n), \Pi^\infty(\cdot| D_n) \right) \\
	&\le& 2 d_V ( \widetilde{w}^\infty, w) + \sum_{S\in \calS} w_S d_V(\widetilde{N}_{n,S} , N_{n,S} ) \\
	&\le& 2 \sum_{S\in \calS_n} \widetilde{w}_S^\infty \left|  1 - \frac{w_S}{\widetilde{w}_S^\infty} \right| + \sum_{S\in \calS_n} w_S \cdot \sup_{S\in \calS_n} d_V(\widetilde{N}_{n,S} , N_{n,S} )  \\
	&&+\,\, 4 \sum_{S\in \calS_n^c} w_S,
	\eea
	where $w = (w_S)_{S \in \calS}$ and $\widetilde{w}^\infty = (\widetilde{w}_S^\infty)_{S \in \calS_n}$.
	It implies that \eqref{restPiinf} holds by \eqref{W1}, \eqref{W2} and Theorem \ref{thm:dimupper}.
	
	Now we have \eqref{restPi} and \eqref{restPiinf}, so it suffices to prove that
	\bean\label{dV_on_rest}
	d_V \left( \breve{\Pi}(\cdot| D_n) , \breve{\Pi}^\infty(\cdot| D_n)  \right) &=& o_{P_0}(1).
	\eean
	Again by Lemma 4.5 in \cite{chae2019bayesian}, if we show that
	\bean
	d_V(\widetilde{w}, \widetilde{w}^\infty ) &=& o_{P_0}(1) \quad \text{ and} \label{W3} \\
	\sup_{S\in \calS_n} d_V ( \widetilde{Q}_S, \widetilde{N}_{n,S} ) &=& o_{P_0}(1) ,\label{W4}
	\eean
	where $\widetilde{w} = (\widetilde{w}_S)_{S \in \calS_n}$, it implies the desired result, \eqref{dV_on_rest}.
	Note that
	\bea
	d_V (\widetilde{w}, \widetilde{w}^\infty ) &=& \sum_{S\in \calS_n} |\widetilde{w}_S - \widetilde{w}_S^\infty| \\
	&=& \sum_{S\in\calS_n} \left|1- \frac{\widetilde{w}_S}{\widetilde{w}_S^\infty} \right| \cdot \widetilde{w}_S^\infty \\
	&=& \sum_{S\in\calS_n} \left|1-  Q_S(\Theta_S^*) \frac{w_S}{\widetilde{w}_S^\infty}(1+ o_{P_0}(1) ) \right| \cdot \widetilde{w}_S^\infty \\
	&=& \sum_{S\in \calS_n} |1 -Q_S(\Theta_S^*) (1+ o_{P_0}(1) )  | \cdot \widetilde{w}_S^\infty \\
	&\le& \sup_{S \in \calS_n} (1- Q_S(\Theta_S^*) ) + o_{P_0}(1) \,\,= \,\, o_{P_0}(1) .
	\eea
	The third and fourth equality hold by Theorem \ref{thm:dimupper}, Corollary \ref{cor:convrate_coef} and \eqref{W1}, respectively. Thus, we have proved \eqref{W3}. 
	For any measurable set $B$,
	\bea
	&& \breve{\Pi} (\theta_S \in B \mid D_n ,\eta, S_\theta= S) \\
	&=& \frac{\int_{B \cap \Theta_S^*} \exp \left( L_n(\widetilde{\theta}_S, \eta ) - L_n(\theta_0,\eta)  \right) \cdot g_S(\theta_S)/g_S(\theta_{0,S}) \, d\theta_S }{ \int_{\Theta_S^*} \exp \left( L_n(\widetilde{\theta}_S, \eta ) - L_n(\theta_0,\eta)  \right) \cdot g_S(\theta_S)/g_S(\theta_{0,S})\, d\theta_S} \\
	&=& \frac{\int_{B \cap \Theta_S^*} \exp \left( \sqrt{n}(\theta_S - \theta_{0,S})^T G_{n,\eta_0,S} - \frac{n}{2}(\theta_S-\theta_{0,S})^T V_{n,\eta_0,S}(\theta_S-\theta_{0,S})  \right)  d\theta_S }{ \int_{\Theta_S^*} \exp \left( \sqrt{n}(\theta_S - \theta_{0,S})^T G_{n,\eta_0,S} - \frac{n}{2}(\theta_S-\theta_{0,S})^T V_{n,\eta_0,S}(\theta_S-\theta_{0,S})  \right)  d\theta_S} \\
	&&+ o_{P_0}(1) 
	\eea
	by Lemma \ref{thm:misLAN} and
	\bea
	\sup_{S\in \calS_n} \sup_{\theta_S\in \Theta_S^*} \left|\log \frac{g_S(\theta_S)}{g_S(\theta_{0,S})} \right| 
	&=& \sup_{S\in \calS_n} \sup_{\theta_S\in \Theta_S^*} \bigg| \log \exp \left(\lambda \|\theta_{0,S}- \theta_S\|_1 \right) \bigg|  \\
	&\lesssim& \sup_{S\in \calS_n}  \lambda \cdot \frac{M_n s_n}{\phi(s_n)} \sqrt{\frac{\log p}{ n}} \,\, =\,\, o(1)
	\eea 
	for some sequence $M_n$ such that $\sqrt{\log p}=o(M_n)$   because we assume $\lambda s_n {\log p} = o(\sqrt{n})$.
	Then, 
	\bea
	\widetilde{Q}_S (h_S \in B)
	&=& \int_{\calH_n^*}  \breve{\Pi}( h_S \in B \mid D_n, \eta, S_\theta= S) d \breve{\Pi}(\eta \mid D_n, S_\theta=S) \\
	&=& \int_{\calH_n^*} \widetilde{N}_{n,S}(B) d \breve{\Pi}(\eta \mid D_n, S_\theta=S) + o_{P_0}(1) \\
	&=& \widetilde{N}_{n,S}(B)  + o_{P_0}(1),
	\eea
	which implies $\sup_{S\in \calS_n} d_V (\widetilde{Q}_S, \widetilde{N}_{n,S} ) = o_{P_0}(1)$.  \hfill $\blacksquare$
\end{proof}

\section{Proof for Strong Model Selection Consistency}
\begin{proof}[Proof of Theorem \ref{thm:selection}]
	Define $\calS_n$ and $\breve{\Pi}$ as in the proof of Theorem 3.5.
	Define the set $\calS_n' = \{S \in \calS_n : S \supsetneq S_0  \}$, then it suffices to show that $\breve{\Pi}(S_\theta\in \calS_n' \mid D_n) \lra 0$ by \eqref{restPi}. 
	Note that
	\bea
	&&\breve{\Pi}(S_\theta = S \mid D_n, \eta) \\
	&=& \frac{\pi_p(|S|) \binom{p}{|S|}^{-1} \int_{\Theta_S^*}\exp \left(L_n(\widetilde{\theta}_S, \eta) - L_n(\theta_0,\eta) \right) g_S(\theta_S)d\theta_S }{\sum_{S\in \calS_n}\pi_p(|S|) \binom{p}{|S|}^{-1} \int_{\Theta_S^*}\exp \left(L_n(\widetilde{\theta}_S, \eta) - L_n(\theta_0,\eta) \right) g_S(\theta_S)d\theta_S }.
	\eea
	Then, by Lemma \ref{thm:misLAN},
	\bea
	&&\breve{\Pi}(S_\theta \in \calS_n' \mid D_n, \eta) \\
	&=& \frac{\sum_{S\in \calS_n'} \pi_p(|S|) \binom{p}{|S|}^{-1} \int_{\Theta_S^*}\exp \left(L_n(\widetilde{\theta}_S, \eta) - L_n(\theta_0,\eta) \right) g_S(\theta_S)d\theta_S }{\sum_{S\in \calS_n}\pi_p(|S|) \binom{p}{|S|}^{-1} \int_{\Theta_S^*}\exp \left(L_n(\widetilde{\theta}_S, \eta) - L_n(\theta_0,\eta) \right) g_S(\theta_S)d\theta_S } \\
	&\le& \sum_{S\in \calS_n'} \frac{\what{w}_S }{\what{w}_{S_0}}  e^{2\xi_n} \\
	&\lesssim& \sum_{s=s_0+1}^{s_n/2} \frac{\pi_p(s)}{\pi_p(s_0)} \binom{s}{s_0} \left(\frac{\lambda \sqrt{\pi}}{\sqrt{2\nu_{\eta_0}}}  \right)^{s-s_0} \\
	&& \times 	\max_{|S|=s} \left[ \frac{|X_{S_0}^T X_{S_0}|^{1/2}}{|X_{S}^T X_{S}|^{1/2}} \exp \left\{ \frac{1}{2\nu_{\eta_0}} \|(H_S- H_{S_0}) \dot{L}_{n,\eta_0} \|_2^2  \right\}  \right]
	\eea
	for any $\eta$ and some sequence $\xi_n\to 0$, where 
	\bea
	&& \what{w}_S \\
	&\propto& \pi_p(|S|) \binom{p}{|S|}^{-1} \times \\
	&&  \int_{\Theta_S^*}\exp \left( \sqrt{n}(\theta_S-\theta_{0,S})^T G_{n,S} - \frac{n}{2}(\theta_S-\theta_{0,S})^T V_{n,S}(\theta_S-\theta_{0,S})  \right) g_S(\theta_S)d\theta_S .
	\eea
	Note that, by the condition on $\pi_p$ and the definition of $\psi^2(s)$, $\pi_p(s)/\pi_p(s_0) \le A_2^{s-s_0} p^{-A_4(s-s_0)}$ and $|X_{S_0}^T X_{S_0}|/|X_{S}^T X_{S}| \le (n \psi^2(s_n))^{|S|-s_0}$ for any $S\in \calS_n'$. Thus, it suffices to prove that
	\bean\label{quadL_bound}
	&&\bbP_{\eta_0} \Big(\, \frac{1}{2\nu_{\eta_0}}\|(H_S-H_{S_0})\dot{L}_{n,\eta_0} \|_2^2 > K_{\rm sel}(s-s_0)\log p , \text{ for some } S\in\calS_n' \, \Big) \nonumber\\
	&=& o(1) \quad\quad\quad
	\eean
	for some positive constant $K_{\rm sel}$ depending only on $\eta_0$ such that $ A_4 > K_{\rm sel}$.
	
	The left hand side of \eqref{quadL_bound} is bounded above by 
	\bea
	&& \sum_{s=s_0+1}^{s_n/2} \binom{p-s_0}{s-s_0} \bbP_{\eta_0} \Big( \|(H_S-H_{S_0})\dot{L}_{n,\eta_0} \|_2^2 > 2\nu_{\eta_0} K_{\rm sel}(s-s_0)\log p  \Big)  \\
	&\le& \sum_{s=s_0+1}^{s_n/2} \binom{p-s_0}{s-s_0} e^{ -t \cdot 2\nu_{\eta_0}  K_{\rm sel}(s-s_0)(\log p - \nu_{\eta_0}K_{\rm sel}^{-1})  } \times \bbE_{\theta_0,\eta_0} e^{t \|(H_S-H_{S_0})\dot{L}_{n,\eta_0} \|_2^2}
	\eea
	for any $t >0$, where $H_S = X_S (X_S^T X_S)^{-1} X_S^T$. Note that $\dot{\ell}_{\eta_0}(y_i-x_i^T\theta_0)$ is a sub-Gaussian by assumption.
	By Lemma B.2 in \cite{chae2019bayesian} (Hanson-Wright inequality), 
	\bea
	\bbE_{\theta_0,\eta_0} e^{t_0 \|(H_S-H_{S_0})\dot{L}_{n,\eta_0} \|_2^2}
	&\lesssim& e^{C(|S|-s_0)}
	\eea
	for some positive constants $C$ and $t_0$ depending only on $\eta_0$. 
	Thus, if we choose $K_{\rm sel} = (\nu_{\eta_0} t_0)^{-1}$, the left hand side of \eqref{quadL_bound} tends to zero as $n\to\infty$.  \hfill $\blacksquare$
\end{proof}

\section{Auxiliary Lemmas}\label{sec:auxiliary}

We first introduce Lemma \ref{lemma:intpq_ub}, which is used to prove lemmas \ref{lemma:score}, \ref{lemma:score2} and \ref{lemma:score_L2_not_uniform}.

\begin{lemma}\label{lemma:intpq_ub}
	Let $B$ be a subset of $\bbR$ and for given $\epsilon>0$, $p$ and $q$ be probability densities on $\bbR$ such that $d_H^2(p,q) \le \epsilon^2$. 
	Suppose $M_\delta^2 := \int_B p (p/q)^\delta < \infty$ for some $\delta \in (0,1)$.
	Then,
	\bea
	\int_B  p \left( \log \frac{p}{q} \right)^2 &\le& 20 \epsilon^2 \left[ \frac{1}{\delta}   \left(1 \vee \log \frac{M_\delta}{\epsilon} \right)  \right]^2.
	\eea
\end{lemma}
\begin{proof}
	The main strategy for the proof is similar to the proof of Theorem 5 in \cite{wong1995probability}.
	Note that 
	\bea
	\int_B p \left(\log \frac{p}{q} \right)^2 
	&\le& \int_{0< p/q \le K^2} p \left(\log \frac{p}{q} \right)^2 + \int_{B \cap (p/q > K^2)} p \left(\log \frac{p}{q} \right)^2
	\eea
	for any $K>0$. 
	Let $K^\delta = e\vee (M_\delta/\epsilon) >1$ and $r = \sqrt{p/q}-1$.
	Then,
	\bea
	\int_{0< p/q \le K^2} p \left(\log \frac{p}{q} \right)^2 
	&=& \int_{-1< r \le K-1} q (r+1)^2 (2\log (r+1))^2 \\
	&=& \int_{-1< r \le K-1,\, r\neq 0} qr^2 \left(\frac{r+1}{r} \right)^2 (2\log (r+1))^2 \\
	&\le& 16\int_{-1< r \le K-1,\, r\neq 0} qr^2  (\log K)^2 \,\,\, \le \,\,\, 16 \epsilon^2 (\log K)^2
	\eea
	because $(x+1)/x \log(x+1)$ is increasing for $x>-1, x\neq 0$ and $\int qr^2 = d_H^2(p,q) \le \epsilon^2$ by assumption.
	On the other hand,
	\bea
	\int_{B \cap (p/q > K^2)} p \left(\log \frac{p}{q} \right)^2
	&=& \int_{B \cap (p/q > K^2)} p \left(\frac{p}{q} \right)^{\delta} \frac{(\log \frac{p}{q} )^2}{(\frac{p}{q} )^{\delta}} \\
	&\le& \int_{B \cap (p/q > K^2)} p \left(\frac{p}{q} \right)^{\delta} \frac{(2\log K)^2}{K^{2\delta}}\\
	&\le& 4 M_\delta^2  \frac{( \log K)^2}{K^{2\delta}},
	\eea
	because $\log x/ x^{\delta}$ is decreasing for $x \ge e^{1/\delta}$.
	Thus, we have
	\bea
	\int_B p \left(\log \frac{p}{q} \right)^2 
	&\le& 16\epsilon^2 (\log K)^2 + 4 M_\delta^2  \frac{( \log K)^2}{K^{2\delta}} \\
	&\le& 20\epsilon^2 \left[ \frac{1}{\delta}   \left(1 \vee \log \frac{M_\delta}{\epsilon} \right)  \right]^2
	\eea
	by the definition of $K$.  \hfill $\blacksquare$
\end{proof}

\bigskip
\noindent
The following lemma gives a (uniform) convergence rate for the score function, which plays an important role in proving the BvM theorem.
This lemma is used to prove lemmas \ref{lemma:normal} and \ref{lemma:Gnelldot}.

\begin{lemma}\label{lemma:score}
	Let $\epsilon_n = K_{\rm eta}\sqrt{s_n \log p /n}$ and assume that $(s_n \log p)^{2} = o(n)$. For any constant $\zeta>0$, there exists a constant $K_\zeta>0$ not depending  on $(n,p)$ such that
	\bea
	\int \sup_{\eta\in \calH_n^*} \left(\dot{\ell}_\eta(y) - \dot{\ell}_{\eta_0}(y) \right)^2 dP_{\eta_0} (y)
	&\le& K_\zeta \left(\epsilon_n \right)^{\frac{4}{5}-\zeta} (s_n \log p)^{\frac{16}{5a_2}}
	\eea
	for any $\eta_0$ satisfying  \hyperref[D1]{(D1)}-\hyperref[D5]{(D5)} and all sufficiently large $n$, where $\calH_n^*$ defined at \eqref{H_n_star}.
\end{lemma}

\begin{proof}
	We first state some inequalities that we frequently use in the proof. 
	For any $\eta \in \calH_n^*$ and any $y\in \bbR$,
	\bea
	| \ell_{\eta}(y) | &=& \Big|\log \big\{ \int (2\pi \sigma^2)^{-1/2} \exp \big(  -(y-z)^2 /(2\sigma^2) \big)  d\widebar{F}(z)  \big\}    \Big|  \\
	&\le& \Big|\log \Big\{  (s_n\log p)^{\frac{1}{2a_2}} \exp \Big( -(y^2  + (\log n )^{\frac{4}{\tau}}) (s_n\log p )^{\frac{1}{a_2}}  \Big)  \Big\}    \Big|  \\
	&\le& \frac{1}{a_2} \log (s_n\log p) + \big\{  y^2 + (\log n )^{\frac{4}{\tau}} \big\} (s_n \log p)^{\frac{1}{a_2}}   \\
	&\le& 2 \big\{ y^2 + (\log n )^{\frac{4}{\tau}}  \big\}(s_n \log p)^{\frac{1}{a_2}} ,\\
	|\dot{\ell}_{\eta}(y) | &=& \Big| \frac{\int - (\frac{y-z}{\sigma^2}) \phi_\sigma(y-z)d\widebar{F}(z) }{ \int \phi_\sigma(y-z)d\widebar{F}(z) }  \Big|\\
	&\le& \frac{1}{\sigma^2} \big\{ |y| + (\log n )^{\frac{2}{\tau}}   \big\}  \\
	&\le& \big\{ |y| + (\log n )^{\frac{2}{\tau}}   \big\} (s_n \log p)^{\frac{1}{a_2}}  ,\\
	|\ddot{\ell}_{\eta}(y) | &=& \Big| \frac{\ddot{\eta}(y)}{\eta(y)} -  \Big\{ \frac{\dot{\eta}(y)}{\eta(y)}  \Big\}^2  \Big| \\
	&\le& \frac{|\ddot{\eta}(y)|}{\eta(y)} + |\dot{\ell}_{\eta}(y) |^2 \\
	&\le& \frac{1}{\eta(y)} \Big| \int\frac{1}{\sigma^2}\phi_\sigma(y-z)d\widebar{F}(z) + \int \frac{(y-z)^2}{\sigma^4}\phi_\sigma(y-z)d\widebar{F}(z)  \Big| \\
	&& + 2 \big\{ y^2 + (\log n )^{\frac{4}{\tau}}   \big\} (s_n \log p)^{\frac{2}{a_2}} \\
	&\le& \frac{1}{\sigma^2} + \frac{2}{\sigma^4} \big\{ y^2 + (\log n)^{\frac{4}{\tau}}  \big\}  + 2 \big\{ y^2 + (\log n )^{\frac{4}{\tau}}   \big\} (s_n \log p)^{\frac{2}{a_2}} \\
	&\le& 5 \big\{ y^2 + (\log n )^{\frac{4}{\tau}}   \big\} (s_n \log p)^{\frac{2}{a_2}} 
	\eea
	and 
	\bea
	|\dddot{\ell}_{\eta}(y) | &=&  \Big|\frac{\dddot{\eta}(y)}{\eta(y)} -  \frac{\dot{\eta}(y)\ddot{\eta}(y) }{\{\eta(y)\}^2 } - 2 \dot{\ell}_{\eta}(y) \ddot{\ell}_{\eta}(y) \Big| \\
	&\le& \frac{1}{\eta(y)} \Big\{ \int \frac{(y-z)}{\sigma^4}\phi_\sigma(y-z)d \widebar{F}(z) + \int \frac{2|y-z|}{\sigma^4} \phi_\sigma(y-z)d \widebar{F}(z) \\
	&& + \int \frac{|y-z|^3}{\sigma^6} \phi_\sigma(y-z)d \widebar{F}(z)  \Big\} \\
	&&+ \big\{ |y| + (\log n )^{\frac{2}{\tau}}   \big\} (s_n \log p)^{\frac{1}{a_2}} \, 3 \big\{ y^2 + (\log n )^{\frac{4}{\tau}}   \big\} (s_n \log p)^{\frac{2}{a_2}} \\
	&& + 2 \big\{ |y| + (\log n )^{\frac{2}{\tau}}   \big\} (s_n \log p)^{\frac{1}{a_2}} \, 5 \big\{ y^2 + (\log n )^{\frac{4}{\tau}}   \big\} (s_n \log p)^{\frac{2}{a_2}} \\
	&\le& 43 \big\{ |y|^3 + (\log n)^{\frac{6}{\tau}} \big\} (s_n \log p)^{\frac{3}{a_2}}.
	\eea

	Assume that a small $\zeta>0$ is given.
	Let $A = \{ y\in \bbR : |y| \le C_1 \left(\log (1/\epsilon_n) \right)^{\frac{1}{\tau}} \}$ for some large constant $C_1>0$. 
	Note that 
	\bea
	&& \int_{A^c} \sup_{\eta\in \calH_n^*} \left( \dot{\ell}_\eta(y) - \dot{\ell}_{\eta_0}(y) \right)^2 dP_{\eta_0}(y) \\
	&\lesssim& \int_{A^c} \sup_{\eta\in\calH_n^*} \left(\dot{\ell}_\eta(y) \right)^2 dP_{\eta_0}(y) + \int_{A^c} \left(\dot{\ell}_{\eta_0}(y) \right)^2 dP_{\eta_0}(y) .
	\eea
	It is easy to show that
	\bea
	&& \int_{A^c} \sup_{\eta\in\calH_n^*} \left(\dot{\ell}_\eta(y) \right)^2 dP_{\eta_0}(y) \\
	&\lesssim& \int_{y > C_1 (\log \frac{1}{\epsilon_n})^{\frac{1}{\tau}}} \left(y^2 + [\log n]^{\frac{4}{\tau}} \right) e^{-b y^\tau} dy \cdot (s_n \log p)^{\frac{2}{a_2}} \\
	&\lesssim& \left(\epsilon_n \right)^{\frac{b}{2}C_1^\tau} \cdot (s_n \log p)^{\frac{2}{a_2}} (\log n)^{\frac{4}{\tau}} \,\, \lesssim\,\, \epsilon_n
	\eea
	for some constant large $C_1>0$ by the assumption $(s_n \log p)^{2} = o(n)$. Since 
	\bea
	\int_{A^c} \left(\dot{\ell}_{\eta_0}(y) \right)^2 dP_{\eta_0}(y) 
	&\lesssim& \int_{y > C_1 (\log \frac{1}{\epsilon_n})^{\frac{1}{\tau}}} (|y|^{\gamma_1} +1 ) e^{-b y^\tau } dy   \\
	&\lesssim& \epsilon_n
	\eea
	for some large constant $C_1>0$, we have
	\bea
	\int_{A^c} \sup_{\eta\in \calH_n^*} \left(\dot{\ell}_\eta(y) - \dot{\ell}_{\eta_0}(y) \right)^2 dP_{\eta_0} (y)
	&\lesssim& \epsilon_n.
	\eea
	Thus, it suffices to prove 
	\bea
	\int_{A} \sup_{\eta\in \calH_n^*} \left(\dot{\ell}_\eta(y) - \dot{\ell}_{\eta_0}(y) \right)^2 dP_{\eta_0}(y) 
	&\le& K_\zeta \left(\epsilon_n \right)^{\frac{4}{5}-\zeta} (s_n \log p)^{\frac{16}{5a_2}}
	\eea
	for some positive constants $\zeta$ and $K_\zeta$ not depending on $(n,p)$.
	
	Define for any $x$ and  $y \in\bbR$,
	\bea
	d_\eta(x,y) &:=& \frac{\ell_\eta(y+x) - \ell_\eta(y)}{x} - \frac{\ell_{\eta_0}(y+x) - \ell_{\eta_0}(y)}{x},
	\eea
	then we have that
	\bean
	&& \int_{A} \sup_{\eta\in \calH_n^*} \left(\dot{\ell}_\eta(y) - \dot{\ell}_{\eta_0}(y) \right)^2 dP_{\eta_0}(y) \nonumber\\
	&\lesssim& \int_{A} \sup_{\eta\in \calH_n^*} \left(\dot{\ell}_\eta(y) - \dot{\ell}_{\eta_0}(y) - d_\eta(x,y) \right)^2 dP_{\eta_0} (y)  \label{score_T1} \\
	&+& \frac{1}{x^2}\int_A \sup_{\eta\in\calH_n^*} \left( x \, d_\eta(x,y) \right)^2 dP_{\eta_0}(y). \label{score_T2}
	\eean
	One can obtain the upper bound for \eqref{score_T1} using 
	\bea
	| \dot{\ell}_\eta(y) - \dot{\ell}_{\eta_0}(y) - d_\eta(x,y)|
	&\le& \left| \dot{\ell}_\eta(y) - \frac{\ell_\eta(y+x) - \ell_\eta(y)}{x} \right| \\
	&& +\,\, \left| \dot{\ell}_{\eta_0}(y) - \frac{\ell_{\eta_0}(y+x) - \ell_{\eta_0}(y)}{x} \right| \\
	&\le& |x| \cdot \left\{ |\ddot{\ell}_\eta(y_1)| + |\ddot{\ell}_{\eta_0}(y_2)| \right\} \\
	&\lesssim& |x| \cdot \left\{ y^2 +  (\log n)^{\frac{4}{\tau}} \right\} \left( s_n \log p \right)^{\frac{2}{a_2}} \\
	&\lesssim& |x| (s_n \log p )^{\frac{2}{a_2}} (\log n)^{\frac{4}{\tau}} 
	\eea
	for any $\eta\in \calH_n^*$, $y\in A$, small $|x|$ and some $|y-y_1| \vee |y-y_2| \le |x|$ by the Taylor expansion. Thus, 
	\bean\label{ellelldeta}
	\int_{A} \sup_{\eta\in \calH_n^*} \left(\dot{\ell}_\eta(y) - \dot{\ell}_{\eta_0}(y) - d_\eta(x,y) \right)^2 dP_{\eta_0} (y) 
	&\lesssim& x^2 \cdot \left( s_n \log p \right)^{\frac{4}{a_2}} [\log n]^{\frac{8}{\tau}}. \quad \quad \,\,\,
	\eean

	Note that $|x\, d_\eta(x,y)| \le |\ell_\eta(y+x) - \ell_{\eta_0}(y+x) | + |\ell_\eta(y) - \ell_{\eta_0}(y)|$ and
	\bea
	&& \int_A \sup_{\eta\in \calH_n^*} \left( \ell_\eta(y+x) - \ell_{\eta_0}(y+x) \right)^2 dP_{\eta_0}(y) \\
	&=& \int_A \sup_{\eta\in \calH_n^*} \left( \ell_\eta(y+x) - \ell_{\eta_0}(y+x) \right)^2 \eta_0(y+x) \cdot \frac{\eta_0(y)}{\eta_0(y+x)} dy \\
	&\lesssim& \int_A \sup_{\eta\in \calH_n^*} \left( \ell_\eta(y+x) - \ell_{\eta_0}(y+x) \right)^2 \eta_0(y+x) \cdot e^{b'|y|^{\tau'}} dy 
	\eea
	provided that $|x|$ is small, by condition \hyperref[D5]{(D5)}.
	To calculate the upper bound for \eqref{score_T2}, we first find an upper bound for $f_\eta(y) := (\ell_\eta(y) - \ell_{\eta_0}(y))^2 \eta_0(y)$  on $y\in A$ and $\eta\in \calH_n^*$. 
	Let $\delta_n := \epsilon_n \log (1/\epsilon_n)$ and $B:= \left\{ y\in \bbR : |y| \le 2 C_1 (\log (1/\delta_n))^{\frac{1}{\tau}} \right\}$, so that $A \subset B$ for all sufficiently large $n$.
	By the triangle inequality and the definition of $\calH_n^*$, 
	\bean\label{dot_feta}
	\begin{split}
		| \dot{f}_\eta(y) | &=  \left| \, 2(\ell_\eta(y) -\ell_{\eta_0}(y) ) (\dot{\ell}_\eta(y) - \dot{\ell}_{\eta_0}(y) ) \eta_0(y) + (\ell_\eta(y) -\ell_{\eta_0}(y) )^2 \dot{\eta}_0(y)  \,\right| \\
		&\lesssim \sqrt{f_\eta(y)} \sqrt{\eta_0(y)} \left( |\dot{\ell}_\eta(y) - \dot{\ell}_{\eta_0}(y)| +   |\ell_\eta(y) - \ell_{\eta_0}(y)|\cdot |\dot{\ell}_{\eta_0}(y)|  \right) \\
		&\lesssim \sqrt{f_\eta(y)}   \left(s_n \log p\right)^{\frac{1}{a_2}} (\log n)^{\frac{4}{\tau}},
	\end{split}
	\eean
	and
	\bean\label{ddot_feta}
	\begin{split}
		&| \ddot{f}_\eta(y)| \\
		&\lesssim \eta_0(y) \bigg\{ \left(\dot{\ell}_\eta(y) - \dot{\ell}_{\eta_0}(y)  \right)^2 + | \ddot{\ell}_{\eta}(y) - \ddot{\ell}_{\eta_0}(y) | \cdot |\ell_\eta(y) - \ell_{\eta_0}(y) | \\
		&+ |\ell_\eta(y) - \ell_{\eta_0}(y) |\cdot |\dot{\ell}_\eta(y) - \dot{\ell}_{\eta_0}(y)| \cdot | \dot{\ell}_\eta(y) | + \left(\ell_\eta(y) -\ell_{\eta_0}(y) \right)^2 |\ddot{\ell}_{\eta_0}(y)|  \bigg\}  \quad\quad \\
		&\lesssim  \left(s_n \log p\right)^{\frac{3}{a_2}} (\log n)^{\frac{8}{\tau}}
	\end{split}
	\eean
	for any $\eta\in \calH_n^*$ and $y\in\bbR$.
	By the Taylor expansion,
	\bea
	&&| f_\eta(y+x) - f_\eta(y)| \\
	&\lesssim& |x|  \sqrt{f_\eta(y)} \left(s_n\log p \right)^{\frac{1}{a_2}} (\log n)^{\frac{4}{\tau}} +  x^2 \left(s_n\log p \right)^{\frac{3}{a_2}} (\log n)^{\frac{8}{\tau}} \\
	&\lesssim& \left(s_n\log p\right)^{\frac{1}{a_2}} [\log n]^{\frac{4}{\tau}} \left\{ |x|\sqrt{f_\eta(y)}  + x^2 \left(s_n \log p\right)^{\frac{2}{a_2}} (\log n)^{\frac{4}{\tau}} \right\} 
	\eea
	for any $y\in\bbR$ and small $|x|$.
	If we take $|x| \le C\left(s_n \log p\right)^{-\frac{3}{2 a_2}}(\log n)^{-\frac{4}{\tau}} \sqrt{f_\eta(y)}$ for some small constant $C>0$, it implies $| f_\eta(y+x) - f_\eta(y)|  \le f_\eta(y)/2$ for any $y\in \bbR$ and small $|x|$.
	Therefore, for any fixed $y_0 \in A$, we have $f_\eta(y_0 +x) > f_\eta(y_0)/2$ for any $|x| \le C \left(s_n \log p\right)^{-\frac{3}{2a_2}} (\log n)^{-\frac{4}{\tau}} \sqrt{f_\eta(y_0)}$ for some small constant $C>0$. 
	Then,
	\bean\label{Bset_lb}
	\int_B f_\eta(y) dy 
	&\ge& \int_{|y-y_0| \le C \left(s_n\log p\right)^{-\frac{3}{2a_2}} [\log n]^{-\frac{4}{\tau}} \sqrt{f_\eta(y_0)}} f_\eta(y) dy  \nonumber\\
	&\gtrsim& \left(s_n\log p\right)^{-\frac{3}{2a_2}} (\log n)^{-\frac{4}{\tau}} \left(f_\eta(y_0)\right)^{\frac{3}{2}} 
	\eean
	for any $y_0 \in A$ and $\eta\in\calH_n^*$.
	On the other hand,  
	$$1/\eta(y) \lesssim (\log n )^{\frac{1}{2}} \exp \{ 2(s_n \log p)^{\frac{1}{a_2}} (\log n)^{\frac{4}{\tau}} \}$$ 
	for any $y\in B$ and $\eta\in\calH_n^*$, which implies
	\bea
	\int_B \Big\{ \frac{\eta_0(y)}{\eta(y)} \Big\}^{\delta} \eta_0(y) dy 
	&\lesssim& \int_B \eta_0(y)^{1+\delta} (\log n)^{\frac{\delta}{2}} \exp \{ 2\delta (s_n \log p)^{\frac{1}{a_2}} (\log n)^{\frac{4}{\tau}} \} dy \\
	&\lesssim& 1
	\eea
	by taking $\delta = (s_n \log p)^{-\frac{1}{a_2}} (\log n )^{-\frac{4}{\tau}}$.
	Thus, by Lemma \ref{lemma:intpq_ub}, we have
	\bean\label{Bfeta_ub}
	\int_B f_\eta(y) dy 
	&\lesssim& \delta_n^2 \left(s_n\log p\right)^{\frac{2}{a_2}} [\log n]^{\frac{12}{\tau}}
	\eean
	for any $\eta\in\calH_n^*$. 
	By combining \eqref{Bset_lb} and \eqref{Bfeta_ub}, it implies that
	\bean\label{first_up}
	f_\eta(y_0) &\lesssim& \delta_n^{\frac{4}{3}} \left(s_n\log p\right)^{\frac{7}{3a_2}} [\log n]^{\frac{32}{3\tau}}
	\eean
	for any $y_0\in A$ and $\eta\in \calH_n^*$.

	Next, we claim that if $f_\eta(y) \lesssim \delta_n^{d_1} \left(s_n\log p \right)^{d_2} [\log n]^{d_3}$ for some $d_1,d_2$ and $d_3>0$, then we have $f_\eta(y) \lesssim \delta_n^{1+\frac{3}{8}d_1 -\zeta } \left(s_n \log p\right)^{\frac{3}{8}d_2 +\frac{3}{2a_2} } [\log n]^{\frac{3}{8}d_3 + \frac{7}{\tau}}$ for any $y \in A$ and $\eta \in \calH_n^*$.	
	Suppose that $f_\eta(y) \lesssim \delta_n^{d_1} \left(s_n\log p \right)^{d_2} [\log n]^{d_3}$ on $y\in A$ and $\eta\in \calH_n^*$ for some positive constants $d_1,d_2$ and $d_3$. 
	Due to \eqref{first_up}, there exist constants $d_1=4/3, d_2=7/(3a_2)$ and $d_3=32/(3\tau)$ satisfying $f_\eta(y) \lesssim \delta_n^{d_1} \left(s_n\log p \right)^{d_2} [\log n]^{d_3}$.
	Note that for any small constant $\zeta>0$,
	\begin{align}\label{dotell_eta_eta0}
	\begin{split}
	&| \dot{\ell}_\eta(y) - \dot{\ell}_{\eta_0}(y) | \sqrt{\eta_0(y)} \\
	&\lesssim |x| \left( |\ddot{\ell}_\eta(y_1) | + | \ddot{\ell}_{\eta_0}(y_2)|  \right)\sqrt{\eta_0(y)} \\
	&+ \frac{|\ell_\eta(y+x) - \ell_{\eta_0}(y+x)| + |\ell_\eta(y) - \ell_{\eta_0}(y)|}{|x|}  \sqrt{\eta_0(y)} \\
	\end{split} \\
	&\lesssim |x| \left(s_n\log p\right)^{\frac{2}{a_2}} [\log n]^{\frac{4}{\tau}} 
	+ \frac{e^{\frac{b'}{2}|y|^{\tau'}}}{|x|}\cdot  \delta_n^{\frac{d_1}{2}} \left(s_n\log p\right)^{\frac{d_2}{2}} [\log n]^{\frac{d_3}{2}} \nonumber  \\
	&\lesssim |x| \left(s_n\log p\right)^{\frac{2}{a_2}} [\log n]^{\frac{4}{\tau}} 
	+ \frac{1}{|x|} \delta_n^{\frac{d_1}{2} - 4\zeta} \left(s_n\log p\right)^{\frac{d_2}{2}} [\log n]^{\frac{d_3}{2}} \nonumber
	\end{align}
	for some $|y-y_1| \vee |y-y_2| \le |x|$, thus
	\bean
	| \dot{\ell}_\eta(y) - \dot{\ell}_{\eta_0}(y) | \sqrt{\eta_0(y)} 
	&\lesssim& \delta_n^{\frac{d_1}{4} - 2\zeta} \left(s_n\log p\right)^{\frac{d_2}{4} +\frac{1}{a_2} } [\log n]^{\frac{d_3}{4} + \frac{2}{\tau}} \label{dotell_eta_eta0_rate}
	\eean
	on $y \in A$ and $\eta\in \calH_n^*$, by taking $|x|= \delta_n^{\frac{d_1}{4} - 2\zeta} \left(s_n\log p\right)^{\frac{d_2}{4}-\frac{1}{a_2} } [\log n]^{\frac{d_3}{4} -\frac{2}{\tau} }$. 
	Then, by \eqref{dot_feta}, 
	\bea
	| \dot{f}_\eta(y) | 
	&\lesssim& \delta_n^{\frac{3}{4}d_1 -2\zeta } \left(s_n \log p\right)^{\frac{3}{4}d_2 +\frac{1}{a_2} } [\log n]^{\frac{3}{4}d_3 + \frac{2}{\tau}}
	\eea
	for any $y\in A$ and $\eta\in \calH_n^*$, which implies that 
	\bea
	f_\eta(y+x) &\ge& \frac{1}{2}f_\eta(y) 
	\eea
	for any $y\in A, \eta\in \calH_n^*$, $|x|\le C_3\delta_n^{-\frac{3}{4}d_1 +2\zeta } \left(s_n \log p \right)^{-\frac{3}{4}d_2 -\frac{1}{a_2} } [\log n]^{-\frac{3}{4}d_3 - \frac{2}{\tau}} f_\eta(y)$ and for some small constant $C_3>0$, by the first-order Taylor expansion.
	Thus, similar to \eqref{Bset_lb}, 
	\bea
	\int_B f_\eta(y) dy 
	&\gtrsim& \left(f_\eta(y_0) \right)^2 \delta_n^{-\frac{3}{4}d_1 +2\zeta } \left(s_n \log p\right)^{-\frac{3}{4}d_2 -\frac{1}{a_2} } [\log n]^{-\frac{3}{4}d_3 - \frac{2}{\tau}},
	\eea
	for any $y_0\in A, \eta\in \calH_n^*$ and small $\zeta>0$.
	Again by \eqref{Bfeta_ub}, 
	\bean\label{new_feta_ub}
	f_\eta(y) 
	&\lesssim& \delta_n^{1+\frac{3}{8}d_1 -\zeta } \left(s_n \log p\right)^{\frac{3}{8}d_2 +\frac{3}{2a_2} } [\log n]^{\frac{3}{8}d_3 + \frac{7}{\tau}},
	\eean
	for any $y\in A, \eta\in \calH_n^*$ and small $\zeta>0$.
	
	Note that  the upper bound \eqref{new_feta_ub} is obtained from the assumption $\sup_{\eta\in \calH_n^*}$ $f_\eta(y) \lesssim \delta_n^{d_1} \left(s_n \log p \right)^{d_2} [\log n]^{d_3}$. 
	Thus, by applying the claim repeatedly, one can check that $\sup_{\eta\in \calH_n^*} f_\eta(y) \lesssim \delta_n^{\frac{8}{5} -2\zeta} \left(s_n\log p \right)^{\frac{12}{5a_2} } [\log n]^{\frac{56}{5\tau}} $ for any $y\in A$ and a given small constant $\zeta>0$.
	
	Therefore, we finally obtain the following upper bound
	\bea
	&& \int_{A} \sup_{\eta\in \calH_n^*} \left(\dot{\ell}_\eta(y) - \dot{\ell}_{\eta_0}(y)  \right)^2 dP_{\eta_0} (y) \\
	&\lesssim& x^2 \cdot \left( s_n \log p \right)^{\frac{4}{a_2}} [\log n]^{\frac{8}{\tau}}  \\
	&&+ \frac{1}{x^2}  \int_A \sup_{\eta\in \calH_n^*} \left( \ell_\eta(y+x) - \ell_{\eta_0}(y+x) \right)^2 \eta_0(y+x) \cdot e^{b'|y|^{\tau'}} dy  \\
	&\lesssim& x^2 \cdot \left( s_n \log p \right)^{\frac{4}{a_2}} [\log n]^{\frac{8}{\tau}} 
	+ \frac{ \delta_n^{\frac{8}{5} -2\zeta} \left(s_n\log p \right)^{\frac{12}{5a_2} } [\log n]^{\frac{56}{5\tau}}  }{x^2} 
	\eea
	by \eqref{ellelldeta}. By taking $|x| = \delta_n^{\frac{2}{5} -\frac{\zeta}{2}} \left(s_n\log p \right)^{-\frac{2}{5a_2} } [\log n]^{\frac{4}{5\tau}}$, 
	\bea
	\int_{A} \sup_{\eta\in \calH_n^*} \left(\dot{\ell}_\eta(y) - \dot{\ell}_{\eta_0}(y)  \right)^2 dP_{\eta_0} (y) 
	&\le& K_\zeta \delta_n^{\frac{4}{5} -\zeta} \left(s_n\log p \right)^{\frac{16}{5a_2} } [\log n]^{\frac{48}{5\tau}}\\ 
	&\le& K_\zeta \epsilon_n^{\frac{4}{5} -2\zeta} \left(s_n\log p \right)^{\frac{16}{5a_2} }
	\eea
	for some constant $K_\zeta>0$ not depending on $(n,p)$.  \hfill $\blacksquare$
\end{proof}

\bigskip\noindent
This lemma gives slightly faster convergence rate, under stronger condition,  compared with Lemma \ref{lemma:score}, and is used to prove the misspecified LAN (Lemma \ref{thm:misLAN}).
Although Lemma \ref{lemma:score2} seems similar to Lemma \ref{lemma:score}, we stated them separately to avoid assuming redundant conditions for Lemma \ref{lemma:score}.

\begin{lemma}\label{lemma:score2}
	Let $\epsilon_n = K_{\rm eta}\sqrt{s_n \log p /n}$. For any constant $\zeta>0$, there exists a constant $K_\zeta>0$ not depending on $(n,p)$ such that
	\bea
	\int \sup_{\eta\in \calH_n^*} \left(\dot{\ell}_\eta(y) - \dot{\ell}_{\eta_0}(y) \right)^2 dP_{\eta_0} (y)
	&\le& K_\zeta \left(\epsilon_n \right)^{\frac{4}{5}-\zeta} 
	\eea
	for any $\eta_0$ satisfying \hyperref[D1]{(D1)}-\hyperref[D5]{(D5)} and all sufficiently large $n$, provided that $\left(s_n \log p \right)^{1+\frac{15}{a_2}} = o(n^{1-\zeta})$, where $\calH_n^*$ defined at \eqref{H_n_star}.
\end{lemma}

\begin{proof}	

	Assume that a small $\zeta>0$ is given. Let $\varphi_n := \epsilon_n^{\frac{4}{5} -\zeta} \left(s_n\log p\right)^{\frac{6}{5a_2} }[\log n]^{\frac{4}{\tau}}$, $A' := \{y\in A: \eta_0(y) \gtrsim \varphi_n^2  \}$ and $B' := \{y\in B: \eta_0(y) \gtrsim \varphi_n^2  \}$, where $A$ and $B$ are defined in Lemma \ref{lemma:score}.
	Note that
	\bea
	\int_{{(A')}^c} \sup_{\eta\in\calH_n^*} \left(\dot{\ell}_\eta(y) \right)^2 dP_{\eta_0}(y)
	&\lesssim& 
	\int_{A^c} \sup_{\eta\in\calH_n^*} \left(\dot{\ell}_\eta(y) \right)^2 dP_{\eta_0}(y) \\
	&+& \int_{A\cap \{y:\, \eta_0(y)\lesssim \varphi_n^2 \}  } \sup_{\eta\in\calH_n^*} \left(\dot{\ell}_\eta(y) \right)^2 dP_{\eta_0}(y) \\
	&\lesssim& \epsilon_n + \varphi_n^2 \left(s_n\log p \right)^{\frac{2}{a_2} } [\log n ]^{\frac{4}{\tau}} \cdot \int_A  (y^2 +1) dy \\
	&\lesssim& \epsilon_n + \varphi_n^2 \left(s_n\log p \right)^{\frac{2}{a_2} } [\log n ]^{\frac{7}{\tau}} \\
	&\lesssim& \epsilon_n^{\frac{4}{5} -\zeta },
	\eea
	provided that $\left(s_n \log p \right)^{1+\frac{11}{a_2}} = o(n)$. Similarly, it is easy to check that $$\int_{{(A')}^c} \sup_{\eta\in\calH_n^*} \left(\dot{\ell}_{\eta_0}(y) \right)^2 dP_{\eta_0}(y) \lesssim \epsilon_n^{4/5 -\zeta}.$$
	Hence, it suffices to show that 
	\bea
	\int_{A'} \sup_{\eta\in \calH_n^*} \left(\dot{\ell}_\eta(y) - \dot{\ell}_{\eta_0}(y) \right)^2 dP_{\eta_0}(y) 
	&\le& K_\zeta \left(\epsilon_n \right)^{\frac{4}{5}-\zeta}
	\eea
	for some positive constants $\zeta$ and $K_\zeta$.
	
	Note that similar to \eqref{dotell_eta_eta0},
	\bean
	&& |\ddot{\ell}_\eta(y) - \ddot{\ell}_{\eta_0}(y)| \sqrt{\eta_0(y)} \nonumber \\
	&\lesssim& |x| \left\{|\dddot{\ell}_\eta(y_1)|+ |\dddot{\ell}_{\eta_0}(y_2)|  \right\} \sqrt{\eta_0(y)} 
	+ \frac{e^{\frac{b'}{2}|y|^{\tau'} }}{|x|} |\dot{\ell}_\eta(y) - \dot{\ell}_{\eta_0}(y)| \sqrt{\eta_0(y)} \quad\,\,\,\,\label{ddotell_eta_eta0} \\
	&\lesssim& |x| \left(s_n \log p\right)^{\frac{3}{a_2} } [\log n ]^{\frac{6}{\tau}} + \frac{1}{|x|} \delta_n^{\frac{2}{5}-\zeta }\left(s_n\log p\right)^{\frac{8}{5a_2} } [\log n ]^{\frac{4}{\tau}} \nonumber
	\eean
	for some $|y-y_1| \vee |y-y_2| \le |x|$ on $y\in A'$ and $\eta\in \calH_n^*$ by \eqref{dotell_eta_eta0_rate}. Then, by taking appropriate $|x|$, we have
	\bean
	|\ddot{\ell}_\eta(y) - \ddot{\ell}_{\eta_0}(y)| \sqrt{\eta_0(y)}
	&\lesssim& \delta_n^{\frac{1}{5}-\zeta }\left(s_n\log p\right)^{\frac{23}{10a_2} } [\log n ]^{\frac{5}{\tau}} \label{ddotell_diff_B} \\
	&\lesssim& \left(s_n\log p\right)^{\frac{4}{5a_2} } \nonumber
	\eean
	on $y\in A'$ and $\eta\in \calH_n^*$, because we assume that $\left(s_n \log p \right)^{1+\frac{15}{a_2}} = o(n^{1-\zeta})$. 
	Suppose that $\sup_{\eta\in \calH_n^*}|\ddot{\ell}_\eta(y) - \ddot{\ell_0}_\eta(y)|\sqrt{\eta_0(y)} \lesssim \left(s_n\log p \right)^{K}$ and $\sup_{\eta\in \calH_n^*} f_\eta(y) \lesssim \delta_n^{d_1} \left(s_n\log p \right)^{d_2} [\log n]^{d_3}$ on $y\in B'$ for some positive constants $K, d_1,d_2$ and $d_3$. 
	Note that from the proof of Lemma \ref{lemma:score} and the definition of $B'$,
	\bea
	\frac{\eta_0(y)}{\eta(y)} 
	&\lesssim& \exp \left( \frac{\varphi_n}{\sqrt{\eta_0(y)}}  \right) \,\,\lesssim\,\, 1
	\eea
	for any $y\in B'$ and $\eta\in \calH_n^*$, then, similar to \eqref{Bfeta_ub}, it is easy to show that
	\bean\label{Bfeta_ub2}
	\int_{B'} f_\eta(y) dy 
	&\lesssim& \delta_n^2,
	\eean
	by Lemma \ref{lemma:intpq_ub}.
	Applying \eqref{dotell_eta_eta0}, 
	\bean\label{dotell_eta_eta0_rate2}
	|\dot{\ell}_\eta(y) - \dot{\ell}_{\eta_0}(y)| \sqrt{\eta_0(y)}
	&\lesssim& \delta_n^{\frac{d_1}{4} - \zeta} \left(s_n\log p \right)^{\frac{d_2}{4} + \frac{K}{2}} [\log n]^{\frac{d_3}{4}}
	\eean
	for any $y \in A'$ and $\eta\in \calH_n^*$.
	Then by \eqref{Bfeta_ub2} and the similar arguments to the proof of Lemma \ref{lemma:score}, we have
	\bea
	f_\eta(y) 
	&\lesssim& \delta_n^{1+ \frac{3}{8}d_1 - \zeta} \left(s_n\log p \right)^{\frac{3}{8}d_2  + \frac{K}{4}} [\log n]^{\frac{3}{8}d_3} 
	\eea
	for any $y \in A'$ and $\eta\in \calH_n^*$.
	By a recursion, one can check that $d_1,d_2$ and $d_3$ converge to $8/5-\zeta, 2K/5$ and $0$, respectively.
	Thus, by \eqref{dotell_eta_eta0_rate2}, we have
	\bean\label{Krate}
	|\dot{\ell}_\eta(y) - \dot{\ell}_{\eta_0}(y)| \sqrt{\eta_0(y)}
	&\lesssim& \delta_n^{\frac{2}{5} -\zeta} \left(s_n\log p \right)^{\frac{3}{5}K}
	\eean
	for any $y \in A'$ and $\eta\in \calH_n^*$, and it implies that 
	\bea
	|\ddot{\ell}_\eta(y) - \ddot{\ell_0}_\eta(y)|\sqrt{\eta_0(y)} 
	&\lesssim& \delta_n^{\frac{1}{5}-\zeta }\left(s_n\log p \right)^{\frac{3}{2a_2} + \frac{3}{10}K } [\log n ]^{\frac{5}{\tau}} \\
	&\lesssim& \left(s_n\log p \right)^{\frac{3}{10}K }
	\eea
	for any $y\in A'$ and $\eta\in \calH_n^*$ by \eqref{ddotell_eta_eta0}.
	Thus, we obtain $\sup_{\eta\in \calH_n^*}|\ddot{\ell}_\eta(y) - \ddot{\ell_0}_\eta(y)|\sqrt{\eta_0(y)} \lesssim \left(s_n\log p \right)^{\frac{3}{10}K }$ from the assumption $\sup_{\eta\in \calH_n^*}|\ddot{\ell}_\eta(y) - \ddot{\ell_0}_\eta(y)|$ $\times\sqrt{\eta_0(y)}  \lesssim \left(s_n \log p \right)^{K}$ on $y\in A'$. Suppose that a small constant $\zeta'>0$ is given, then we have $\sup_{\eta\in \calH_n^*} |\ddot{\ell}_\eta(y)| \lesssim \left(s_n \log p \right)^{\zeta'}$ on $y\in B'$ by repeatedly applying the above arguments. 
	Finally, by \eqref{Krate}, 
	\bea
	\left( \dot{\ell}_\eta(y) - \dot{\ell}_{\eta_0}(y)\right)^2 \eta_0(y)
	&\lesssim& \delta_n^{\frac{4}{5} -\zeta}
	\eea
	for some given constant $\zeta>0$, any $y\in A'$ and $\eta\in \calH_n^*$.
	Therefore,
	\bea
	\int_{A'} \sup_{\eta\in \calH_n^*} \left(\dot{\ell}_\eta(y) - \dot{\ell}_{\eta_0}(y) \right)^2 dP_{\eta_0}(y) 
	&\le&  K_\zeta \left(\epsilon_n \right)^{\frac{4}{5}-\zeta}
	\eea
	for some positive constants $\zeta$ and $K_\zeta$ not depending on $(n,p)$.   \hfill $\blacksquare$
\end{proof}

\begin{lemma}\label{lemma:score_L2_not_uniform}
	If $(s_n \log p)^{1+ \frac{11}{2a_2}} = o(n^{1-\zeta})$ for some constant $\zeta>0$, we have
	\bea
	\sup_{\eta\in \calH_n^*}  \int \left(\dot{\ell}_\eta(y) - \dot{\ell}_{\eta_0}(y) \right)^2 dP_{\eta_0}(y) &=& o(1)
	\eea
	for any $\eta_0$ satisfying  \hyperref[D1]{(D1)}-\hyperref[D5]{(D5)}, where $\calH_n^*$ defined at \eqref{H_n_star}.
\end{lemma}
\begin{proof}
	Note that
	\bea
	\int \left(\dot{\ell}_\eta(y) - \dot{\ell}_{\eta_0}(y) \right)^2 dP_{\eta_0}(y)
	&=& - \int (\ell_{\eta}(y) - \ell_{\eta_0}(y)) (\dot{\ell}_{\eta}(y) - \dot{\ell}_{\eta_0}(y)) \dot{\eta}_0(y) dy \\
	&-& \int (\ell_{\eta}(y) - \ell_{\eta_0}(y)) (\ddot{\ell}_{\eta}(y) - \ddot{\ell}_{\eta_0}(y)) {\eta}_0(y) dy
	\eea
	follows from the integration by parts. 
	By Lemma \ref{lemma:intpq_ub}, \eqref{dotell_eta_eta0_rate} and \eqref{ddotell_diff_B}, one can show that the absolute value of the above equality is bounded above by $\epsilon_n^{\frac{6}{5}-\zeta} (s_n \log p)^{\frac{33}{10a_2}}$ for some constant $\zeta>0$, up to some constant not depending on $\eta$, which implies the desired result.  \hfill $\blacksquare$
\end{proof}

\bigskip\noindent
The following lemma is used to prove Lemma \ref{lemma:Gnfupper}.

\begin{lemma}\label{lemma:Gnfupper}
	Let $s_n$ be a sequence of positive integers. Define
	\bea
	\Theta_{n,1} &:=& \left\{ \theta\in\bbR^p : s_\theta \le s_n, \|\theta-\theta_0\|_1\le 1 \right\}
	\eea
	and $f_{\theta,\bar{\theta},\eta} := (\theta-\theta_0)^T \ddot{\ell}_{\bar{\theta},\eta}(\theta-\theta_0)$. 
	If we assume $\left(s_n \log p \right)^{1+\frac{15}{a_2}} = o(n^{1-\zeta})$ for some constant $\zeta>0$, then for any small constant $\zeta'>0$,
	\bean\label{Gnupper_strong}
	&&\Eaa \left( \sup_{\theta,\bar{\theta}\in \Theta_{n,1} }  \sup_{\eta\in\calH_n^* } \frac{1}{\sqrt{n}} \bigg| \bbG_n  f_{\theta,\bar{\theta},\eta}  \bigg|  \right) \nonumber\\ 
	&\lesssim& \left( \frac{ s_n(\log p)^3 + \left(s_n \log p \right)^{\frac{3}{a_2}} (\log p )^4  }{n} \right)^{\frac{1}{2}} \left(s_n\log p\right)^{\zeta'}   \quad\quad\quad
	\eean
	for any $\eta_0$ satisfying  \hyperref[D1]{(D1)}-\hyperref[D5]{(D5)} and all sufficiently large $n$, where $\calH_n^*$ defined at \eqref{H_n_star}.
\end{lemma}

\begin{proof}
	Without loss of generality, we assume that $\theta_0=0$.
	For a given $\zeta'>0$, define 
	\bean\label{F_nset}
	\widetilde{\calF}_n &:=& \left\{ \widetilde{f}_{\theta,\bar{\theta},\eta} = \left(s_n\log p \right)^{- \zeta' } (\log p)^{-1} \cdot {f}_{\theta,\bar{\theta},\eta}  : \theta, \bar{\theta}\in \Theta_{n,1}, \eta \in \calH_n^* \right\}. \quad\quad
	\eean
	Then for any $\widetilde{f}_{\theta,\bar{\theta},\eta} \in \widetilde{\calF}_n$, 
	\bea
	&&|\widetilde{f}_{\theta,\bar{\theta},\eta}(x,y)| \\
	&\le& \sup_{\theta, \bar{\theta}\in \Theta_{n,1}} \sup_{\eta\in \calH_n^*} (x^T \theta)^2 | \ddot{\ell}_\eta(y-x^T\bar{\theta})| \left(s_n\log p \right)^{- \zeta' } (\log p)^{-1}
	\,\,=:\,\, \widetilde{F}_n(x,y).
	\eea
	$\widetilde{F}_n$ is an envelop function of $\widetilde{\calF}_n$ such that $\Eaa \widetilde{F}_n^2(x_i, Y_i) \lesssim 1$ for any $i=1,\ldots, n$ because
	\bea
	&&\Eaa \widetilde{F}_n^2(x,Y) \\
	&=& \int \sup_{\theta,\bar{\theta}\in \Theta_{n,1}} \sup_{\eta\in \calH_n^*} (x^T \theta)^4 | \ddot{\ell}_\eta(y-x^T\bar{\theta})|^2  \eta_0(y) dy\cdot \left(s_n\log p \right)^{- 2\zeta' } (\log p)^{-2} \\
	&\lesssim& \int_{A'} \sup_{\bar{\theta}\in \Theta_{n,1}} \sup_{\eta\in \calH_n^*} | \ddot{\ell}_\eta(y-x^T\bar{\theta})|^2 \eta_0(y) dy\cdot \left(s_n\log p \right)^{-2\zeta' } \\
	&+& \int_{(A')^c} \sup_{\bar{\theta}\in \Theta_{n,1}} \sup_{\eta\in \calH_n^*} | \ddot{\ell}_\eta(y-x^T\bar{\theta})|^2  \eta_0(y) dy\cdot \left(s_n\log p \right)^{-2\zeta' } \\
	&\lesssim& \left(s_n\log p \right)^{-2\zeta' } + \int_{A^c} \sup_{\bar{\theta}\in \Theta_{n,1}} \sup_{\eta\in \calH_n^*} | \ddot{\ell}_\eta(y-x^T\bar{\theta})|^2  \eta_0(y) dy\cdot \left(s_n \log p \right)^{-2\zeta' } \\
	&+& \int_{A \cap \{y:\eta_0(y) \lesssim \varphi_n^2 \} } \sup_{\bar{\theta}\in \Theta_{n,1}} \sup_{\eta\in \calH_n^*} | \ddot{\ell}_\eta(y-x^T\bar{\theta})|^2  \eta_0(y) dy\cdot \left(s_n\log p \right)^{-2\zeta' }  \\
	&\lesssim& \left(s_n\log p \right)^{-2\zeta' } + \left(s_n \log p \right)^{\frac{4}{a_2} }  \varphi_n^2 \left(s_n\log p \right)^{-2\zeta' } \,\, \lesssim\,\, \left(s_n\log p \right)^{-2\zeta' }
	\eea
	provided that $\left(s_n \log p \right)^{1+\frac{15}{a_2} } =o(n)$, where $A, A'$ and $\varphi_n$ are defined in the proof of Lemma \ref{lemma:score2}.
	Thus, $\| \widetilde{F}_n \|_n^2 = n^{-1} \sum_{i=1}^n \Eaa \widetilde{F}_n^2(x_i,Y_i) \lesssim \left(s_n\log p \right)^{-2\zeta' }$. We will use Corollary A.1 in \cite{chae2019bayesian}, which implies
	\bean\label{Gnfupper}
	&&\Eaa \left( \sup_{\theta,\bar{\theta}\in \Theta_{n,1} }  \sup_{\eta\in\calH_n^* } \frac{1}{\sqrt{n}} \bigg| \bbG_n  f_{\theta,\bar{\theta},\eta}  \bigg|  \right) \nonumber\\
	&\lesssim& \int_0^{\|\widetilde{F}_n\|_n} \sqrt{\log N^n_{[\, ]}(\epsilon, \widetilde{\calF}_n) } d\epsilon \cdot \frac{\left(s_n \log p \right)^{ \zeta' } }{\sqrt{n}} \, \log p . \quad\quad\quad\,\,
	\eean
	
	Now, we calculate $N^n_{[\, ]}(\epsilon, \widetilde{\calF}_n)$ defined at \eqref{Gnfupper}.
	For $\theta^j, \bar{\theta}^j \in \Theta_{n,1}$ and $\eta_j \in \calH_n^*, j=1,2$, write
	\bea
	\widetilde{f}_{\theta^1, \bar{\theta}^1, \eta_1} - \widetilde{f}_{\theta^2, \bar{\theta}^2, \eta_2} &\equiv& \widetilde{f}_1 + \widetilde{f}_2 + \widetilde{f}_3,
	\eea
	where $\widetilde{f}_1 := \widetilde{f}_{\theta^1, \bar{\theta}^1, \eta_1} - \widetilde{f}_{\theta^2, \bar{\theta}^1, \eta_1}, \widetilde{f}_{2} := \widetilde{f}_{\theta^2, \bar{\theta}^1, \eta_1} - \widetilde{f}_{\theta^2, \bar{\theta}^2, \eta_1}$ and $\widetilde{f}_{3} := \widetilde{f}_{\theta^2, \bar{\theta}^2, \eta_1} - \widetilde{f}_{\theta^2, \bar{\theta}^2, \eta_2}$. 
	It is easy to show  $|\widetilde{f}_1(x,y)| \lesssim \|\theta^1- \theta^2\|_1 \cdot \left( y^2  +1\right) \left( s_n \log p \right)^{\frac{2}{a_2}} [\log n]^{\frac{4}{\tau}}$ and  $|\widetilde{f}_2(x,y)| \lesssim \|\bar{\theta}^1 - \bar{\theta}^2 \|_1 \cdot \left( |y|^3 +1 \right) \left( s_n \log p \right)^{\frac{3}{a_2}} [\log n]^{\frac{6}{\tau}} \sqrt{\log p}$.  
	Then, we have
	\bea
	&& \Eaa \left( \sup_{\theta^1, \theta^2} \sup_{\eta_1,\eta_2} |\widetilde{f}_{\theta^1, \bar{\theta}^1, \eta_1}(x,Y) - \widetilde{f}_{\theta^2, \bar{\theta}^2, \eta_2}(x,Y) |^2  \right) \\
	&\lesssim& \sup_{\theta^1,\theta^2} \|\theta^1-\theta^2\|_1^2  \left( s_n \log p \right)^{\frac{6}{a_2}} [\log n]^{\frac{12}{\tau}} \, \log p   + \Eaa \left( \sup_{\theta^1, \theta^2} \sup_{\eta_1,\eta_2} | \widetilde{f}_3(x,Y)|^2  \right).
	\eea 
	To deal with $\widetilde{f}_3$, define
	\bea
	\widetilde{\calG}_{K_n} &:=& \left\{ \ddot{\ell}_\eta \cdot I_{[-K_n,K_n]} : \eta \in \calH_n^*   \right\}
	\eea
	and $\widetilde{H}_{K_n} := \sup_{\eta\in\calH_n^*} \max_{k=0,1} \sup_{|y|\le K_n} | \ddot{\ell}_\eta^{\,(k)}(y) | $ for some $K_n >0$. 
	Then, Theorem 2.7.1 of \cite{vaart1996weak}, which implies for every $\epsilon>0$,
	\bea
	\log N(\epsilon) &:=& \log N(\epsilon, \widetilde{\calG}_{K_n} , \|\cdot\|_\infty) \\
	&\lesssim& K_n \cdot \widetilde{H}_{K_n} \cdot\frac{1}{\epsilon}\\
	&\lesssim& K_n \cdot K_n^3 \left(s_n\log p \right)^{\frac{3}{a_2}}(\log n)^{\frac{6}{\tau}} \,\frac{1}{\epsilon} .
	\eea
	By the definition of the covering number, there is a partition $\{ \calH^l : 1\le l \le N(\epsilon) \}$ of $\calH_n^*$ such that
	\bea
	&& \int_{|y|\le K_n - M\sqrt{\log p} } \sup_{\theta\in \Theta_{n,1} } \sup_{\eta_1,\eta_2 \in \calH^l} | \ddot{\ell}_{\eta_1}(y-x^T \theta) - \ddot{\ell}_{\eta_2}(y-x^T \theta) |^2 dP_{\eta_0}(y)   \\
	&\lesssim& \int_{|y|\le K_n - M\sqrt{\log p} } \epsilon^2 dP_{\eta_0}(y) \,\, \le \,\, \epsilon^2.
	\eea
	Let $K_n= C(\log (1/\epsilon))^{1/\tau} + C(\log n)^{1/\tau} + M\sqrt{\log p}$ for some constant $C>0$, then 
	\bea
	&&\int_{|y|> K_n - M\sqrt{\log p} } \sup_{\theta\in \Theta_{n,1} }\sup_{\eta\in\calH_n^*} | \ddot{\ell}_\eta(y - x^T\theta)|^2 dP_{\eta_0}(y)   \\
	&\lesssim& \int_{|y|>K_n - M\sqrt{\log p} } y^4 e^{-b|y|^\tau} dy\cdot \left(s_n \log p \right)^{\frac{4}{a_2}} [\log n]^{\frac{8}{\tau}}  \\
	&\lesssim& e^{-\frac{b}{4}K_n^\tau}\cdot \left(s_n \log p\right)^{\frac{4}{a_2}} [\log n]^{\frac{8}{\tau}}  \,\,\le\,\, \epsilon^2  .
	\eea
	Thus, we have
	\bea
	\int \sup_{\theta^2,\bar{\theta}^2\in \Theta_{n,1} }\sup_{\eta_1,\eta_2\in \calH^l} | \widetilde{f}_3(x,y) |^2 dP_{\eta_0}(y) 	&\lesssim& \epsilon^2,
	\eea
	for some constant $C>0$ and any $1\le l \le N(\epsilon)$.

	By the above arguments, 
	\bea
	\log N^n_{[\, ]} (\epsilon, \widetilde{\calF}_n ) 
	&\lesssim& \log N(\epsilon) + \log N \left(\epsilon \left(s_n \log p \right)^{-\frac{3}{a_2}} [\log n]^{-\frac{6}{\tau}}[\log p]^{-\frac{1}{2}} , \Theta_n, \|\cdot\|_1 \right) \\
	&\lesssim& K_n^4 \left(s_n\log p\right)^{\frac{3}{a_2}}(\log n)^{\frac{6}{\tau}} \cdot\frac{1}{\epsilon}  + s_n \log p + s_n \log \frac{1}{\epsilon} \\
	&\lesssim& \epsilon^{-\frac{3}{2}} \cdot \left(s_n\log p \right)^{\frac{3}{a_2}}(\log n)^{\frac{6}{\tau}} (\log p)^2 + s_n \log p + s_n \log \frac{1}{\epsilon} .
	\eea
	Hence, by \eqref{Gnfupper}, we get the inequality \eqref{Gnupper_strong}.    \hfill $\blacksquare$
\end{proof}

\bigskip\noindent
The following lemma is used to prove Lemma \ref{thm:misLAN}.

\begin{lemma}[Misspecified LAN: version 1]\label{lemma:LAN}
	Let $s_n$ be a positive integer sequence and $\epsilon_n$ be a sequence such that $\epsilon_n \to 0$. 
	Define $\Theta_{n,\epsilon_n} := \{ \theta\in \Theta : s_\theta\le s_n, \|\theta-\theta_0\|_1\le \epsilon_n \}$ and
	$\tilde{r}_n(\theta,\eta):= L_n(\theta,\eta) - L_n(\theta_0,\eta_0) - \sqrt{n} (\theta-\theta_0)^T \bbG_n \dot{\ell}_{\theta_0,\eta} + n (\theta-\theta_0)^T V_{n,\eta} (\theta-\theta_0)/2$. 
	If we assume that $\left(s_n \log p \right)^{1+\frac{15}{a_2}} = o(n^{1-\zeta})$ for some constant $\zeta>0$, then 
	\bean\label{LAN1_strong}
	&&\Eaa \left( \sup_{\theta\in \Theta_{n,\epsilon_n}} \sup_{\eta\in \calH_n^*} |\tilde{r}_n(\theta,\eta)|  \right) \nonumber\\
	&\lesssim& n \epsilon_n^2 \cdot \rho_n + \epsilon_n \sqrt{\log p} \cdot \sup_{\theta\in\Theta_{n,\epsilon_n}} \| X(\theta-\theta_0)\|_2^2, \quad\quad\quad
	\eean 
	for any $\eta_0$ satisfying  \hyperref[D1]{(D1)}-\hyperref[D5]{(D5)} and all sufficiently large $n$, where $\calH_n^*$ defined at \eqref{H_n_star} and
	\bea
	\rho_n &:=& \left( \frac{ s_n(\log p)^3 + \left(s_n \log p \right)^{\frac{3}{a_2}} (\log p )^4  }{n} \right)^{\frac{1}{2}} \left(s_n\log p\right)^{\zeta'} 
	\eea
	for a given constant $\zeta'>0$.
\end{lemma}

\begin{proof}
	By the Taylor expansion, where $\theta(t):= \theta_0 + t(\theta-\theta_0)$,
	\bea
	L_n(\theta,\eta) &=& L_n(\theta(1),\eta)\\
	&=& L_n(\theta_0,\eta) + \frac{\partial}{\partial t} L_n(\theta(t),\eta) \big|_{t=0} + \int_0^1 \frac{\partial^2}{\partial t^2} L_n(\theta(t),\eta)(1-t)dt .
	\eea
	Since $\bbE_{\theta_0,\eta_0} \dot\ell_{\theta_0, \eta} = 0$ for every $\eta$ by \hyperref[D4]{(D4)}, we have that
	\bea
	\frac{\partial}{\partial t} L_n(\theta(t),\eta) \big|_{t=0} &=& \sqrt{n}(\theta-\theta_0)^T \bbG_n \dot{\ell}_{\theta_0,\eta}
	\eea
	and
	\bea
	\frac{\partial^2}{\partial t^2} L_n(\theta(t),\eta) &=& n(\theta-\theta_0)^T \bbP_n \ddot{\ell}_{\theta(t),\eta}(\theta-\theta_0).
	\eea
	Define 
	\bea
	A_{n1}(\theta,\eta) &:=& n \int_0^1 (1-t) \frac{1}{\sqrt{n}} \bbG_n(\theta-\theta_0)^T \ddot{\ell}_{\theta(t),\eta} (\theta-\theta_0) dt ,\\
	A_{n2}(\theta,\eta) &:=& \int_0^1 (1-t) \\
	&& \sum_{i=1}^n \left[(\theta-\theta_0)^T \bbE_{\theta_0,\eta_0} \left\{ \ddot{\ell}_{\theta(t),\eta}(x_i,Y_i) - \ddot{\ell}_{\theta_0,\eta}(x_i,Y_i) \right\} (\theta-\theta_0)  \right] dt ,\\
	A_{n3}(\theta,\eta) &:=& \frac{1}{2} \sum_{i=1}^n (\theta-\theta_0)^T \bbE_{\theta_0,\eta_0} \ddot{\ell}_{\theta_0,\eta}(x_i,Y_i)(\theta-\theta_0),
	\eea
	then, it is easy to show that
	\bea
	\int_0^1 \frac{\partial^2}{\partial t^2} L_n(\theta(t),\eta) (1-t)dt &=&  A_{n1}(\theta,\eta) +  A_{n2}(\theta,\eta) + A_{n3}(\theta,\eta) .
	\eea
	Since 
	\bea
	\frac{1}{\sqrt{n}} \bbG_n (\theta-\theta_0)^T \ddot{\ell}_{\theta(t),\eta}(\theta-\theta_0) 
	&=& \frac{\|\theta-\theta_0\|_1^2}{\sqrt{n}} \bbG_n \frac{(\theta-\theta_0)^T}{\|\theta-\theta_0\|_1} \ddot{\ell}_{\theta(t),\eta}\frac{(\theta-\theta_0)}{\|\theta-\theta_0\|_1},
	\eea
	we have
	\bea
	\bbE_{\theta_0,\eta_0} \left( \sup_{\theta\in \Theta_{n,\epsilon_n}} \sup_{\eta \in \calH_n^*} \Big|A_{n1}(\theta,\eta) \Big| \right) &\lesssim& n \epsilon_n^2 \cdot \rho_n
	\eea
	by \eqref{Gnupper_strong} in Lemma \ref{lemma:Gnfupper}, provided that $\left(s_n \log p \right)^{1+\frac{15}{a_2}} = o(n^{1-\zeta})$ for some $\zeta>0$. 
	Since $A_{n3}(\theta,\eta) = -n/2 \cdot (\theta-\theta_0)^T V_{n,\eta}(\theta-\theta_0)$, if we only need to show that
	\bea
	\sup_{\theta\in \Theta_{n,\epsilon_n}} \sup_{\eta\in \calH_n^*}  \Big| A_{n2}(\theta,\eta) \Big| 
	&\lesssim& \epsilon_n \sqrt{\log p} \cdot \sup_{\theta\in \Theta_{n,\epsilon_n}} \| X(\theta-\theta_0)\|_2^2,
	\eea 
	where $\theta(t):= \theta_0 + t(\theta-\theta_0)$ for $0\le t \le 1$. 
	To show the above inequality, it suffices to prove that
	\bean
	&&(\theta-\theta_0)^T \left\{ \Eaa \ddot{\ell}_{\theta(t),\eta}(x_i,Y_i) - \Eaa \ddot{\ell}_{\theta_0,\eta}(x_i, Y_i)  \right\} (\theta-\theta_0) \label{Qn2_sum} \\
	&\lesssim& | x_i^T (\theta-\theta_0)^T |^2 \, \sqrt{\log p} \|\theta- \theta_0\|_1 	 \nonumber
	\eean
	for any $i=1,\ldots,n$. Note that \eqref{Qn2_sum} is bounded above by
	\bea
	&& |x_i^T(\theta-\theta_0)|^2  \left| \Eaa \left( \ddot{\ell}_{\eta}(Y_i - x_i^T \theta(t)) - \ddot{\ell}_{\eta}(Y_i - x_i^T \theta_0)  \right)  \right| \\
	&\lesssim& |x_i^T(\theta-\theta_0)|^2 \, \sqrt{\log p} \|\theta-\theta_0\|_1 \cdot \left| \Eaa \dddot{\ell}_\eta(Y_i- x_i^T \theta(t_1))  \right| ,
	\eea 
	for some constant $0\le t_1 \le t$.
	Also note that
	\bea
	&& \left| \Eaa \left( \dddot{\ell}_\eta(Y- x^T \theta(t_1)) -  \dddot{\ell}_{\eta_0}(Y- x^T \theta(t_1)) \right) \right| \\
	&=& \left| \int \left( \dddot{\ell}_\eta(y-x^T\theta(t_1)) - \dddot{\ell}_{\eta_0}(y-x^T\theta(t_1))  \right) \eta_0(y-x^T\theta_0) dy \right| \\
	&=& \left| \int \left( \dot{\ell}_\eta(y-x^T\theta(t_1)) - \dot{\ell}_{\eta_0}(y-x^T\theta(t_1) \right) \ddot{\eta}_0(y-x^T\theta_0) dy  \right|  \\
	&\le& \left[ \int ( \dot{\ell}_\eta(y)- \dot{\ell}_{\eta_0}(y))^2 \eta_0(y) dy \right]^{\frac{1}{2}} \\
	&&\times \,\, \left[ \int \left(\frac{\ddot{\eta_0}(y-x^T\theta_0)}{\eta_0(y-x^T\theta_0)} \right)^2 \frac{\eta_0(y-x^T\theta_0)}{\eta_0(y-x^T\theta(t_1))} \eta_0(y-x^T\theta_0) dy  \right]^{\frac{1}{2}} .
	\eea
	The above equality follows from the integration by parts, and the last inequality follows from the H\"{o}lder's inequality. The last term is of order $O(1)$ by Lemma \ref{lemma:score_L2_not_uniform}.
	Since $ \big| \Eaa \dddot{\ell}_{\eta_0}(Y-x^T \theta(t_1))  \big| \lesssim 1$, it completes the proof for \eqref{LAN1_strong}. \hfill $\blacksquare$
\end{proof}

\bigskip\noindent
Finally, the following lemma is used to prove Lemma \ref{lemma:normal}.

\begin{lemma}\label{lemma:Gnelldot}
	Suppose that $(s_n \log p )^{1+ \frac{8}{a_2}} = o(n^{1-\zeta})$ holds for some constant $\zeta>0$, then
	\bea
	\Eaa \left( \sup_{\eta \in \calH_n^*} \| \bbG_n \dot{\ell}_{\theta_0,\eta}\|_\infty \right) &\lesssim& {\log p}
	\eea
	for any $\eta_0$ satisfying  \hyperref[D1]{(D1)}-\hyperref[D5]{(D5)}, where $\calH_n^*$ defined at \eqref{H_n_star}.
\end{lemma}

\begin{proof}
	Without loss of generality, we assume that $\theta_0=0$. 
	Define 
	$$\calF_n := \left\{ e_j^T \dot{\ell}_{\theta_0,\eta} \, (\log p)^{-\frac{1}{2}}  : 1\le j \le p,\,\, \eta\in \calH_n^* \right\},$$ where $e_j$ is the $j$th unit vector in $\bbR^p$. Then,
	\bea
	\sup_{\eta \in \calH_n^*}\|\bbG_n \dot{\ell}_{\theta_0,\eta} \|_\infty &=& \sup_{f \in \calF_n} |\bbG_n f| \sqrt{\log p}.
	\eea
	We first show that $F_n(x,y) :=\sup_{\eta\in \calH_n^*} |\dot{\ell}_\eta(y)- \dot{\ell}_{\eta_0}(y) | + |\dot{\ell}_{\eta_0}(y)|$ is an envelop function of $\calF_n$ and $\Eaa F_n^2(x_i, Y_i) \lesssim 1$ for any $i=1,\ldots,n$.	
	Note that for any $f\in \calF_n$ and $x=(x_1,\ldots,x_p)^T$,    
	\bea
	|f(x,y)| &=& \left| e_j^T \dot{\ell}_{\theta_0,\eta}(x,y)  \right| (\log p)^{-\frac{1}{2}} \\
	&=& \left| x_j \cdot \dot{\ell}_\eta(y) \right|  (\log p)^{-\frac{1}{2}} \\
	&\lesssim&  \sup_{\eta\in \calH_n^*} |\dot{\ell}_\eta(y)- \dot{\ell}_{\eta_0}(y) | + |\dot{\ell}_{\eta_0}(y)|  .
	\eea
	By Lemma \ref{lemma:score}, we have $\Eaa F_n^2(x_i, Y_i) \lesssim 1$ if $(s_n\log p)^{1+\frac{8}{a_2}} =O(n^{1-\zeta})$ for some $\zeta>0$. 
	Then, we have
	\bea
	\bbE_{\theta_0,\eta_0} \left(\sup_{\eta \in\calH_n^*} \|\bbG_n \dot{\ell}_{\theta_0,\eta}\|_\infty  \right) 
	&\lesssim& \int_0^{\|F_n\|_n} \sqrt{\log N_{[\,]}^n(\epsilon, \calF_n )} \, d\epsilon \, \sqrt{\log p}\\
	&\lesssim& \int_0^{\|F_n\|_n} \sqrt{\epsilon^{-1} + \log p } \, d\epsilon \, \sqrt{\log p} \,\,\lesssim\,\, {\log p},
	\eea  
	where the second inequality follows from Corollary 2.7.4 of \cite{vaart1996weak}.  \hfill $\blacksquare$
\end{proof}


%
%

\bibliographystyle{spbasic}      
\bibliography{bayes-HDLR-beyond-SG}   


\end{document}